%% file: paper.tex
\documentclass{amsart}
\usepackage{amsthm}
\usepackage{amsmath}
\usepackage{amstext}
\usepackage{amsfonts}
\usepackage{amssymb}
\usepackage{latexsym}
\usepackage{amscd}
\usepackage{xypic}
\theoremstyle{plain}
\newtheorem{theorem}{Theorem}[section]
\newtheorem{lemma}[theorem]{Lemma}
\newtheorem{proposition}[theorem]{Proposition}
\newtheorem{corollary}[theorem]{Corollary}

\theoremstyle{definition}
\newtheorem{definition}[theorem]{Definition}

\newtheorem{problem}[theorem]{Problem}
\newtheorem{problems}[theorem]{Problems}
\newtheorem{example}[theorem]{Example}
\newtheorem{remark}[theorem]{Remark}
\newtheorem{remarks}[theorem]{Remarks}

\numberwithin{equation}{theorem}
\newtheorem{claim}[theorem]{Claim}
\begin{document}

\title[Automorphisms of canonically polarized surfaces in char.  $P>0$.]{Automorphisms of smooth canonically polarized surfaces in positive characteristic.}
\author{Nikolaos Tziolas}
\address{Department of Mathematics, University of Cyprus, P.O. Box 20537, Nicosia, 1678, Cyprus}
\email{tziolas@ucy.ac.cy}
\curraddr{Department of Mathematics, Princeton University, Fine Hall, Washington Road, Princeton, NJ 08544} 
\email{ntziolas@princeton.edu}
\thanks{The author is supported by a Marie Curie International Outgoing  Fellowship ,  grant no. PIOF-GA-2013-624345.}


\dedicatory{This paper is dedicated to my wife Afroditi and my son Marko.}


\begin{abstract}
This paper investigates the geometry of a smooth canonically polarized surface $X$ defined over an algebraically closed field of characteristic $p>0$ in the case when the automorphism scheme of $X$ is not smooth. This is a situation that appears only in positive characteristic and it is closely related to the structure of the moduli stack of canonically polarized surfaces. Restrictions on certain numerical invariants of $X$ are obtained in order for 
$\mathrm{Aut}(X)$ to be smooth or not and information is provided about the structure of the component of $\mathrm{Aut}(X)$ containing the identity. In particular, it is shown that a smooth canonically polarized surface $X$ with 
$1\leq K_X^2\leq 2$ and non smooth automorphism scheme tends to be uniruled and simply connected. Moreover, $X$ is the purely inseparable quotient of a ruled or rational surface by a rational vector field.
\end{abstract}

\maketitle

\input{intro}

\input{sec0}

\input{sec1}

\input{sec2}

\input{sec3}

\input{sec4}

\input{sec5}

\input{sec6}
\input{sec7}
\input{sec8}

\input{sec9}
\input{bib}
\end{document}

%% file: intro.tex

\section{Introduction}
One of the most important problems in algebraic geometry is the classification up to isomorphism of algebraic varieties defined over an algebraically closed field $k$. In order to deal with this problem, the class of varieties is divided into smaller classes where a reasonable answer is expected to exist. Quite generally, once such a class $\mathcal{C}$ is chosen, the corresponding moduli functor $\mathcal{M}_{\mathcal{C}}\colon Sch(k) \rightarrow (Sets)$ is defined by setting for any $k$-scheme $S$, $\mathcal{M}_{\mathcal{C}}(S)$ to be the set of isomorphism classes of flat morphisms $X\rightarrow S$ whose fibers are in $\mathcal{C}$. Sometimes, depending on the case, some extra structure on the family morphisms are required to create a ``reasonable'' moduli functor. The best that one could hope for is that the functor is representable, which means hat 
there is a universal family $\mathcal{X} \rightarrow \mathrm{M}_{\mathcal{C}}$ such that any other family is obtained from it by base change. In this case $\mathrm{M}_{\mathcal{C}}$ is called a fine moduli space for $\mathcal{M}_{\mathcal{C}}$. Unfortunately fine moduli spaces rarely exist. The reason for this failure is the presence of nontrivial automorphisms of the objects that one wants to parametrize.  

One way of dealing with this difficulty is instead of looking for a universal family, to settle for less and search for a variety $\mathrm{M}_{\mathcal{C}}$ whose $k$-points are in one to one correspondence with the varieties of the moduli problem and which satisfies some uniqueness property. Such a variety is called a coarse 
moduli space. However, the biggest disadvantage of this approach is that usually the coarse moduli space does not support a family and therefore it gives very little 
information about families of varieties in the moduli problem. In order to study families as well, the universality condition is relaxed and one looks for a so called modular family. Loosely speaking a modular family is a family $X \rightarrow S$ such that up to \'etale base change, any other family is obtained from it by base change and that for any closed point $s \in S$, the completion $\hat{\mathcal{O}}_{S,s}$ prorepresents the local deformation functor $Def(X_s)$. In some sense a modular family is a connection between the local moduli functor (which behaves well) and the global one. In modern language one says that the moduli stack associated to the moduli problem is Deligne-Mumford~\cite{DM69}.
  
In dimension 1 curves are separated by their genus $g$. The moduli functor $\mathcal{M}_g$ of smooth curves of genus $g$ defined over an algebraically closed field has a coarse moduli space $\mathrm{M}_g$. For $g=0$ it is a reduced point, for $g=1$ it is the $j$-line and for $g\geq 2$ it is irreducible of dimension $3g-3$ and it admits a compactification $\bar{\mathrm{M}}_g$ whose boundary points correspond to stable curves, i.e, to reduced curves $C$ with at worst nodes as singularities and $\omega_C$ ample~\cite{DM69}. In all cases the corresponding moduli stack is Deligne-Mumford. 
   
In dimension 2, surfaces are divided according to their kodaira dimension $\kappa$ which takes the values $-\infty, 0, 1$ and $ 2$. Surfaces with kodaira dimension 2 are the corresponding cases to the case of curves of genus $\geq 2$ and are called surfaces of general type. Early on in the theory of moduli of surfaces of general type in characteristic zero, it was realized that the correct objects to parametrize are not the surfaces of general type themselves but their canonical models.  For compactification reasons, the moduli functor is extended to include the so called stable surfaces. These are reduced two dimensional schemes $X$ with semi-log-canonical (slc) singularities such that there is an integer $N$ such that $\omega^{[N]}_X$ is ample. They are the higher dimensional analog of stable curves. Then for any fixed integer valued function $H(m)$  the moduli functor $\mathcal{M}^s_{H} \colon Sch(k) \rightarrow (Sets)$ of stable surfaces with fixed Hilbert polynomial is defined~\cite{KSB88}~\cite{Ko10}. 
It is known 
that $\mathcal{M}^s_{H}$ has a separated coarse moduli space $\mathrm{M}^s_{H}$ which is of finite type over the base field $k$. Moreover, the corresponding moduli stack is a separated, proper Deligne-Mumford stack of finite type~\cite{KSB88}~\cite{Ko97}~\cite{Ko10}.  

This paper is supposed to be a small contribution into the study of the moduli of canonically polarized surfaces in positive characteristic. It also inspires to bring attention to the positive characteristic case and motivate people to work on it. In particular, it is a first attempt to study the following general problem.

\begin{problem}
Study the moduli stack of stable varieties with fixed Hilbert polynomial defined over an algebraically closed field of characteristic $p>0$. In particular, is it proper, Deligne-Mumford or of finite type? In the case when one of these properties fails, why does it fail and how can the moduli problem be modified in order for the corresponding stack to satisfy them?
\end{problem}

As mentioned earlier, there has been tremendous progress in the characteristic zero case lately. In particular, the case of surfaces has been completely settled. However, the positive characteristic case is to the best of my knowledge a wide open area. The main reason is probably the many complications and pathologies that appear in positive characteristic. For example, Kodaira vanishing fails and the minimal model program and semistable reduction, two ingredients essential in the compactification of the moduli in the characteristic zero case, are not known, at the time of this writing, to work in positive characteristic. Moreover, the moduli stack of smooth canonically polarized surfaces is not Deligne-Mumford. The reason for this failure in positive characteristic is the existense of smooth canonically polarized surfaces with non smooth automorphism scheme~\cite{La83},~\cite{SB96},~\cite{Li08}.  This does not happen in characteristic zero simply because every group scheme in characteristic zero is smooth. 
Therefore the non smoothness of the automorphism scheme is an essential obstruction for the moduli stack of canonically polarized surfaces to be Deligne-Mumford. 

However, in the case of surfaces there is reason to believe that it would be possible to obtain a good moduli theory even in positive characteristic. Resolution of singularities exists for surfaces. Moreover, canonical models of surfaces are classically known to exist and the semistable minimal model program for semistable threefolds holds~\cite{Kaw94} . In addition, the definition of stable surfaces is characteristic free and therefore the definition of the functor of stable surfaces applies in positive characteristic too. Finally, canonically polarized surfaces with fixed Hilbert polynomial are bounded~\cite{Ko84} and Koll\'ar's quotient results show that there exist a separated coarse moduli space of finite type for the moduli functor of canonically polarized surfaces over $\mathrm{Spec}\mathbb{Z}$ with a fixed Hilbert polynomial~\cite{Ko84}. 

One way that one could put some order in positive characteristic is initially to distinguish the largest class of canonically polarized surfaces with smooth automorphism scheme which is closed under deformations and specialization. Then the methods of characteristic zero show that the corresponding stack is Deligne-Mumford. Moreover, if semistable-reduction works as well as the minimal model program for semistable log canonical threefolds, the stack is also proper. The class of all surfaces with smooth automorphism scheme does not work because it is not closed under specialization, as shown by example~\ref{ex2}. However, according to Theorem~\ref{th1}, nonsmoothness of the automorphism scheme of a stable surface $X$ with hilbert function $H(n)$ happens for small values of the characteristic of the base field compared with the coefficients of $H(n)$. Hence the surfaces that are the best candidates to have smooth automorphism scheme are those with small invariants compared to the characteristic. 

Based on these observations, the first step in the study of the moduli stack of canonically polarized surfaces in positive characteristic is to investigate the structure of the automorphism scheme of a canonically polarized surface and study why and when it is not smooth. In particular to find numerical relations (preferably deformation invariant) between the characteristic of the base field and certain numerical invariants of the surface that hold if the automorphism scheme is not smooth. According to the previous discussion, this investigation will provide important information on how the moduli functor can be modified in order to obtain proper Deligne-Mumford stacks and how to deal with surfaces with non smooth automorphism schemes. In any case, it is essential to study surfaces with non smooth automorphism scheme. 

The main result of this paper is the following.
\newpage

\begin{theorem}\label{main-theorem}
Let $X$ be a smooth canonically polarized surface defined over an algebraically closed field of characteristic $p>0$ such that $1\leq K_X^2 \leq 2$. Suppose that $\mathrm{Aut}(X)$ is not smooth or equivalently that $X$ 
possesses a non trivial global vector field. Then,
\begin{enumerate}
\item If $K_X^2=2$ and $p\not= 3,5$, then $X$ is uniruled. If in addition $\chi(\mathcal{O}_X)\geq 2$, then $X$ is unirational and $\pi_1^{et}(X)=\{1\}$.
\item If $K_X^2=1$ and $p\not= 7$, then $X$ is unirational and $\pi_1^{et}(X)=\{1\}$,  except possibly if $p\in\{3,5\}$ and $X$ is a simply connected supersingular Godeaux surface. Moreover, if in addition $p \not= 3,5,7$, then  $p_g(X) \leq 1$.
\end{enumerate}
In all cases, $X$ is the quotient of a ruled  surface  by a rational vector field.
\end{theorem}

\begin{corollary}\label{main-corollary}
Let $X$ be a smooth canonically polarized surface defined over an algebraically closed field of characteristic $p>0$. Suppose that either $K_X^2=2$, $p \not\in\{3,5\}$  and $X$ is not uniruled, or $K_X^2=1$, $p\not= 7$  and in addition one of the following happens
\begin{enumerate}
\item $p_g(X)=2$ and $p\not= 3,5$.
\item $\chi(\mathcal{O}_X) =3$ and $p\not= 3,5$.
\item $\pi_1^{et}(X) \not= \{1\}$.
\item $X$ is not unirational.
\end{enumerate}
Then $\mathrm{Aut}(X)$ is smooth.
\end{corollary}

The conditions of Corollary~\ref{main-corollary} on the euler characteristic and the \'etale fundamental group are deformation invariant and therefore they are good conditions for the moduli problem Hence we get the following.

\begin{corollary}
Let  $\mathcal{M}_{1,ns}$ and $\mathcal{M}_{1,3}$  be the moduli stacks of canonically polarized surfaces $X$ with $K_X^2=1$ and  $\pi_1^{et}(X)\not= \{1\}$, or $K_X^2=1$ and $\chi(\mathcal{O}_X)=3$, $p\not=3,5$, respectively. 
Then $\mathcal{M}_{1,3}$ and $\mathcal{M}_{1,ns}$ are Deligne-Mumford stacks.
\end{corollary}

In particular, Theorem~\ref{main-theorem} applies to the case of Godeaux surfaces. Considering their significance I find it appropriate to write a statement for this case.
\begin{corollary}\label{godeaux}
Let $X$ be a canonically polarized Godeaux surface defined over an algebraically closed field $k$ of characteristic $p>0$, $p\not= 7$.  Suppose that  $\mathrm{Aut}(X)$ is not smooth. Then $X$ is unirational and simply connected,  except possibly if $p\in\{3,5\}$ and $X$ is a simply connected supersingular Godeaux surface. Moreover,
\[
\mathrm{Pic}^{\tau}(X)=
\begin{cases}
\mathrm{Spec}k & \text{if} \;\;\; p > 7\\
\mathrm{Spec}k \;\;\;\; \text{or}\;\;\;\; \mathbb{Z}/5\mathbb{Z} \;\;\;\; \text{or} \;\;\;\; \alpha_5 & \text{if} \;\;\; p=5 \\
\left(\oplus_{i=1}^m \mathbb{Z}/p^{n_i}\mathbb{Z}\right) \oplus N & \text{if} \;\;\; p =2,3.
\end{cases}
\]
where $\mathrm{Pic}^{\tau}(X)$ be the torsion subgroup scheme of $\mathrm{Pic}(X)$, $N$ is a finite commutative group scheme which is either trivial or is obtained by successive extensions by $\alpha_p$.
\end{corollary}

This corollary follows directly from Theorem~\ref{main-theorem},  the discussion about the possible structure of it in section~\ref{sec-0},  basic properties about the structure of finite commutative groups schemes and the classification of Godeaux surfaces in characteristic 5~\cite{La83},\cite{Li09}. 

 From the characteristics zero and five cases that have been extensively studied~\cite{Re78},~\cite{La83},\cite{Li09}, one might expect that if $p =2$, then $\mathrm{Pic}^{\tau}(X)$ is either $\mathbb{Z}/2\mathbb{Z}$, $\mathbb{Z}/4\mathbb{Z}$ or $\alpha_2$, and if $p=3$ it is either $\mathbb{Z}/3\mathbb{Z}$ or $\alpha_3$. However, there is no classification in the characteristics 2 and 3 cases yet and since many pathologies appear in these characteristics, one must be careful. 
 
 A final comment that I would like to make about Godeaux surfaces is that I believe it will be interesting to know if there are examples of canonically polarized surfaces with nonreduced automorphism scheme but reduced Picard scheme. 

Next I would like to discuss the significance and how restrictive or effective the conditions of Theorem~\ref{main-theorem} are. 

According to Theorem~\ref{th1}, nonsmoothness of the automorphism scheme happens for small values of $K^2$ compared to the characteristic of the base field. However Theorem~\ref{th1} does not give any effective lower bound for the characteristic, relative to $K^2$, after which the automorphism scheme will become smooth. From this point of view, the automorphism scheme of surfaces with $K^2=1$ or $2$ should not be very complicated.  If it is not smooth then these surfaces should be rather pathological. The case when it is expected to be the most complicated and where most pathologies should appear must be the case when $p=2$. All these are exhibited by the results of Theorem~\ref{main-theorem}. In particular any surface with $K^2\leq 2$ and non reduced automorphism scheme is uniruled, something that is not possible in characteristic zero. However there are examples of non uniruled canonically polarized surfaces with high $K^2$ in characteristic 2~\cite{SB96}.

The exclusion in some cases of the characteristics $p=3,5,7$ in Theorem~\ref{main-theorem} is not because of any fundamental reason. As it can be seen from the proof of the theorem, it is due to certain difficulties with numerical calculations and the existence of fibrations with singular general fiber of arithmetic genus 1 or 2 in characteristics 3 and 5. However, these problems are more evident in characteristic 2, a case which is fully studied in 
sections~\ref{sec-4},~\ref{sec-5}. I believe that the method that was used in the characteristic 2 case can in principle be used in order to resolve the problematic characteristic cases. However at the moment I am having some technical difficulties and the papers is already too long.

One cannot expect that it will be possible to get results similar to those of Theorem~\ref{main-theorem} for any surface with nonreduced automorphism scheme. This is because there are examples in characteristic 2 of surfaces of general type with arbitrary large $K^2$ and non reduced automorphism scheme~\cite{La83},~\cite{Li09}, many of them not uniruled and not simply connected~\cite{SB96}. However, due to the lack of examples, I do not know if surfaces with $K^2\leq 2$ is the maximal class of surfaces that Theorem~\ref{main-theorem} holds. 

Suppose that $X$ is a canonically polarized surface $X$ with $K_X^2\leq 2$. If $K_X^2=2$, then by~\cite{Ek87}, it follows that $0 \leq \chi(\mathcal{O}_X)\leq 4$. Moreover, if $\chi(\mathcal{O}_X) \geq 3$, then by Lemma~\ref{b1}, $\pi_1^{et}(X)=\{1\}$ and hence the condition that $X$ is simply connected in Theorem~\ref{main-theorem}.2 has value only for $\chi(\mathcal{O}_X) =2$.

Suppose that $K_X^2=1$. Then it is well known~\cite{Li09} that  $\mathrm{p}_g(X)\leq 2$, $1 \leq \chi(\mathcal{O}_X) \leq 3$ and $|\pi_1^{et}(X)| \leq 6$. In characteristic zero there are examples for all cases and similar examples are expected to exist in any characteristic. Hence many surfaces with $K_X^2=1$ are excluded in the theorem and hence by the Corollary~\ref{main-corollary} have smooth automorphism scheme. For 
example, Godeaux surfaces that are quotients of a smooth quintic in $\mathbb{P}^3_k$ by a free action of $\mathbb{Z}/5\mathbb{Z}$, since they are not simply connected. 

At this point I would like to mention that it is hard to find examples of smooth canonically polarized surfaces with low $K^2$ and non smooth automorphism scheme. In fact, I believe that the biggest disadvatage of this paper is the lack of examples of surfaces with norreduced automorphism scheme and low $K^2$ which will show how effective the results of Theorem~\ref{main-theorem} are. N. I. Shepherd-Barron~\cite{SB96} has constructed an example of a smooth canonically polarized surface $X$ in characteristic 2  with non smooth automorphism scheme and $K_X^2=8$. This is the example with the lowest $K^2$ that I know. Simply connected Godeaux surfaces exist in all characteristics~\cite{LN12},~\cite{La81} but it is not known if their automorphism scheme is smooth or not. Singular examples are much easier to find. Two such examples are presented in section~\ref{examples}

Based on the previous discussion, I believe it would be an interesting problem to search for examples of canonically polarized surfaces with non smooth automorphism scheme and low $K^2$. In particular, Corollary~\ref{godeaux} motivates the following problem.
\begin{problem}
Are there any simply connected Godeaux surfaces with non smooth automorphism scheme? Equivalently, are there any with non trivial global vector fields? If there are then what is the torsion part of their picard scheme? Can it be reduced?
\end{problem}

The main steps in the proof of Theorem~\ref{main-theorem} are the following. 

In Section~\ref{sec-2} it is shown that $\mathrm{Aut}(X)$ is not smooth if and only if $X$ admits a nonzero global vector field $D$ such that either $D^p=0$ or $D^p=D$. This is equivalent to the property that $X$ admits a nontrivial $\alpha_p$ or $\mu_p$ action or that $\mathrm{Aut}(X)$ has a subgroup scheme isomorphic to either $\alpha_p$ or $\mu_p$. Corollary~\ref{lifts-to-zero} shows that this is impossible if $X$ lifts to characteristic zero.

In Section~\ref{sec-3} quotients of a smooth surface $X$ by an $\alpha_p$ or $\mu_p$ action are studied. In particular the singularities of the quotient $Y$ are described. Proposition~\ref{prop4} shows that there is an essential difference between $\alpha_p$ quotients and $\mu_p$ quotients. In particular, if $p=2$ and $\mu_2$ acts on $X$ then $Y$ has only canonical singularities of type $A_1$. However, in the $\alpha_2$ case, $Y$ may have even non rational singularities. 

In Section~\ref{preparation} the general strategy for the proof of Theorem~\ref{main-theorem} is explained. Quite generally it is the following. Suppose that $\mathrm{Aut}(X)$ is not smooth. Then $X$ has a nontrivial global vector field $D$ such that either $D^p=D$ or $D^p=0$. This vector field induces a nontrivial $\mu_o$ or $\alpha_p$ action on $X$. Let $\pi \colon X \rightarrow Y$ be the quotient and $g \colon Y^{\prime} \rightarrow Y$ its minimal resolution. The results are obtained by considering cases with respect to the Kodaira dimension $\kappa(Y^{\prime})$ of $Y^{\prime}$ and comparing invariants and the geometry of $X$ and $Y^{\prime}$. The basic idea of this method was first used by Rudakov and Shafarevich~\cite{R-S76} in order to show that a smooth K3 surface has no global vector fields.

In Section~\ref{p>2}, Theorem~\ref{main-theorem} is proved in the case $p\geq 3$. The case when $p=2$ has certain complications which makes it necessary to treat it separately. These complications are explained during the proofs of Theorems~\ref{sec-all-p-prop-1},~\ref{sec-all-p-prop-2} and~\ref{sec-all-p-prop-3}.

In Sections~\ref{sec-4},~\ref{sec-5} the characteristic 2 case is treated. In particular:

In Section~\ref{sec-4} it is investigated which smooth canonically polarized surfaces $X$ defined over an algebraically closed field of characteristic 2  admit vector fields of multiplicative type, or equivalently $\mathrm{Aut}(X)$ contain a subgroup scheme isomorphic to $\mu_2$. The results are presented in Theorem~\ref{mult-type}. 

In Section~\ref{sec-5} it is investigated which smooth canonically polarized surfaces $X$ admit vector fields of additive type, or equivalently $\mathrm{Aut}(X)$ contain a subgroup scheme isomorphic to $\alpha_2$. The results are presented in Theorem~\ref{additive-type}. 

Theorem~\ref{main-theorem} is the combination of the results of Theorems~\ref{sec-all-p-prop-1},~\ref{sec-all-p-prop-2},~\ref{sec-all-p-prop-3},~\ref{mult-type} and~\ref{additive-type}. 

At this point I would like to say that if $\mathrm{Aut}(X)$ is not smooth then it would be interesting to know its group scheme structure, in particular that of its connected component containing the identity. This is a hard problem. However, finding its subgroups is easier and gives a lot of information about it. If $\mathrm{Aut}(X)$ is not smooth, then it contains either $\mu_p$ or $\alpha_p$, or both. In characteristic 2, Theorems~\ref{mult-type} and~\ref{additive-type} describe surfaces $X$ such that $\mathrm{Aut}(X)$ contains $\mu_2$ or $\alpha_2$, respectively.
 
In Section~\ref{examples} an overview of known examples is given. Moreover two examples in characteristic 2 are given of singular surfaces $X$ and $Y$ with non smooth automorphism scheme and  $K_X^2=1$, $K_X^2=5$. $X$ has singularities of index 2 and $Y$ has canonical singularities of type $A_n$. The significance of these examples are twofold. First singular surfaces should be studied because they are important in the compactification of the moduli problem. The first example shows that if there are no restrictions on the singularities then $K^2$ can be as low as possible. The second has canonical singularities and is therefore a stable surface, a surface that is in the moduli problem, with low $K^2$. Moreover, $Y$ is smoothable to a smooth canonically polarized surface with smooth automorphism scheme. This shows that the property ``smooth automorphism scheme'' is not deformation invariant and does not produce a proper moduli stack.

%% file: sec0.tex

\section{Preliminaries.}\label{sec-0}
Let $X$ be a scheme of finite type over a field $k$ of characteristic $p>0$.

$X$ is called a smooth canonically polarised surface if and only if $X$ is a smooth surface and $\omega_X$ is ample.

$Der_k(X)$ denotes the space of global $k$-derivations of $X$ (or equivalently of global vector fields). It is canonically identified with $\mathrm{Hom}_X(\Omega_X,\mathcal{O}_X)$.

A nonzero global vector field  $D$ on $X$ is called of additive or multiplicative type if and only if $D^p=0$ or $D^p=D$, respectively. A prime divisor $Z$ of $X$ is called an integral divisor of $D$ if and only if locally there is a derivation $D^{\prime}$ of $X$ such that $D=fD^{\prime}$, $f \in K(X)$,  $D^{\prime}(I_Z)\subset I_Z$ and $D^{\prime}(\mathcal{O}_X) \not\subset I_Z$ ~\cite{R-S76}. 

Let $\mathcal{F}$ be a coherent sheaf on $X$. By $\mathcal{F}^{[n]}$ we denote the double dual $(\mathcal{F}^{\otimes n})^{\ast\ast}$. 

For any prime number $l\not= p$, the cohomology groups $H_{et}^i(X,\mathbb{Q}_l)$ are independent of $l$, they are finite dimensional of $\mathbb{Q}_l$ and are called the $l$-adic cohomology groups of $X$. The $i$-Betti number $b_i(X)$ of $X$ is defined to be the dimension of $H_{et}^i(X,\mathbb{Q}_l)$. It is well known that $b_i(X)=0$ for any $i>2n$, where $n=\dim X$~\cite{Mi80}. 

The \'etale Euler characteristic of $X$ is defined by \[
\chi_{et}(X)=\sum_{i}(-1)^i\dim_{\mathbb{Q}_l}H^i(X,\mathbb{Q}_l)=\sum_i(-1)^ib_i(X).
\]
If $X$ is a smooth surface then $c_2(X)=\chi_{et}(X)$~\cite{Mi80}. Both the Betti numbers and the \'etale euler characteristic are invariant under \'etale equivalence. In particular, if $f \colon X \rightarrow Y$ is a purely inseparable morphism of varieties, then $f$ induces an equivalence of the \'etale sites of $X$ and $Y$ and hence $b_i(X)=b_i(Y)$ and $\chi_{et}(X)=\chi_{et}(Y)$~\cite{Mi80}.

$X$ is called  simply connected if $\pi_1^{et}(X)=\{1\}$, where $\pi_1^{et}(X)$ is the \'etale fundamental group of $X$.

The symbol $\equiv$ denotes numerical equivalence of divisors.

We will use the terminology of terminal, canonical, log terminal and log canonical singularities as in~\cite{KM98}. Their definition and basic properties, in particular~\cite[Corollary4.2, Corollary 4.3, Theorem 4.5]{KM98} are independent of the characteristic of the base field and therefore their theory applies in positive characteristic too. The contraction theorems~\cite{Art62},~\cite{Art66} are also independent of the characteristic and will be used frequently in this paper.

Let $P\in X$ be a normal surface singularity and $f \colon Y \rightarrow X$ its minimal resolution. If $P\in X$ is canonical, then $K_Y=f^{\ast}K_X$. By~\cite{KM98} canonical surface singularities are classified according to the Dynkin diagrams of their minimal resolution and they are called accordingly of type $A_n$, $D_n$, $E_6$, $E_7$ and $E_8$. In characteristic zero these are exactly the DuVal singularities and their dynkin diagrams correspond to explicit equations. However in positive characteristic I am not aware of a classification with respect to local equations. In this paper, canonical surface singularities will be distinguished according to their dynkin diagrams.   

A Godeaux surface is a surface of general type with the lowest possible numerical invariants and it is a classical object of study (at least in characteristic zero). More precisely,
\begin{definition}\cite{Re78}
A numerical Godeaux surface is a minimal surface $X$ of general type such that $K_X^2=1$ and $\chi(\mathcal{O}_X)=1$. 
\end{definition}
Numerical Godeaux surfaces are divided in two disjoint classes. \textit{Classical} and \textit{Nonclassical}. A Godeaux surface $X$ is called classical if and only if the torsion part $\mathrm{Pic}^{\tau}(X)$ of $\mathrm{Pic}(X)$ is reduced, and nonclassical if and only if $\mathrm{Pic}^{\tau}(X)$ is not reduced. Nonclassical are further divided into two disjoint classes. \textit{Singular} if $\mathrm{Pic}^{\tau}(X)$ contains $\mu_p$ and \textit{Supersingular} if $\mathrm{Pic}^{\tau}(X)$ contains $\alpha_p$ (since $\mathrm{Pic}^{\tau}(X)$ is a finite commutative group scheme with one dimensional tangent space, this cases are mutually exclusive).

Since all group schemes are reduced in characteristic zero, nonclassical Godeaux surfaces appear only in positive characteristic $p>0$. It is well known that in characteristic zero  
$\mathrm{Pic}^{\tau}(X)$ can be $\{1\}$, $\mathbb{Z}/2\mathbb{Z}$,  $\mathbb{Z}/3\mathbb{Z}$,$\mathbb{Z}/4\mathbb{Z}$ and  $\mathbb{Z}/5\mathbb{Z}$~\cite{Re78}. In positive characteristic however, $\mathrm{Pic}^{\tau}(X)$ may have a non-reduced part and one gets either singular or supersingular Godeaux surfaces too. However, non classical Godeaux surfaces exist only for $p\leq 5$~\cite{Li09}.

Classical Godeaux surfaces have been classified in characteristic zero by M. Reid~\cite{Re78}, and in characteristic $p=5$ by W. Lang~\cite{La83}. Nonclassical in characteristic 5 have been classified by C. Liedtke~\cite{Li09}. There is no classification yet in characteristics 2 or 3  but it is expected that all the characteristic zero cases appear also  together with cases where $\mathrm{Pic}^{\tau}(X)=\alpha_p$ or $\mu_p$, $p=2, 3$. 

There is a correspondence between the structure of $\mathrm{Pic}^{\tau}(X)$ and the existence of torsors over $X$.  The correspondence is given by the isomorphism $\mathrm{Hom}(G, \mathrm{Pic}^{\tau}(X))\cong H^1_{fl}(X,G^{\ast})$~\cite{Ray70}, where $G$ is any commutative subgroup scheme of $\mathrm{Pic}^{\tau}(X)$ and $G^{\ast}$ is its Cartier dual. According to this, a Godeaux surface is singular if there is a $\mathbb{Z}/p\mathbb{Z}$ torsor and supersingular if there is an $\alpha_p$ torsor over $X$. Equivalently, if the induced map of the Frobehious on $H^1(\mathcal{O}_X)$ is either bijective or zero.

%% file: sec1.tex
\section{Nonreducedness happens for small p.}\label{sec-1}

The purpose of this section is to show that nonreducedness of the automorphism scheme of a canonically polarised normal surface defined over an algebraically closed field of characteristic $p>0$ is a property that happens for relatively small values of $p$. Moreover, the length of the automorphism schemes of canonically polarised surfaces with a fixed Hilbert polynomial is bounded by a number that depends only on the Hilbert polynomial and not on the characteristic of the base field. In particular I will show the following.

\begin{theorem}\label{th1}
Let $f(x)\in \mathbb{Q}[x]$ be a numerical polynomial. Then there are positive integers $m_0$ and $M_0$ depending only on $f(x)$ such that for any Gorenstein canonically polarised surface $X$ defined over an algebraically closed field $k$ of characteristic $p>0$ and with Hilbert polynomial $f(x)$,
\begin{enumerate} 
\item \[
\mathrm{length}(\mathrm{Aut}(X)) \leq M
\]
\item $\mathrm{Aut}(X)$ is smooth for all $p>m_0$. 
\end{enumerate}
\end{theorem}

\begin{proof}
Let $\Omega$ be the set of all Gorenstein canonically polarised surfaces with fixed Hilbert polynomial $f(n)$ defined over any field of any characteristic. Then this set is bounded~\cite{Ko84}~\cite{M70}~\cite{M-M64}~\cite{M86}. This means that there is a flat morphism $f \colon \mathcal{X} \rightarrow S$, where $S$ is of finite type over $\mathbb{Z}$ whose geometric fibers are Gorenstein canonically polarised surfaces and such that for any Gorenstein canonically polarised surface $X$ with Hilbert polynomial $f(n)$ defined over an algebraically closed field $k$, there is a morphism $\mathrm{Spec}k \rightarrow S$ such that $X \cong \mathrm{Spec}k \times_S \mathcal{X}$. Let
\[
\Phi \colon \mathrm{Aut}(\mathcal{X}/S) \rightarrow S
\]
be the induced morphism on relative automorphism schemes. I will show that this map is finite. For this it suffices to show that $\Phi$ is proper with finite fibers. It is well known that for any canonically polarised surface $X$, $\mathrm{Aut}(X)$ is a finite group scheme. Therefore $\Phi$ has finite fibers. Properness of $\Phi$ follows from the valuative criterion of properness. This is equivalent to the following property. Let $R$ be a discrete valuation ring with function field $K$ and residue field $k$ (perhaps of mixed characteristic). Let $X \rightarrow \mathrm{Spec}R$ be a projective flat morphism such that $\omega_{X/R}$ is ample. Let $X_K$ and $X_k$ be the generic and special fibers. Then any automorphism of $X_K$ lifts to an automorphism of $X$ over $\mathrm{Spec}R$. The proof of this statement is identical with the one for characteristic zero~\cite[Proposition 3]{Ko10} and I omit its proof. It essentially depends on the existence of resolutions of singularities which exist also in 
any characteristic in dimension two.

Now since $\Phi$ is a finite morphism the lengths of its fibers are bounded by some number $M$ that depends only on $f(x)$. Hence if $X$ is any Gorenstein canonically polarised surface defined over an algebraically closed field $k$ of characteristic $p>0$, $\mathrm{length}(\mathrm{Aut}(X))\leq M$. Therefore if $p>M$, $\mathrm{Aut}(X)$ is smooth over $k$. This follows from the fact that $\mathrm{Aut}(X)$ is a finite group scheme defined over a field of characteristic $p>0$ and any such group scheme is smooth if $p$ is bigger than its length~\cite{Mu70}~\cite{Mi12}. This concludes the proof of the theorem.
\end{proof}
The previous theorem motivates the following problem.

\begin{problems}
Let $X$ be a canonically polarised surface defined over an algebraically closed field of characteristic $p>0$.
\begin{enumerate}
\item Find effective bounds for the length of $\mathrm{Aut}(X)$, or its component containing the identity, depending only on its Hilbert polynomial. In characteristic zero it is known that $|\mathrm{Aut}(X)| \leq 42K_X^2$~\cite{Xi94},~\cite{Xi95}.
\item Find explicit relations between the characteristic $p$ of the base field and the coefficients of the Hilbert polynomial of a canonically polarised surface which guarantee the smoothness of $\mathrm{Aut}(X)$.
\end{enumerate}
\end{problems}

\begin{remarks}
\begin{enumerate}
\item The previous theorem shows that the structure of the automorphism scheme is expected to be more complicated in small characteristics. It also suggests that if $1\leq K_X^2 \leq 2$, then the case $p=2$, may be the case where most pathologies appear.
\item The proof of the theorem depends on boundedness of canonically polarised surfaces with a given Hilbert polynomial. At the time of this writing this is not known in higher dimensions.  
\end{enumerate}

\end{remarks}

%% file: sec2.tex
\section{Nonreducedness of the automorphism scheme, derivations and group scheme actions.}\label{sec-2}
Let $X$ be a canonically polarized $\mathbb{Q}$-Gorenstein surface defined over an algebraically closed field $k$ of characteristic $p>0$. The purpose of this section is essentially  to show that $\mathrm{Aut}(X)$ is not smooth over $k$ if and only if $X$ admits a nontrivial global vector field or equivalently if it admits a nontrivial $\mu_p$ or $\alpha_p$ action. These results are easy and probably known but I include them here for lack of reference and the convenience of the reader.

\begin{proposition}\label{prop1}
Let $X$ be a canonically polarized $\mathbb{Q}$-Gorenstein surface defined over an algebraically closed field $k$ of characteristic $p>0$. Then $\mathrm{Aut}(X)$ is not smooth over $k$ if and only if $X$ admits a nontrivial global vector field of either additive or multiplicative type.
\end{proposition}

\begin{proof}
It is well known that for any canonically polarized $\mathbb{Q}$-Gorenstein surface, $\mathrm{Aut}(X)$ is finite. Therefore it is not smooth if and only if its tangent space at the identity is not trivial. Now the tangent space of $\mathrm{Aut}(X)$ at the identity is $\mathrm{Hom}_X(\Omega_X,\mathcal{O}_X)$ which is the space of global derivations of $X$. Therefore $\mathrm{Aut}(X)$ is not smooth if and only if $X$ has a nontrivial global vector field. 

Now let \[
\Phi \colon \mathrm{Hom}_X(\Omega_X,\mathcal{O}_X) \rightarrow \mathrm{Hom}_X(\Omega_X,\mathcal{O}_X)
\]
be the map defined by $\Phi(D)=D^p$. This is a $p$-linear map. Therefore by~\cite[Lemma 4.13]{Mi80}, there is a decomposition $\mathrm{Hom}_X(\Omega_X,\mathcal{O}_X)=V_n\oplus V_s$ where $V_n$, $V_s$ are $\Phi$-invariant subspaces such that the restriction of $\Phi$ to $V_n$ is nilpotent and to $V_s$ bijective. If $V_n \not= \emptyset$ then $X$ admits a vector field of additive type and if $V_s\not= \emptyset$ then it admits a vector field of multiplicative type. 
\end{proof}
\begin{corollary}\label{lifts-to-zero}
Let $X$ be a smooth canonically polarized surface which lifts to characteristic zero, or to $W_2(k)$, the ring of 2-Witt vectors. Then $\mathrm{Aut}(X)$ is smooth.
\end{corollary}
\begin{proof}
Since $X$ lifts to characteristic zero or $W_2(k)$, then Kodaira-Nakano vanishing holds for $X$~\cite{DI87},~\cite{EV92} and hence
\[
\mathrm{Hom}_X(\Omega_X,\mathcal{O}_X)=H^0(T_X)=H^0(\Omega_X\otimes \omega_X^{-1})=0,
\]
since $\omega_X$ is ample. Hence $X$ has no nontrivial global vector fields and therefore by Proposition~\ref{prop1}, $\mathrm{Aut}(X)$ is smooth.
\end{proof}
This result is a kind of accident since the smoothness of the automorphism scheme is not the consequence of any vanishing theorems but rather the fact that any group scheme is smooth in characteristic zero. Moreover, this result does not hold for singular surfaces as shown by the examples in Section~\ref{examples}.

\begin{proposition}[Proposition 3.1~\cite{Tz14}]\label{prop2}
Let $X$ be a canonically polarized $\mathbb{Q}$-Gorenstein surface defined over an algebraically closed field $k$ of characteristic $p>0$. Then $\mathrm{Aut}(X)$ is not smooth over $k$ if and only if $X$ admits a nontrivial $\alpha_p$ or $\mu_p$ action.
\end{proposition}

An immediate corollary of this is the following.

\begin{corollary}\label{subgroup-of-aut}
Let $X$ be a canonically polarized $\mathbb{Q}$-Gorenstein surface defined over an algebraically closed field $k$ of characteristic $p>0$. Then $\mathrm{Aut}(X)$ is not smooth if and only if it has a subgroup scheme isomorphic to either $\alpha_p$ or $\mu_p$.
\end{corollary}

%% file: sec3.tex
\section{Quotients by $\alpha_p$ or $\mu_p$ actions.}\label{sec-3}
In this section, unless otherwise stated,  $X$ denotes a scheme of finite type over an algebraically closed field $k$ of characteristic $p>0$ which admits a nontrivial $\alpha_p$ or $\mu_p$ action. As explained in the previous section, such an action is induced  by a global nontrivial vector field $D$ of $X$ of either additive or multiplicative type. Let $\pi \colon X \rightarrow Y$ be the quotient of $X$ by the $\alpha_p$ or $\mu_p$ action (which exists as an algebraic scheme by~\cite{Mu70}). Locally, if $X=\mathrm{Spec}A$, then $Y=\mathrm{Spec}B$, where $B=A^D=\{a\in A , \; Da=0\}$.

The purpose of this section is to describe the structure of the map $\pi$ and the singularities of $Y$.

\begin{definition}\cite{Sch07}
\begin{enumerate}
\item The fixed locus of the action of $\alpha_p$ or $\mu_p$ (or of $D$) on $X$ is the closed subscheme of $X$ defined by the ideal sheaf generated by $D(\mathcal{O}_X)$.
\item A point $P \in X$ is called an isolated singularity of $D$ if there is an embeded component $Z$ of the fixed locus of $D$ such that $P \in Z$. The vector field $D$ is said to have only divisorial singularities if the ideal $D(\mathcal{O}_X)$ has no embeded components.
\end{enumerate}
\end{definition}

The next proposition gives some general information about the singularities of $Y$. 
\begin{proposition}\label{prop3}
Let $X$ be an integral scheme of finite type over an algebraically closed field of characteristic $p>0$. Suppose $X$ has an $\alpha_p$ or $\mu_p$ action induced by a vector field $D$ of either additive or multiplicative type. Let $\pi \colon X \rightarrow Y$ be the quotient. Then
\begin{enumerate}
\item If $X$ is normal then $Y$ is normal.
\item If $X$ is $S_2$ then $Y$ is $S_2$ as well.
\item If $X$ is smooth then the singularities of $Y$ are exactly the image of the embedded part of the fixed locus of the action.
\item If $X$ is normal and $\mathbb{Q}$-Gorenstein, then $Y$ is also $\mathbb{Q}$-Gorenstein. In particular, let $D$ be a divisor in $Y$ and $\tilde{D}$ be the divisorial part of $\pi^{-1}(D)$. Then if $n\tilde{D}$ is Cartier, $pnD$ is Cartier too.
\end{enumerate}
\end{proposition}

\begin{proof}
Normality of $Y$ is a local property so we may assume that $X$ and $Y$ are affine. Let $X=\mathrm{Spec}A$ and $Y=\mathrm{Spec}B$, where $B=\{a\in A, \; Da=0\} \subset A$. Let $\bar{B} \subset K(B)$ be the integral closure of $B$ in its function field $K(B)$. Let $z=b_1/b_2 \in \bar{B}$. Then since $K(B)\subset K(A)$ and $A$ is normal, $z\in A$. But $Dz=D(b_1/b_2)=(b_2Db_1-b_1Db_2)/b_2^2=0$. Therefore $z\in B$ and hence $B$ is integrally closed.

The fact that if $X$ is $S_2$, so is $Y$ is proved in ~\cite{Sch07} and the statement that if $X$ is smooth then the singularities of $Y$ are exactly the image of the embeded part of the fixed locus of $D$ is in~\cite{AA86}.

Suppose that $X$ is $\mathbb{Q}$-Gorenstein. Let $D$ be a divisor on $Y$. The property that $D$ is $\mathbb{Q}$-Cartier is local so we may assume that $X$ and $Y$ are affine, say $X=\mathrm{Spec}A$ and $Y=\mathrm{Spec}B$. Then $D$ is $\mathbb{Q}$-Cartier if and only if $nD=0$ in $\mathrm{Cl}(B)$, for some $n \in \mathbb{N}$. Consider the natural map \[
\phi\colon \mathrm{Cl} (B) \rightarrow \mathrm{Cl} (A)
\]
Then according to~\cite{Fo73}, \[
\mathrm{Ker}\phi \subset  H^1(G, A^{\ast}+K(A)t)
\]
where $G$ is the additive subgroup of $Der_k(A)$ generated by $D$, $A^{\ast}+K(A)t \subset K(A)[t]/(t^2)$ is the multiplicative subgroup and $G$ acts on it by the usual automorphisms induced by $D$. Then since $k$ has characteristic $p>0$, $G \cong \mathbb{Z}/p\mathbb{Z}$  and therefore $H^1(G, A^{\ast}+K(A)t)$ is $p$-torsion. Therefore $\mathrm{Ker}\phi$ is $p$-torsion as well. Hence if $n\tilde{D}$ is $\mathbb{Q}$-Cartier, then $nD\in \mathrm{Ker}\phi$ and hence $pnD=0$ in $\mathrm{Cl} (B)$. Therefore $pnD$ is Cartier as claimed. 
\end{proof}

Even if $X$ is smooth, it is very hard to give more detailed information about the singularities of $Y$ , even more to classify them. The difficulty arises mainly from the complex structure of $\alpha_p$ actions and quotients. In this case the quotient may not even have rational singularities (such examples can be found in~\cite{Li08}). Quotients by $\mu_p$ are much easier to describe. The characteristic 2 case is also simpler (as is probably expected from the characteristic zero case of quotients by $\mathbb{Z}/2\mathbb{Z}$). 

\begin{definition}
Let $P \in X$ be a surface singularity defined over an algebraically closed field $k$. Let $f \colon Y \rightarrow X$ be its minimal resolution.  $P \in X$ is called a toric singularity of type $\frac{1}{p}(1,m)$, $1\leq m\leq p-1$, where $p$ is a prime number,  if the exceptional set of $f$ is a chain of smooth rational curves
\[
\underset{E_1}{\bullet}-\underset{E_2}{\bullet}-\cdots -\underset{E_m}{\bullet}
\]
such that the intersection numbers $-b_i=E_i^2$, are obtained from the continuous fraction decomposition of $p/m$. If the characteristic of the base field is not $p$, then $P\in X$ is locally analytically isomorphic to the cyclic quotient singularity $\mathbb{A}^2_k/\mu_p$, where the group $\mu_p$ of $p$-roots of unity act on $\mathbb{A}^2_k$ by $\zeta \cdot x =\zeta x$, $\zeta \cdot y =\zeta^m y$, where $\zeta$ is a primitive $p$-root of unity and $x$, $y$ the coordinates of $\mathbb{A}^2_k$. However, if $p$ is the characteristic of the ground field $k$, then $P \in X$ is locally analytically isomorphic to the singularity $\mathbb{A}^2_k/\mu_p$, where the action of the multiplicative group scheme  $\mu_p$ on $\mathbb{A}^2_k$ is induced by the multiplicative closed derivation $D=x\partial/\partial x +m y \partial /\partial y$~\cite{Hi99}.
\end{definition}

The previous class of singularities is exactly the class of singularities that appear as $\mu_p$ quotients.

\begin{proposition}\cite{Hi99}\label{class-of-mult-quot}
Let $X$ be a smooth surface with a nontrivial $\mu_p$ action. Let $\pi \colon X \rightarrow Y$ be the quotient, Then the singularities of $Y$ are toric singularities of type $\frac{1}{p}(1,m)$, $m=1,2,\ldots, p-1$
\end{proposition}

The next proposition provides some more information about the singularities of the quotient that will be useful in this paper. 

\begin{proposition}\label{prop4}
Let $X$ be a smooth surface with a nontrivial $\alpha_p$ or $\mu_p$ action induced by a global vector field $D$ or either additive or multiplicative type. Let $\pi \colon X \rightarrow Y$ be the quotient. Then there exists a commutative diagram 
\begin{gather}\label{prop4-diagram}
\xymatrix{
X^{\prime}\ar[r]^f \ar[d]^{\pi^{\prime}} & X \ar[d]^{\pi} \\
Y^{\prime} \ar[r]^g & Y \\
}
\end{gather}
such that 
\begin{enumerate}
\item $g \colon Y^{\prime} \rightarrow Y$ is the minimal resolution of $Y$, $X^{\prime}$ is normal and $f$ is birational.
\item The vector field $D$ on $X$ lifts to a global vector field $D^{\prime}$ on $X^{\prime}$ and $\pi^{\prime} \colon X^{\prime} \rightarrow Y^{\prime}$ is the quotient of $X^{\prime}$ by $D^{\prime}$.
\end{enumerate}
In particular, if $p=2$ then in addition to the above,
\begin{enumerate}\setcounter{enumi}{2}
\item $Y$ is Gorenstein. Moreover, if $Y$ has canonical singularities, then $Y$ has singularities of type either $A_1$, $D_{2n}$, $E_7$ or $E_8$.
\item Every $f$-exceptional curve is contained in the divisorial part of $D^{\prime}$.
\item Suppose that $D$ is of multiplicative type. Then,
\begin{enumerate}
\item $X^{\prime}$ is smooth, $D^{\prime}$ has no isolated singular points and $f\colon X^{\prime} \rightarrow X$ is a sequence of blow ups of the isolated singular points of $D$.
\item The divisorial part of $D$ is smooth, disjoint from the isolated singular points of $D$ and is not an integral curve of $D$ (this holds in all characteristics).
\end{enumerate}
\item If $D$ is of additive type then a diagram like (\ref{prop4-diagram}) exists where both $X^{\prime}$ and $Y^{\prime}$ are smooth, but $Y^{\prime}$ is not necessarily the minimal resolution of $Y$ (\cite{Hi99}).
\end{enumerate}
\end{proposition}

\begin{proof}
As mentioned in the statement of the proposition,~\ref{prop4}.6 was proved by M. Hirokado~\cite{Hi99}. It remains to prove parts 1-5.

Let $g \colon Y^{\prime} \rightarrow Y$ be the minimal resolution of $Y$. Let $\pi^{\prime} \colon X^{\prime}\rightarrow Y^{\prime}$ be the normalization of $Y^{\prime}$ in the field of fractions $K(X)$ of $X$. Then $\pi^{\prime}$ is a purely inseparable map of degree $p$. I will show that there exists a map $f \colon X^{\prime} \rightarrow X$ such that $\pi f =g\pi^{\prime}$ giving rise to the diagram~\ref{prop4-diagram}. The rational map $X\dasharrow Y^{\prime}$ defined by $\pi$ and $g$ is resolved after a sequence of blow ups of $X$. Therefore there exists a commutative diagram
 \[
\xymatrix{
 X^{\prime}\ar[dr]_{\pi^{\prime}} &   Z\ar[l]_{\psi}\ar[r]^{\phi} \ar[d]^{\delta} & X \ar[d]^{\pi} \\
       &Y^{\prime} \ar[r]^g & Y \\
}
\]
 where $\phi$ is a sequence of blow ups resolving $X\dasharrow Y^{\prime}$ and $\psi \colon Z \rightarrow X^{\prime}$ is the factorization of $\delta$ through $X^{\prime}$ which exists  since $Z$ is normal in $K(X)$. Since 
 $\pi^{\prime}$ and $\pi$ are finite morphisms, it follows that every $\psi$-exceptional curve is also a $\phi$-exceptional curve. Therefore the rational map $f=\phi\psi^{-1}$ is in fact a morphism and hence there exists a commutative diagram as claimed in~\ref{prop4}.1.

Next I will show that the vector field $D$ of $X$ lifts to a vector field $D^{\prime}$ of $X^{\prime}$. Since $X$ and $X^{\prime}$ are birational, $D$ gives a rational vector field $D^{\prime}$ of $X^{\prime}$. It is not hard to see that $Y^{\prime}=X^{\prime}/D^{\prime}$. Then in order to show that $D^{\prime}$ is regular it suffices to show that it has no poles. Let $\Delta$ be the divisorial part of $D$. Then~\cite{R-S76}
 \[
K_X=\pi^{\ast}K_Y+(p-1)\Delta.
\]
Moreover, since $Y^{\prime}$ is the minimal resolution of $Y$,
\[
K_{Y^{\prime}}=g^{\ast}K_Y-F,
\]
where $F$ is an effective $g$-exceptional $\mathbb{Q}$-divisor. Therefore from the commutative diagram~\ref{prop4-diagram} it follows that
\begin{gather*}
K_{X^{\prime}}=f^{\ast}K_X+E=f^{\ast}(\pi^{\ast}K_Y+(p-1)\Delta)+E=(\pi^{\prime})^{\ast}g^{\ast}K_Y+(p-1)f^{\ast}\Delta+E= \\
(\pi^{\prime})^{\ast}K_{Y^{\prime}}+(\pi^{\prime})^{\ast}F+(p-1)f^{\ast}\Delta+E,
\end{gather*}
where $E$ is an $f$-exceptional divisor. But the last adjunction formula shows that the divisor of $D^{\prime}$ is $(\pi^{\prime})^{\ast}F+(p-1)f^{\ast}\Delta+E$, which is effective. 
Hence $D^{\prime}$ has no poles and therefore it is regular. This concludes the proof of~\ref{prop4}.1 ~\ref{prop4}.2.

Suppose that $p=2$. Then $\pi$ factors through the geometric Frobenious $F \colon X \rightarrow X^{(2)}$. In fact there is a commutative diagram
\[
\xymatrix{
     & Y \ar[dr]^{\nu} \\
X \ar[ur]^{\pi}\ar[rr]^F & & X^{(2)}
}
\]
Since $X^{(2)}$ is smooth and $Y$ is normal, then $\nu$ is an $\alpha_L$-torsor over $X^{(2)}$, for some line bundle $L$ on $X^{(2)}$~\cite{Ek87}. Then since $X$ is smooth, $Y$ has hypersurface singularities and therefore it is Gorenstein.
Suppose that $Y$ has canonical singularities. Then the dynking diagram of any singular point of $Y$ is of type either $A_n$, $D_n$, $E_6$, $E_7$ or $E_8$~\cite{KM98}. By Proposition~\ref{prop3}.4, the local Picard groups of the singular points of $Y$ are 2-torsion. Therefore these can be only $A_1$, $D_{2n+1}$, $E_7$ or $E_8$. This shows~\ref{prop4}.3.

Suppose that $D$ is of multiplicative type. Let $\Delta$ be its divisorial part. Then in suitable local coordinates of a fixed point $P \in X$ of $D$, $D$ is given by $D=ax\partial /\partial x +by \partial /\partial y$, where $a, b \in \mathbb{F}_p$~\cite{R-S76}. This shows immediately that the divisorial part $\Delta$ of $D$ is smooth, it is disconnected from the isolated singular points of $D$ and is not an integral curve of $D$. Therefore, if $\Delta^{\prime}$ is the image of $\Delta$ in $Y$ with reduced structure, then $\pi^{\ast}\Delta^{\prime}=2\Delta$. Moreover, it is a straightforward calculation to find the lifting $D^{\prime}$ of $D$ on the blow up $X^{\prime}$ of $X$ at $P$ and see that indeed the exceptional curve is contained in the divisorial part of the fixed locus of $D^{\prime}$. 

Next I will show that $X^{\prime}$ is smooth. Since $D$ is multiplicative, it follows from Proposition~\ref{class-of-mult-quot} that $Y$ has canonical singularities of type $A_1$. The $g$-exceptional curves are exactly the images under $\pi^{\prime}$ of the $f$-exceptional curves. Let $E$ be an $f$-exceptional curve and $F$ its image. Then $F^2=-2$ and $(\pi^{\prime})^{\ast}F=2E$. Then $2F^2=4E^2$ and hence $E^2=-1$. Hence all $f$-exceptional curves are smooth rational curves of self intersection -1. Therefore $f\colon X^{\prime}\rightarrow X$ is a sequence of blow ups and in particular $X^{\prime} $ is smooth. Since the singularities of $Y$ are under the isolated fixed points of $D$, $f$ is obtained by successively blowing up isolated singular points of $D$. This shows~\ref{prop4}.5

It remains to show that if $D$ is of additive type then every $f$-exceptional curve is contained in the divisorial part of $D^{\prime}$. By~\ref{prop4}.6, there exists a diagram like~\ref{prop4-diagram} such that both $X^{\prime}$ and $Y^{\prime}$ are smooth, but unlike the multiplicative case, $Y^{\prime}$ is not necessarily the minimal resolution of $Y$. Hence it suffices to consider this case only. Then the map $f$ is obtained by successively blowing up the isolated singular points of $D$. In order then to show that every $f$-exceptional curve is contained in the divisorial part of the fixed locus of $D^{\prime}$, it suffices to assume that $f$ is a single blow up and show that $D^{\prime}$ induces the zero derivation on the exceptional curve $E$. 

Suppose $P$ is on the divisorial part of the fixed locus. Then in suitable local coordinates, $D=h(f \partial /\partial x +f \partial / \partial y)$ such that $f,g,h \in m_P$, where $m_P$ is the maximal ideal of $\mathcal{O}_{X,P}$ and $f,g$ have no common factor. Then $Dx =fh \in m_P^2$ and $Dy=hg\in m_P^2$. Let $E$ be the $f$-exceptional curve. Then $E= \mathrm{Proj} R$, where 
\[
R= \oplus_{d\geq 0} m_P^d/m_p^{d+1}
\]
and $D$ induces a graded derivation of $R$. But since $D(m_P) \subset m_P^2$, it follows that in fact the induced derivation is the zero derivation.

Suppose now that $P$ is an isolated singular point of $D$ that does not belong on the divisorial part of the fixed locus. Then again in suitable local coordinates, $D=f \partial /\partial x +g \partial / \partial y$ such that $f,g\in m_P$ have no common factor. I will show that $f, g \in m_P^2$, and hence $D(\mathcal{O}_X)\subset m_p^2$. Indeed, since we are in characteristic 2, $\partial^2/\partial x^2=\partial^2/\partial y^2 =0$. Then an easy calculation shows that
\[
D^2=\left( f\frac{\partial f}{\partial x}+g\frac{\partial f}{\partial y}\right) \frac{\partial}{\partial x} +\left( f\frac{\partial g}{\partial x}+g\frac{\partial g}{\partial y}\right) \frac{\partial}{\partial y}.
\]
Now the relation $D^2=0$ implies that 
\begin{eqnarray*}
f\frac{\partial f}{\partial x}=g\frac{\partial f}{\partial y} & \text{and} & f\frac{\partial g}{\partial x}=g\frac{\partial g}{\partial y}.
\end{eqnarray*}

Suppose that at least one of $f,g$ is not in $m_P^2$. Suppose $f \in m_P-m_P^2$. Then $\frac{\partial f}{\partial x}\not= 0$ and $\frac{\partial f}{\partial y}\not= 0$ because if $\frac{\partial f}{\partial x}=\frac{\partial f}{\partial y}= 0$, then $f(x,y)=f_1(x^2,y^2)$ and hence $f\in m_P^2$. Now considering that $f$ and $g$ have no common factor, it follows that there is $\phi \in \mathcal{O}_X$ such that 
\begin{eqnarray*} 
\frac{\partial f}{\partial x} = g \phi & \text{and} & \frac{\partial f}{\partial y} = f \phi. \\
\end{eqnarray*}
Now since $f \in m_P-m_P^2$, it follows that
\[
f(x,y)=ax+by +f_{\geq 2}(x,y),
\]
where $f_{\geq 2}(x,y)\in m_P^2$, $a,b, \in k$, not both zero. Then either $\partial f/\partial x \in \mathcal{O}_X^{\ast}$ or $\partial f/\partial y \in \mathcal{O}_X^{\ast}$. But then either $g \in \mathcal{O}_X^{\ast}$ or $f \in \mathcal{O}_X^{\ast}$, which is impossible. Hence $D(m_P)\subset m_P^2$. Now arguing exactly as in the case when the singular point $P$ is on the divisorial part of the fixed locus of $D$ we get that the lifting $D^{\prime}$ of $D$ on $X^{\prime}$ restricts to zero on the exceptional curve.
\end{proof}

\begin{remark}
\item Proposition~\ref{prop4}.5 essentially says that if $p=2$ then the isolated singularities of the vector field $D$ can be resolved by a sequence of blow ups. If $p>2$ then this is not possible in general. Take for example  $X=\mathbb{A}^2_k$, $p=5$ and $D=x \partial/\partial x+2y\partial/\partial y$. This is a vector field of multiplicative type. Suppose that a diagram like in Proposition~\ref{prop4} exists with both $X^{\prime}$ and $Y^{\prime}$ smooth. Then $f$ is obtained by successively blowing up the isolated singularities of $D$. Let $X_1\rightarrow X$ be the first blow up, i.e., the blow up of the singular point of $D$. Then a straightforward calculation shows that the lifting $D_1$ of $D$ on $X_1$ has exactly two isolated singular points, say $P$ and $Q$. Moreover, locally at $P$, $D_1=x   \partial/\partial x+y\partial/\partial y$ and locally at $Q$, $D_1=2(x   \partial/\partial x+2y\partial/\partial y)$. 
Hence at $Q$, $D_1$ has exactly the same form as $D$. Hence every time a singular point is blown up at which the vector field has the form $\lambda (x   \partial/\partial x+2y\partial/\partial y)$, $\lambda \in \mathbb{Z}_p$, another singular point will appear in the blow up where the lifted vector field will have the same form. Hence the process of blowing up the isolated singular points of the vector field does not lead to a vector field without isolated singular points and hence a diagram like in Proposition~\ref{prop4}  does not exist in this case.

However, even though there is no resolution in general  of the isolated singularities of $D$ by usual blow ups as in the case $p=2$, Proposition~\ref{prop4}.1,2 says that there exists a partial resolution by weighted blow ups instead, hence the singularities on $X^{\prime}$.
\end{remark}
\begin{remark}
If $D$ is of additive type it might happen that, unlike in the multiplicative case, the divisorial part of it is an integral divisor of $D$. For example, let $p=2$, $X=\mathbb{A}^2_k$ and $D=x^2\partial /\partial x +xy \partial /\partial y$. Then $D^2=0$, its divisorial part is given by $x=0$ and it is an integral curve of $D$. 
\end{remark}

The next proposition describes the structure of the quotient map $\pi$.

\begin{proposition}\cite{Tz14}\label{structure-of-pi}
Let $X$ be a normal Cohen Macauley integral scheme of finite type over an algebraically closed field $k$ of characteristic $p>0$ with a $\mu_p$ or $\alpha_p$ action induced by a vector field $D$ or either multiplicative or additive. Let $\pi \colon X \rightarrow Y$ be the quotient. Then:
\begin{enumerate}
\item Suppose that $X$ admits a $\mu_p$ action. Then there is a direct sum decomposition of $\mathcal{O}_Y$-modules
 \[
\pi_{\ast}\mathcal{O}_X=\oplus_{k=1}^{p-1}L_k(D),
\]
where $L_k(D)=\{a \in \mathcal{O}_X / \; Da=ka\}$ is a rank one torsion free sheaf of $\mathcal{O}_Y$-modules. 
\item Suppose that $X$ admits a $\alpha_p$ action. Then there exists a filtration of $\mathcal{O}_Y$-modules
\[
\mathcal{O}_Y=E_0\subsetneqq E_1 \subsetneqq E_2 \subsetneqq \cdots      \subsetneqq E_{k}  \subsetneqq E_{k+1}   \subsetneqq \cdots   \subsetneqq E_{p-2}   \subsetneqq E_{p-1} =\pi_{\ast}\mathcal{O}_X,
\]
where $E_k=\{a \in \mathcal{O}_X |\; D^{k+1}a=0\}$ is a reflexive sheaf of $\mathcal{O}_Y$-modules of rank $k+1$ and the quotients $L_k=E_k/E_{k-1}$ are torsion free sheaves of rank 1. Moreover, $E_k$ and $L_k$ are locally free outside the fixed locus of the $\alpha_p$ action, for all $k$. 
\item Suppose that $p=2$. Then $\pi \colon X \rightarrow Y$ is a $\mu_2$ or $\alpha_2$ torsor over a codimension 2 open subset of $Y$. In particular, there exists an exact sequence 
\[
0 \rightarrow \mathcal{O}_Y \rightarrow \pi_{\ast}\mathcal{O}_X \rightarrow L \rightarrow 0,
\]
where $L$ is a rank one reflexive sheaf on $Y$ and $\omega_X=\left(\pi^{\ast}(\omega_Y \otimes L^{-1})\right)^{[1]}$.
\end{enumerate}

\end{proposition}

\begin{remark}
The statement in Proposition~\ref{structure-of-pi}.3 has been proved by T. Ekedhal ~\cite{Ek86}. 
\end{remark}

\begin{remark}
Propositions~\ref{prop4},~\ref{structure-of-pi} show that the characteristic 2 case, despite all its pathologies, has certain advantages over higher characteristics. The quotient $X\rightarrow Y$ is a  torsor and the isolated singularities of the vector field $D$ of $X$ can be resolved by a sequence of blow ups.
\end{remark}

The next lemma and proposition relate the size of the singular locus of $Y$ with certain numerical invariants of $X$ in the case when $X$ is a smooth surface.

\begin{lemma}\label{ex-seq-1}
Let $D$ be a global vector field on a smooth surface $X$ defined over an algebraically closed field $k$ of characteristic $p>0$. Let $I_Z $ be the ideal sheaf of the embeded part $Z$ or the fixed locus of $D$ and let $\Delta$ be its divisorial part. Then 
\begin{enumerate}
\item There exists an exact sequence
\[
0 \rightarrow \mathcal{O}_X(\Delta) \rightarrow T_X \rightarrow L \otimes I_Z \rightarrow 0,
\]
where $L$ is an invertible sheaf on $X$.
\item Let $ P \in X$ be an isolated fixed point of $D$. Then locally in the \'etale topology  $\mathcal{O}_X=k[x,y]$,  $D=h(f\partial/\partial x +g \partial/\partial y)$, where $h,f,g \in k[x,y]$ and $f,g$ are relatively prime. Moreover, \[
\mathcal{O}_Z=\frac{\mathcal{O}_X}{(f,g)}.
\] 
In particular, if $D$ is of multiplicative type then $Z$ is reduced and its length is equal to the number of isolated fixed points of $D$.
\end{enumerate}
\end{lemma}
\begin{proof}
The vector field $D$ induces a short exact sequence
\[
0 \rightarrow \mathcal{O}_X(\Delta) \rightarrow T_X \rightarrow \mathcal{F} \rightarrow 0, \label{ex-seq-2}
\]
where $\mathcal{F}$ is torsion free and moreover the singular locus of the quotient $Y$ of $X$ by the action induced by $D$ is the image of the subset of $X$ where $\mathcal{F}$ is not free~\cite{Ek87}. Let $Q=\mathcal{F}^{\ast\ast}/\mathcal{F}$. Then $Q$ has finite support and its support is exactly the isolated fixed points of $D$ since by Proposition~\ref{prop3} the singular locus of $Y$ is the set theoretic image of the isolated fixed points of $D$. Let $I_Q$ be the ideal sheaf of $Q$ in $X$. Tensoring the exact sequence
\[
0 \rightarrow I_Q \rightarrow \mathcal{O}_X \rightarrow Q \rightarrow 0 
\]
with $\mathcal{F}^{\ast\ast}$ we get the exact sequence
\[
0 \rightarrow \mathcal{F}^{\ast\ast}\otimes I_Q \rightarrow \mathcal{F}^{\ast\ast} \rightarrow Q \rightarrow 0.
\]
Therefore \[
\mathcal{F}=\mathcal{F}^{\ast\ast}\otimes I_Q.
\]
Since $X$ is smooth, $\mathcal{F}^{\ast\ast}$ is invertible and then the first part of the proposition follows.

Next we will show the second part of the lemma. It will moreover imply that the scheme structure of $Q$ is the same as the scheme structure of the embeded part and therefore they are the same as schemes and not only as sets.

Let $P\in X$ be an isolated singularity of $X$. Then locally in the \'etale topology $\mathcal{O}_X=k[x,y]$. Hence there are $f,g,h\in k[x,y]$ with $f,g$ relatively prime such that 
$D=h(f\partial/\partial x +g \partial/\partial y)$. Then the embeded part of $D$ is given by the ideal $(f,g)$ and the divisorial by $(h)$. Then map $\mathcal{O}_X(\Delta) \rightarrow T_X$ is given by
\[
\Phi \colon \mathcal{O}_X \rightarrow \mathcal{O}_X\frac{\partial}{\partial x} \oplus \mathcal{O}_X \frac{\partial}{\partial y}
\]
defined by $\Phi(1)=f\partial /\partial x +g \partial/ \partial y$. It is now easy to see that the map 
\[
\Psi \colon \mathcal{O}_X\frac{\partial}{\partial x} \oplus \mathcal{O}_X \frac{\partial}{\partial y} \rightarrow (f,g)
\]
given by $\Psi(F\partial/\partial x +G \partial /\partial y)= Gf-Fg$ induces an isomorphism between the cokernel $\mathrm{CoKer} (\Phi)$ and the ideal $(f,g)$. Therefore $I_Z= I_Q$. 

Finally, suppose that $D$ is of multiplicative type, i.e., $D^p=D$. Then by~\cite{R-S76}, in suitable choice of the local parameters $x$ and $y$, $D=a x\partial /\partial x + b y \partial /\partial y$, where $a,b \in \mathbb{F}_p$, the finite field of order $p$. Therefore at an isolated singular point of $D$, $I_Z=(x,y)$. Therefore if $D$ is of multiplicative type $Z$ is reduced and its length is equal to the number of isolated singular points of $D$. 
\end{proof}
\begin{remark}
The exact sequence in~\ref{ex-seq-1}.1 is not new. However I am not aware of any reference of it and also an explicit description of the relation between $Z$ and the embeded part of the fixed locus of $D$. This is the reason that I have included its proof here. 
\end{remark}
\begin{remark}
If $D$ is not of multiplicative type then the embeded part $Z$ of the fixed locus of $D$ may be nonreduced and its length strictly bigger than the number of isolated singular point of $D$, and hence of the number of singular points of the quotient $Y$. For example, let $A=k[x,y]$ and $D=x^2\partial/\partial x +y^2 \partial /\partial y$. If $p=2$ then $D^2=0$. $D$ has exactly one isolated singular point but the embeded part of the fixed locus is given by the ideal $(x^2,y^2)$ and therefore has length $4$.  
\end{remark}
\begin{proposition}\label{size-of-sing}
Let $X$ be a smooth surface defined over an algebraically closed field $k$ of characteristic $p>0$ . Let $D$ be a nontrivial global vector field on $X$ and let $I_Z $ be the ideal sheaf of the embeded part $Z$ or the fixed locus of $D$ and let $\Delta$ be its divisorial part. Then
\[
\mathrm{length}(\mathcal{O}_Z)=K_X \cdot \Delta +\Delta^2 +c_2(X).
\] 
Moreover, if $D$ is of multiplicative type ,i.e.,  $D^p=D$, $Z$ is reduced and then the number of isolated fixed points of $D$ is $K_X \cdot \Delta +\Delta^2 +c_2(X)$.
\end{proposition}

\begin{proof}

From Proposition~\ref{ex-seq-1} there is an exact sequence
\begin{gather}
0 \rightarrow \mathcal{O}_X(\Delta) \rightarrow T_X \rightarrow L \otimes I_Z \rightarrow 0.  \label{ex-seq-3}
\end{gather}
Since $Z$ has codimension 2 it follows that $L\cong \omega_X^{-1}\otimes \mathcal{O}_X(-\Delta)$. Moreover, from the  exact sequence
\[
0 \rightarrow L\otimes I_Z \rightarrow L \rightarrow \mathcal{O}_Z \rightarrow 0
\]
and the exact sequence~(\ref{ex-seq-3}) it follows that
\begin{gather*}
\mathrm{length}(\mathcal{O}_Z)=\chi(L)-\chi(L\otimes I_Z)=\chi(L)-\chi(T_X)+\chi(\mathcal{O}_X(\Delta))=\\
\chi(\mathcal{O}_X(\Delta))+\chi(\omega_X^{-1}\otimes \mathcal{O}_X(-\Delta))-\chi(T_X).
\end{gather*}
Be Riemann-Roch and Serre duality we get the following equalities.
\begin{gather*}
\chi(T_X)=2\chi(\mathcal{O}_X)+K_X^2-c_2(X),\\
\chi(\mathcal{O}_X(\Delta))=\chi(\mathcal{O}_X)+1/2(\Delta^2-\Delta \cdot K_X),\\
\chi(\omega_X^{-1}\otimes \mathcal{O}_X(-\Delta))= \chi(\omega_X^{2}\otimes \mathcal{O}_X(\Delta))=\\
\chi(\mathcal{O}_X)+1/2\left((\Delta+2K_X)^2-(\Delta+2K_X)\cdot K_X\right)=\\
\chi(\mathcal{O}_X)+1/2(2K_X^2+3K_X\cdot \Delta+\Delta^2).
\end{gather*}
Therefore from the above equations it follows that
\begin{gather*}
\mathrm{length}(\mathcal{O}_Z)=\chi(\mathcal{O}_X(\Delta))+\chi(\omega_X^{-1}\otimes \mathcal{O}_X(-\Delta))-\chi(T_X)=\\
2\chi(\mathcal{O}_X)+K_X^2+K_X\cdot \Delta +\Delta^2-\chi(T_X)=K_X\cdot \Delta +\Delta^2 +c_2(X),
\end{gather*}
as claimed. Suppose $D$ is of multiplicative type. Then by Lemma~\ref{ex-seq-1}, the embeded part $Z$ of the fixed locus of $D$ is reduced and its length is the same as the number of isolated singular points of $D$. Therefore the number of isolated singular points of $D$ is $K_X\cdot \Delta +\Delta^2 +c_2(X)$.
\end{proof}

\begin{corollary}\label{no-of-sing-of-quot}
Let $X$ be a smooth surface defined over an algebraically closed field $k$ of characteristic $p>0$ with a $\mu_p$ action. Let $\Delta$ be the divisorial part of the fixed locus of the action. Let $\pi \colon X \rightarrow Y$ be the quotient. Then $Y$ has exactly $K_X \cdot \Delta +\Delta^2 +c_2(X)$ singular points.
\end{corollary}
\begin{proof}
This follows immediately from Proposition~\ref{no-of-sing-of-quot}. Indeed, by Proposition~\ref{prop3}, the singular locus of $Y$ is exactly the set theoretic image of the embeded part of the fixed locus of the $\mu_p$ action. However, by Proposition~\ref{no-of-sing-of-quot} the number of isolated fixed points of $D$ is $K_X \cdot \Delta +\Delta^2 +c_2(X)$ and therefore this is the number of singular points of $Y$.
\end{proof}

Proposition~\ref{size-of-sing} suggests that $K_X \cdot \Delta$ is closely related with the size of the isolated singularities of $D$ and hence of the singular locus of the quotient $Y$. The next proposition shows that it decreases after blowing up a singular point of $D$.

\begin{proposition}\label{K-decreases}
Let $X$ be a smooth surface defined over a field of characteristic $p>0$. Let $D$ be a nonzero global vector field on $X$. Let $\Delta$ be its divisorial part. Let $f \colon X^{\prime} \rightarrow X$ be the blow up of an isolated singular point of $D$, $D^{\prime}$ the lifting of $D$ in $X^{\prime}$ and $\Delta^{\prime}$ its divisorial part. Then
\[
K_{X^{\prime}}\cdot \Delta^{\prime} \leq K_X \cdot \Delta.
\]
\end{proposition} 
\begin{proof}
Let $E$ be the $f$-exceptional curve. I will show that $\Delta^{\prime}=f^{\ast}\Delta +aE$, $a\geq 0$. Then \[
K_{X^{\prime}}\cdot \Delta^{\prime}=K_X\cdot \Delta -a \leq K_X\cdot \Delta.
\]
The proof of the previous claim will be by a direct local calculation of $D^{\prime}$. In suitable local coordinates at an isolated singular point of $D$, $\mathcal{O}_X=k[x,y]$ and $D$ is given by $D=h \left( f\partial/\partial x +g \partial /\partial y\right)$, where $f,g$ have no common factor. Locally at the standard open affine covers, the blow up is given by $\phi \colon k[x,y] \rightarrow k[s,t]$, $\phi(x)=s$, $\phi(y)=st$. Then it is easy to see that 
\[
D^{\prime}=h(s,st)\left(f(s,st)\frac{\partial}{\partial x}+\frac{1}{s}\left(tf(s,st)+g(s,st)\right)\frac{\partial}{\partial y}\right).
\]
It is now clear that $\Delta^{\prime}=f^{\ast}\Delta +aE$, $a\geq 0$.
\end{proof}

\begin{corollary}\label{sec3-cor-1}
Let $D$ be a nonzero global vector field of either multiplicative or additive type on a smooth surface $X$ defined over a field of characteristic 2. Let 
\[
\xymatrix{
X^{\prime}\ar[r]^f \ar[d]^{\pi^{\prime}} & X \ar[d]^{\pi} \\
Y^{\prime} \ar[r]^g & Y \\
}
\]
be the resolution of singularities of $D$ as in Proposition~\ref{prop4}. Let $\Delta$ be the divisorial part of $D$ and $\Delta^{\prime}$ the divisorial part of the lifting $D^{\prime}$ of $D$ on $X^{\prime}$. Then 
\begin{enumerate}
\item \[
K_{X^{\prime}}\cdot \Delta^{\prime} \leq K_X \cdot \Delta.
\]
\item
\[
K_{X^{\prime}}\cdot \Delta^{\prime}=4\left(\chi(\mathcal{O}_{X^{\prime}})-2\chi(\mathcal{O}_{Y^{\prime}})\right).
\]
\end{enumerate}
\begin{proof}
The proof of the first statement follows immediately from Proposition~\ref{K-decreases} since $f$ is a composition of blow ups of isolated singular points of $D$.

For the proof of the second statement, recall From Proposition~\ref{prop4} that $Y^{\prime}$ is the quotient of $X^{\prime}$ by the lifting $D^{\prime}$ of $D$ on $X^{\prime}$. Then by adjunction for purely inseperable maps~\cite{R-S76}, 
\begin{gather}\label{sec3-eq-10}
K_{X^{\prime}}=(\pi^{\prime})^{\ast}K_{Y^{\prime}}+\Delta^{\prime}.
\end{gather} 
Moreover, from Proposition~\ref{structure-of-pi} it follows (since $Y^{\prime}$ is smooth) that $\pi^{\prime} \colon X^{\prime} \rightarrow Y^{\prime}$ is a torsor. In particular $\pi_{\ast}\mathcal{O}_{X^{\prime}}$ fits in an exact sequence
\begin{gather}\label{sec3-eq-11}
0 \rightarrow \mathcal{O}_{Y^{\prime}} \rightarrow \pi_{\ast}\mathcal{O}_{X^{\prime}} \rightarrow M^{-1} \rightarrow 0,
\end{gather}
where $M=\mathcal{O}_{Y^{\prime}}(C^{\prime})$ is an invertible sheaf on $Y^{\prime}$. If the sequence splits then $D^2=D$ and if it doesn't split then $D^2=0$. Moreover, $K_{X^{\prime}}=(\pi^{\prime})^{\ast}(K_{Y^{\prime}}+C^{\prime})$. Therefore from~\ref{sec3-eq-10} we get that and $\Delta^{\prime}=(\pi^{\prime})^{\ast}C^{\prime}$. Then from~\ref{sec3-eq-11} we get that
\begin{gather}\label{sec3-eq-12}
\chi(M^{-1})=\chi(\pi^{\prime}_{\ast}\mathcal{O}_{X^{\prime}})-\chi(\mathcal{O}_{Y^{\prime}})=\chi(\mathcal{O}_{X^{\prime}})-\chi(\mathcal{O}_{Y^{\prime}}).
\end{gather}
From Riemann-Roch it follows that
\begin{gather}\label{sec3-eq-13}
\chi(M^{-1})=\chi(\mathcal{O}_{Y^{\prime}})+\frac{1}{2}((C^{\prime})^2+K_{Y^{\prime}}\cdot C^{\prime})=\\
\chi(\mathcal{O}_{Y^{\prime}})+\frac{1}{2}C^{\prime}\cdot (K_{Y^{\prime}}+C^{\prime})=\chi(\mathcal{O}_{Y^{\prime}})+\frac{1}{4}K_{X^{\prime}}\cdot \Delta^{\prime}
\end{gather}
Finally from~\ref{sec3-eq-12} amd ~\ref{sec3-eq-13} it follows that
\[
K_{X^{\prime}}\cdot \Delta^{\prime}=4\left(\chi(\mathcal{O}_{X^{\prime}})-2\chi(\mathcal{O}_{Y^{\prime}})\right),
\]
as claimed.
\end{proof}

\end{corollary}

%% file: sec4.tex
\section{Strategy of the proof of Theorem~\ref{main-theorem}.}\label{preparation}

Suppose that $\mathrm{Aut}(X)$ is not smooth. Then by Proposition~\ref{prop1}, $X$ admits a nontrivial global vector field $D$ of either additive or multiplicative type which by Proposition~\ref{prop2} induces a nontrivial $\alpha_p$ or $\mu_p$ action on $X$. Let $\pi \colon X \rightarrow Y$ be the quotient. By Proposition~\ref{prop4}, $Y$ is normal, $K_Y$ is $\mathbb{Q}$-Cartier and the local class groups of its singular points are p-torsion. Moreover, there is a  commutative diagram
\begin{equation}\label{sec-all-p-diagram-1}
\xymatrix{
 & X^{\prime}\ar[r]^f \ar[d]^{\pi^{\prime}} & X \ar[d]^{\pi} \\
Z & Y^{\prime} \ar[l]_{\phi}\ar[r]^g & Y \\
}
\end{equation}
such that $g \colon Y^{\prime} \rightarrow Y$ is the minimal resolution of $Y$,  $\phi \colon Y^{\prime} \rightarrow  Z$ its minimal model, $f$ is birational, $D$ lifts to a vector field $D^{\prime}$ on $X^{\prime}$ and $Y^{\prime} $ is the quotient of $X^{\prime}$ by the corresponding $\alpha_p$ or $\mu_p$ action. 

Let $E_i$, $i=1,\ldots, n$ be the $f$-exceptional curves and $F_i=\pi^{\prime}(E_i)$. Then $F_i$, $i=1,\ldots, n$, are exactly the $g$-exceptional curves. Let also $B_j$, $j=1,\ldots,m$ be the $\phi$-exceptional curves. Taking into consideration  that $K_Y$ has index either 1 or $p$ and $g \colon Y^{\prime}\rightarrow Y$ is the minimal resolution of $Y$, we get the following adjunction formulas
\begin{gather}\label{sec-all-p-eq-2}
K_{Y^{\prime}}=g^{\ast}K_Y-\frac{1}{p}F,\\
K_{Y^{\prime}}=\phi^{\ast}K_Z+B,\nonumber
\end{gather}
where $F=\sum_{i=1}^na_iF_i$, $a_i \in\mathbb{Z}_{\geq 0}$, and $B=\sum_{j=1}^mb_jB_j$, $b_j>0$, $j=1,\ldots m$. Moreover since both $Y^{\prime}$ and $Z$ are smooth, $\phi$ is the composition of $m$ blow ups.

Let $\Delta$ be the divisorial part of $D$. Then by adjunction for purely inseparable maps~\cite{R-S76}, 
\begin{gather}\label{sec-all-p-eq-1}
K_X=\pi^{\ast}K_Y+(p-1)\Delta.
\end{gather}

Note that it is possible that $\Delta=0$. For example, if $p\not= 2$  then  the homogeneous vector field $D=(y+z)\partial/\partial x+(x+z)\partial/\partial y+(x+y)\partial/\partial z$ of $k[x,y,z]$ gives a vector field on $\mathbb{P}^2$ with no divisorial part.

As a general strategy, cases with respect to the Kodaira dimension $\kappa(Y^{\prime})$  of $Y^{\prime}$ will be considered. Then the classification of surfaces in positive characteristic~\cite{BM76},~\cite{BM77},~\cite{SB91} will be heavily used in order to get information about $Y^{\prime}$ and then for $X$ by means of the diagram~\ref{sec-all-p-diagram-1}. Moreover, since $\pi$ is a purely inseparable map, it induces an equivalence between the \'etale sites of $X$ and $Y$. Therefore $X$ and $Y$ have the same algebraic fundamental group, $l$-adic betti numbers and \'etale euler characteristic. Then by using the fact that $g$ and $\phi$ are  birational it will be possible to calculate the algebraic fundamental group, $l$-adic betti numbers and \'etale euler characteristic of $X$ from those of $Z$.

The cases $p=2$ and $p\not= 2$ will be treated separately. The case $p=2$ has certain peculiarities and it requires special attention. In some sense this is expected since 2 is the smallest nonzero characteristic and many special situations appear in this case (as for example the existence of quasi-elliptic fibrations). The difficulties of this case will become evident during the proof of the case $p \geq 3$ in Section~\ref{p>2}. The case when $p\geq 3$ will be treated in section~\ref{p>2} and the case $p =2$ will be treated in sections~\ref{sec-4} and~\ref{sec-5}.

%% file: sec5.tex
\section{Vector fields on surfaces in characteristic $p\geq 3$.}\label{p>2}
The purpose of this section is to prove Theorem~\ref{main-theorem} and Corollary~\ref{main-corollary} in characteristic $p\geq 3$.  Their statements are a direct consequence of  Theorems~\ref{sec-all-p-prop-1},~\ref{sec-all-p-prop-2} and~\ref{sec-all-p-prop-3} of this section. 

I will only do the case where $K_X^2=1$ in detail. The proof of the case $K_X^2=2$ is similar. I will sketch its main points, remark on any differences with the case $K_X^2=1$ but I will leave the details to the reader. The method is exactly the same but certain calculations are lengthier and I see no reason to make an already long paper longer.

For the remaining part of this section, fix notation as in section~\ref{preparation}.

As was mentioned in Section~\ref{preparation}, the divisorial part $\Delta$ of $D$ may or may not be zero.  These two cases  behave quite differently and for this reason they will be considered separately. 

\textbf{Case 1:  $\Delta \not= 0$.}

In this case the following holds.

\begin{theorem}\label{sec-all-p-prop-1}
Let $X$ be a smooth canonically polarized surface defined over an algebraically closed field of characteristic $p\geq 3$. Suppose that $K_X^2<p$ and $X$ admits a nontrivial global vector field $D$ such that $\Delta \not= 0$. Then $X$ is uniruled. Moreover, if $5c_1^2<c_2$, then $X$ is unirational and $\pi^{et}_1(X)=\{1\}$ (in particular, this happens if  $K_X^2=1$ or $K_X^2=2$ and $\chi(\mathcal{O}_X)\geq 2$).
\end{theorem}

\begin{proof}

From equation~\ref{sec-all-p-eq-1} it follows that
\begin{gather}\label{sec-all-p-eq-3}
K_X^2=K_X \cdot \pi^{\ast}K_Y+(p-1)\Delta\cdot K_X.
\end{gather}

\textit{Case 1.} Suppose that $K_X \cdot \pi^{\ast}K_Y <0$. Then $\kappa(Z)=-\infty$ because if not then $|nK_Z|\not=\emptyset$ for $n>>0$ and hence $|nK_{Y^{\prime}}|\not= \emptyset$, and also  $|nK_Y|\not=\emptyset$, for $n>>0$. But then, since $K_X$ is ample, $K_X\cdot \pi^{\ast}K_Y \geq 0$, which is impossible. Therefore $Z$ is uniruled and hence $Y$ is also uniruled. Let $F_X\colon X \rightarrow X^{(p)}$ be the relative Frobenious. Then there is a factorization
\[
\xymatrix{
  &   Y \ar[dr]^{\delta} & \\
  X \ar[ur]^{\pi}\ar[rr]^{F_X} &   &X^{(p)}
  }
  \]  
Hence $X^{(p)}$ is purely inseparably uniruled and so $X$ is too.

\textit{Case 2.} Suppose that $K_X \cdot \pi^{\ast}K_Y >0 $. Then since $K_X$ is ample and $\Delta$ an effective divisor, it follows from equation~\ref{sec-all-p-eq-3} that $K_X^2 \geq p$.

\textit{Case 3.} Suppose that $K_X \cdot \pi^{\ast}K_Y =0 $. In this case equation ~\ref{sec-all-p-eq-3} only says that $K_X^2 \geq p-1$. In particular for $p=2$ it provides no information at all.  

Consider now cases with respect to the Kodaira dimension $\kappa(Z)$ of $Z$. 

\textit{Case 3.1.} Suppose that $\kappa(Z)=-\infty$. Then $Z$ is uniruled and exactly as in case 1, it follows that $X$ is uniruled as well.

\textit{Case 3.2.} Suppose that $\kappa(Z)=0$. In this case I will show that if $p\geq 3$, then $K_X^2\geq p$.

If $Y^{\prime}=Z$, i.e, $Y^{\prime}$ is a minimal surface itself, then $K_{Y^{\prime}}\equiv 0$~\cite{BM76},~\cite{BM77} and therefore $K_Y\equiv 0$. Hence $K_X\equiv (p-1)\Delta$ and consequently
\begin{gather}\label{sec-all-p-eq-4}
K_X^2=(p-1)^2\Delta^2\geq (p-1)^2.
\end{gather}
This is $\geq p$ if $p\geq 3$. 

Suppose now that $Y^{\prime}$ is not a minimal surface. Then $\phi$ is not trivial and is a composition of blow ups. Moreover the usual adjuction formula gives that
\begin{gather}\label{sec-all-p-eq-5}
K_{Y^{\prime}}=\phi^{\ast}K_Z+\sum_{i=1}^mb_iB_i,
\end{gather}
where $B_i$, $i=1,\ldots,m$ are the $\phi$-exceptional curves and $b_i>0$ for all $i$. Now since $\kappa(Z)=0$ it follows that $K_Z\equiv 0$~\cite{BM76},~\cite{BM77}. Hence $K_{Y^{\prime}}\equiv \sum_{i=1}^mb_iB_i$. Since $\phi$ is a composition of blow ups, it follows that $B=\sum_{i=1}^mb_iB_i$ contains $\phi$-exceptional curves with self intersection -1. However, since $Y^{\prime}$ is the minimal resolution of $Y$, $g$ does not contract any curve with self intersection -1. Therefore $g_{\ast}\left(\sum_{i=1}^mb_iB_i\right) \not= 0$. Then from equations~\ref{sec-all-p-eq-2},~\ref{sec-all-p-eq-5} it follows that
\[
K_Y\equiv g_{\ast}K_{Y^{\prime}}\equiv g_{\ast}\left(\sum_{i=1}^mb_iB_i\right)
\]
Hence $K_Y$ is numerically equivalent to an effective divisor. Hence $K_X\cdot \pi^{\ast}K_Y>0$ and hence $K_X^2\geq p$. This together with equation~\ref{sec-all-p-eq-4} show that if $\kappa(Z)=0$ and $p\geq 3$, then $K_X^2\geq p$. 

\textit{Case 4.} Suppose that $\kappa(Z)=1$ or $\kappa(Z)=2$. I will show that in both cases $K_{Y^{\prime}}$ is linearly equivalent to an effective divisor $B$  with rational coefficients  which has at least one irreducible component that is not contracted by $g$. Then $K_Y$ is linearly equivalent to a nonzero effective divisor. Hence $K_X\cdot \pi^{\ast}K_Y >0$ and therefore it follows from equation~\ref{sec-all-p-eq-3} that $K_X^2\geq p$.

Suppose that $\kappa(Z)=2$. Then $nK_Z$ is nef and big for $n>>0$~\cite{BM76},~\cite{BM77}. hence $|nK_Z|$ contains an element $W\not \subset \phi_{\ast}E$.  This means that $\phi^{\ast}W $ is not contained in the $g$-exceptional set and hence it has at least one irreducible component which is not contracted by $g$.  By adjunction for $\phi$,
\[
K_{Y^{\prime}}\equiv\phi^{\ast}K_Z+F\equiv \frac{1}{n}\phi^{\ast} W + F,
\]
where $F$ is a $\phi$-exceptional divisor. Hence $K_{Y^{\prime}}\equiv B$, where  $B=\frac{1}{n}\phi^{\ast}W+F$ is an effective divisor such that it has at least one irreducible component which is not contracted by $g$.  
Then from equation~\ref{sec-all-p-eq-2} and considering that $g_{\ast}\phi^{\ast}W\not= 0$, it follows that $K_Y$ is linearly equivalent to an effective divisor with rational coefficients, as claimed.

Suppose that $\kappa(Z)=1$. Then $Z$ admits an elliptic or quasi-elliptic fibration $h \colon Z \rightarrow C$ to a smooth curve $C$~\cite{BM76},~\cite{BM77}. Moreover, for $n>>0$, $nK_Z=h^{\ast}W$, where $W$ is an effective divisor on $C$~\cite{BM76},~\cite{BM77}. Moreover, $W$ can be chosen so that $h^{\ast}W \not\subset \phi_{\ast}E$. Now the argument used in the case when $\kappa(Z)=2$ shows that $K_{Y^{\prime}}$ and hence $K_Y$ are numerically equivalent to a nonzero effective divisor with rational coefficients. Therefore in this case too $K_X\cdot \pi^{\ast}K_Y>0$ and hence $K_X^2\geq p$.

Hence I have shown that if $X$ has a nontrivial global vector field, then either $\kappa(Z)=-\infty$ and hence $X$ is uniruled, or $K_X^2\geq p$.  Suppose that $5c_1^2<c_2$. Then according to Lemma~\ref{b1} below, $b_1(X)=0$. Then $b_1(X^{\prime})=0$ and hence since $\pi^{\prime}$ is an \'etale equivalence, $b_1(Z)=b_1(Y^{\prime})=0$. Therefore in this case, either $K_X^2\geq p$ or $Z$ is unirational and $\pi_1^{et}(Z)=\{1\}$. Since the algebraic fundamental group is a birational invariant and also invariant under \'etale equivalence, it follows from diagram~\ref{sec-all-p-diagram-1} that $\pi_1^{et}(X)=\{1\}$.  Suppose in particular that  $K_X^2=1$ or $K_X^2=2$ and 
$\chi(\mathcal{O}_X)\geq 2$. Then from~\cite[Corollary 1.8]{Ek87} it follows that $1\leq \chi(\mathcal{O}_X) \leq 3$ in the first case, and $2\leq \chi(\mathcal{O}_X) \leq 4$ in the second case. Then from Lemma~\ref{b1} it follows that $b_1(X)=0$ and hence the previous argument shows that $X$ is unirational and simply conncted. This concludes the proof of the theorem.
\end{proof}

\begin{remark}
The proof of Theorem~\ref{sec-all-p-prop-1} shows some of the reasons why the case $p=2$ has to be excluded. As will be seen in the remaining part of this section, the existence of quasi-elliptic fibrations in characteristic 2 is another reason. This is to be expected in some sense since $p=2$ is the smallest possible characteristic where most pathologies appear.
\end{remark}

\begin{lemma}\label{b1}
Let $X$ be a smooth surface of general type defined over an algebraically closed field $k$. Suppose that $5c_1^2<c_2$. Then $b_1(X)=0$ and 
\[
|\pi_1^{et}(X)|\leq \frac{6}{2\chi(\mathcal{O}_X)-K_X^2}=\frac{36}{c_2-5c_1^2}.
\]
\end{lemma}

\begin{proof}
In the case when $K_X^2=1$ the proof  can be found in~\cite{Li09}. I proceed to prove the general case.

It is well known that for all but finitely many primes $l$,
\begin{gather}\label{sec4-eq9}
H^1_{et}(X,\mathbb{Z}/l\mathbb{Z})\cong\left(\mathbb{Z}/l\mathbb{Z}\right)^{b_1(X)}.
\end{gather}
The claim then that $b_1(X)=0$ will follow from the finiteness of $\pi_1^{et}(X)$. Let $f \colon Y \rightarrow X$ be an \'etale cover of degree $n$. Then $K_Y^2=nK_X^2$ and $\chi(\mathcal{O}_Y)=n\chi(\mathcal{O}_X)$. From Noether's inequality we get that
\[
nK_X^2=K_Y^2\geq 2\mathrm{p}_g(Y)-4=2(\chi(\mathcal{O}_Y)-1+h^1(\mathcal{O}_Y))-4\geq 2n\chi(\mathcal{O}_X)-6.
\]
Therefore,
\[
n(2\chi(\mathcal{O}_X)-K_X^2)\leq 6.
\]
Now by Noether's formula, $2\chi(\mathcal{O}_X)-K_X^2=\frac{1}{6}(c_2-c_1^2)>0$, by assumption. Hence
\begin{gather}\label{sec4-eq10}
n \leq \frac{6}{2\chi(\mathcal{O}_X)-K_X^2},
\end{gather}
and therefore, \[
|\pi_1^{et}(X)|\leq \frac{6}{2\chi(\mathcal{O}_X)-K_X^2}=\frac{36}{c_2-5c_1^2}.
\]
as claimed.

\end{proof} 

\textbf{Case 2:  $\Delta = 0$.}  In this case the following holds.

\begin{theorem}\label{sec-all-p-prop-2}
Let $X$ be a smooth canonically polarized surface defined over an algebraically closed field of characteristic $p\geq 3$. Suppose that $X$ admits a nontrivial global vector field $D$ with only isolated singularities. Then,
\begin{enumerate}
\item If $K_X^2=2$ and $p\not= 3,5$, then $X$ is uniruled. If in addition $\chi(\mathcal{O}_X)\geq 2$, then $X$ is unirational and simply connected.
\item Suppose that $K_X^2=1$ and $p\not= 7$. Then $X$ is unirational and simply connected except possibly if $p\in\{3,5\}$ and $X$ is a simply connected supersingular Godeaux surface.
\end{enumerate}
\end{theorem}

\begin{proof}
I will only do the case $K_X^2=1$ in detail. The case $K_X^2=2$ is exactly similar and its details are left to the reader.

For the remainder of the proof, fix notation as in diagram~\ref{sec-all-p-diagram-1}. Since $\Delta=0$, it follows from equation~\ref{sec-all-p-eq-1} that $K_X=\pi^{\ast}K_Y$ and therefore $K_X^2=pK_Y^2$. 

Consider now cases with respect to whether or not $Y^{\prime}$ is a minimal surface or not.

\textbf{Case 1.} Suppose that $Y^{\prime}=Z$, i.e., $Y^{\prime}$ is a minimal surface.

\textbf{Case 1.1} Suppose that $\kappa(Y^{\prime})=2$. In this case I will show that $K_X^2\geq p$.

From equation~\ref{sec-all-p-eq-2} it follows that
\[
K_X^2=pK_Y^2=pK_{Y^{\prime}}^2+K_{Y^{\prime}}\cdot F\geq p,
\] 
since $K_{Y^{\prime}}^2>0$ and $K_{Y^{\prime}}\cdot F\geq 0$.

\textbf{Case 1.2} Suppose that  $\kappa(Y^{\prime})=1$. In this case I will show the following.
\begin{enumerate}
\item If $K_X^2=1$ then $p=3$ and $X$ is a simply connected supersingular Godeaux surface. 
\item If $K_X^2=2$ then $p=3,5$. 
\end{enumerate}

Since $\kappa(Y^{\prime})=1$, $Y^{\prime}$ admits an elliptic or quasi elliptic (this only for $p=2,3$) fibration $\psi\colon Y^{\prime}\rightarrow B$~\cite{BM76},~\cite{BM77}.  In particular $K_{Y^{\prime}}^2=0$.

Let $F_i$, $i=1,\ldots,n$, be the $g$-exceptional curves. Suppose that $F_i^2=-2$, for all $i$. Then $g$ is crepant and therefore $K_Y$ is Cartier. Then 
\[
K_X^2=pK_Y^2\geq p.
\]
Suppose then that there exists at least one $g$-exceptional curve with self intersection $\leq -3$. After renumbering the exceptional curves, we may assume that $F_1^2=-d$, $d\geq 3$. Then from equation~\ref{sec-all-p-eq-1} it follows that
\[
K_{Y^{\prime}}+\frac{1}{p}\left(a_1F_1+\sum_{i=2}^na_iF_i\right)=g^{\ast}K_Y,
\]
where $a_i\in \mathbb{Z}_{\geq 0}$, for all $i=1,\ldots,n$, and $a_1>0$. Intersecting this equation with $E_1$ and $K_{Y^{\prime}}$ and considering that $K_{Y^{\prime}}^2=0$, we get that
\begin{gather}\label{sec-all-p-eq-5}
d-2+\frac{1}{p}\left(-a_1d+\sum_{i=2}^na_i(F_i\cdot F_1)\right)=0 \\
K_X^2=pK_Y^2=(d-2)a_1+\sum_{i=2}^na_i(F_i \cdot K_{Y^{\prime}})\nonumber
\end{gather}

Suppose that $K_X^2=1$. Then from equations~\ref{sec-all-p-eq-5} it follows that
\[
(d-2)a_1+\sum_{i=2}^na_i(F_i \cdot K_{Y^{\prime}})=1.
\]
Therefore, since $K_{Y^{\prime}}\cdot F_i \geq 0$, for all $i$, it follows that $d=3$, $a_1=1$ and $K_{Y^{\prime}}\cdot F_i = 0$, for $i=2,\ldots, n$, and hence $F_i^2=-2$, for $i=2,\ldots,n$. Then again from equations~\ref{sec-all-p-eq-5} it follows that
\[
1+\frac{1}{p}\left(-3+\sum_{i=2}^na_i(F_i\cdot F_1)\right)=0.
\]
Since $p \geq 3$ it now follows that $p=3$ and $F_i\cdot F_1=0$, for $i=2,\ldots,n$. Hence the only possibility if $K_X^2=1$ is that $p=3$ and the $g$-exceptional set consists of an isolated curve with self intersection $-3$ and chains of $(-2)$ curves. Hence $Y$ has a singularity of type $\frac{1}{3}(1,1)$ and isolated DuVal singularities. 

Suppose that $K_X^2=2$. Then arguing as in the case when $K_X^2=1$ it is not difficult to see that if $p\not=3,5$, the equations~\ref{sec-all-p-eq-5} have no solutions. Otherwise we get restrictions on the singularities of $Y$ as in the case when $K_X^2=1$. However the calculations with them are rather messy and at the moment I do not see how to deal with them. In principle one should be able to deal with them by following the method that I will use next for the case $p=3$ and $K_X^2=1$. In any case, if $p\geq 7$, then $K_X^2 \geq 3$, as claimed.

Suppose now that $K_X^2=1$. Then then according to the previous discussion, $p=3$ and  the singularities of $Y$ are one $\frac{1}{3}(1,1)$ singularity and isolated DuVal singularities. Note that the DuVal singularities must be of type $A_2$ because these are precisely the DuVal singularities whose local class groups are 3-torsion. 

As was mentioned earlier, $Y^{\prime}$ admits an elliptic or quasi-elliptic fibration $\psi \colon Y^{\prime} \rightarrow B$. Since $K_X^2=1$, it follows from Lemma~\ref{b1} that $b_1(X)=0$ and therefore $B=\mathbb{P}^1$. Then
\begin{gather}\label{sec-all-p-eq-6}
R^1\psi_{\ast}\mathcal{O}_{Y^{\prime}}=\mathcal{O}_{\mathbb{P}^1}(-d) \oplus T,
\end{gather}
where $d \in \mathbb{Z}$ and $T$ is a torsion sheaf on $\mathbb{P}^1$. By using the Grothendieck spectral sequence and Serre duality, we get that
\begin{gather}\label{sec-all-p-eq-7}
\chi(\mathcal{O}_{Y^{\prime}})=1+h^0(\omega_{Y^{\prime}})-h^0(R^1\psi_{\ast}\mathcal{O}_{Y^{\prime}}) \\
p_g(Y^{\prime})=h^0(\omega_{Y^{\prime}})=h^2(\mathcal{O}_{Y^{\prime}})=h^1(R^1\psi_{\ast}\mathcal{O}_{Y^{\prime}}).\nonumber
\end{gather}

Since $K_X=\pi^{\ast}K_Y$, it follows that 
\begin{gather}\label{eq-pg}
p_g(Y^{\prime})=p_g(Y)\leq p_g(X).
\end{gather} 
Moreover, since $K_X^2=1$, it follows from~\cite[Corollary 1.8]{Ek87} that $p_g(X)\leq 2$. Hence $0\leq p_g(Y^{\prime})\leq 2$.

\textbf{Case 1.2.1. } Suppose that $p_g(Y^{\prime})=0$. Then $\chi(\mathcal{O}_{Y^{\prime}})\leq 1$ and hence from Noether's formula it follows that $c_2(Y^{\prime})\leq 12$. But also 
\begin{gather}\label{sec-all-p-eq-8}
c_2(Y^{\prime}) =    \chi_{et}(Y^{\prime})        =     \chi_{et}(Y)+k=\chi_{et}(X)+k=c_2(X)+k,
\end{gather}
where $k$ is the number of $g$-exceptional curves. Now since $K_X^2=1$ it follows~\cite[Corollary 1.8]{Ek87} that $1\leq \chi(\mathcal{O}_X)\leq 3$ and therefore from Noether's formula it follows that $c_2(X)\geq 11$. Then from equation~\ref{sec-all-p-eq-8} and since $c_2(Y^{\prime})\leq 12$, it follows that $c_2(X)=11$, $k=1$ and $c_2(Y^{\prime})=12$. Hence $Y$ has exactly one singular point, which is necessarily of type $1/3(1,1)$. Let then $F$ be the unique $g$-exceptional curve. Then $F^2=-3$ and by adjunction  we get that
\[
K_{Y^{\prime}}=g^{\ast}K_Y-\frac{1}{3}F.
\]
Let $E$ be the unique $f$-exceptional curve. Then $\pi^{\prime}E=F$. Then again by adjunction, $K_{X^{\prime}}=f^{\ast}K_X+aE$, $a\in\mathbb{Z}_{>0}$. Then since $\Delta=0$, it follows that
\[
K_{X^{\prime}}=f^{\ast}K_X+aE=f^{\ast}\pi^{\ast}K_Y+aE=(\pi^{\prime})^{\ast}K_{Y^{\prime}}+aE+\frac{1}{3}(\pi^{\prime})^{\ast}F.
\]
Now $(\pi^{\prime})^{\ast}F=kE$, where $k=1$ if $E$ is not an integral curve for $D^{\prime}$ and $k=3$ if $E$ is an integral curve. Suppose that $k=1$. Then
\begin{gather}\label{sec-all-p-eq-9}
K_{X^{\prime}}=(\pi^{\prime})^{\ast}K_{Y^{\prime}}+(a+\frac{1}{3})E.
\end{gather}
But this is impossible because from the adjunction formula for purely inseparable maps~\ref{sec-all-p-eq-1}, $a+\frac{1}{3}=p-1=2$, which is an integer.  Hence $k=3$ , $a=1$ and $E$ is an integral curve for $D^{\prime}$. Then since 
$(\pi^{\prime})^{\ast}F=3E$ it follows that $9E^2=3F^2=-9$ and hence $E^2=-1$. It is now easy to see that $f$ is nothing but the blow up of a point of $X$. In particular $X^{\prime}$ is smooth. Recall also that $Y^{\prime}$ is the quotient of $X^{\prime}$ by $D^{\prime}$. Moreover from equation~\ref{sec-all-p-eq-9} it follows that the divisorial part $\Delta^{\prime}$ of $D^{\prime}$ is $E$. Then by Proposition~\ref{size-of-sing} ,
\[
\mathrm{length}(\mathcal{O}_Z)=K_{X^{\prime}} \cdot  E +E^2 +c_2(X^{\prime})=-2+c_2(X)+1=-2+11+1=10>0,
\] 
where $Z$ is the embeded part of the fixed locus of $D^{\prime}$. However, since both $X^{\prime}$ and $Y^{\prime}$ are smooth, from Proposition~\ref{prop3} it follows that the fixed locus of $D^{\prime}$ does not have an embeded part. Hence we get a contradiction. Hence the case $p_g(Y^{\prime})=0$ is impossible.

\textbf{Case 1.2.2. } Suppose that $p_g(Y^{\prime})=1$.  I will show that in this case, $X$ must be a simply connected supersingular Godeaux surface.

From~\ref{sec-all-p-eq-7}  it follows that $h^1(R^1\psi_{\ast}\mathcal{O}_{Y^{\prime}})=1$. Hence $R^1\psi_{\ast}\mathcal{O}_{Y^{\prime}}=\mathcal{O}_{\mathbb{P}^1}(-2)\oplus T$, where $T$ is a torsion sheaf on $\mathbb{P}^1$. Suppose that $T \not= 0$. Then from equations~\ref{sec-all-p-eq-7} it follows that $\chi(\mathcal{O}_{Y^{\prime}})\leq 1$. In this case the argument that was used in the case $p_g(X)=0$ applies giving again a contradiction. Hence $T=0$ and therefore $\chi(\mathcal{O}_{Y^{\prime}})=2$ and $H^1(\mathcal{O}_{Y^{\prime}})=0$.  Moreover, from equation~\ref{eq-pg} and the discussion immediately after it, it follows that  $1\leq p_g(X)\leq 2$. 

Suppose that $p_g(X)=2$. Then also from~\cite[Corollary 1.8]{Ek87} it follows that $h^1(\mathcal{O}_X)\leq 1$ and hence $2\leq \chi(\mathcal{O}_X)\leq 3$. Then 
\begin{gather}\label{sec-all-p-eq-10}
c_2(Y^{\prime})=\chi_{et}(Y)+k=c_2(X)+k=12\chi(\mathcal{O}_X)-1+k,
\end{gather}
where $k$ is the number of $g$-exceptional curves. But since $\chi(\mathcal{O}_{Y^{\prime}})=2$, it follows that $c_2(Y^{\prime})=24$. Hence $12\chi(\mathcal{O}_X)-1+k=24$. Suppose that $\chi(\mathcal{O}_X)=3$. Then $24=35+k$, which is impossible. Suppose that $\chi(\mathcal{O}_X)=2$. Then $24-1+k=24$ and therefore $k=1$. Hence there exists only one $g$-exceptional curve. However I have shown during the study of the case $p_g(Y^{\prime})=0$ that this case is also impossible.

Suppose that $p_g(X)=1$. Then from~\cite[Corollary 1.8]{Ek87} it follows that $h^1(\mathcal{O}_X)\leq 1$.  Suppose that $h^1(\mathcal{O}_X)=0$. Then $\chi(\mathcal{O}_X)=2$. Then $c_2(X)=23$. Hence, since $c_2(Y^{\prime})=24$, it follows from the equation~\ref{sec-all-p-eq-10} that $g$ has exactly one exceptional curve and hence $Y$ has exactly one singularity which is necessarily of type $1/3(1,1)$. But in the study of the case when $p_g(Y^{\prime})=0$, I have shown that this case is impossible. Hence the only possibility is that $h^1(\mathcal{O}_X)=1$ and hence $\chi(\mathcal{O}_X)=1$. Therefore $X$ is a Godeaux surface. I will show that it is a supersingular and simply connected Godeaux. Since $h^1(\mathcal{O}_X)=1$, $X$ can be either a singular or supersingular Godeaux surface. Let $F^{\ast} \colon H^1(\mathcal{O}_X)\rightarrow H^1(\mathcal{O}_X)$ be the map induced by the Frobenious. From the discussion in section~\ref{sec-0}, $X$ is singular if $F^{\ast}$ is bijective and supersingular if it is zero. Suppose that it is bijective. Then from the exact sequence in the \'etale topology
\begin{gather}\label{torsor-seq}
0 \rightarrow \mathbb{Z}/3\mathbb{Z} \rightarrow \mathcal{O}_X \stackrel{F^{\ast}-id}{\rightarrow} \mathcal{O}_X \rightarrow 0,
\end{gather}
it follows that $H^1_{et}(X,\mathbb{Z}/3\mathbb{Z})\not=0$ and therefore $X$ has nontrivial \'etale $\mathbb{Z}/3\mathbb{Z}$ covers. But since $\pi$ is an \'etale equivalence, it follows that $ H^1_{et}(X,\mathbb{Z}/3\mathbb{Z})=
H^1_{et}(Y,\mathbb{Z}/3\mathbb{Z})$. But since $Y$ has rational singularities and $H^1(\mathcal{O}_{Y^{\prime}})=0$, it follows that $H^1(\mathcal{O}_Y)=0$. But then the corresponding sequence~\ref{torsor-seq} for $Y$ shows that $H^1_{et}(Y,\mathbb{Z}/3\mathbb{Z})=0$, a contradiction. Hence $X$ is a supersingular Godeaux surface. It remains to show that it is simply connected. Since the fundamental group is invariant under \'etale and birational equivalence, it follows that
\[
\pi_1^{et}(X)=\pi_1^{et}(Y)=\pi_1^{et}(Y^{\prime}).
\]
Therefore it suffices to show that $Y^{\prime}$ is simply connected. The first step in order to show this is to study more carefully the structure of the elliptic fibration $\psi \colon Y^{\prime}\rightarrow \mathbb{P}^1$. I will show that 
$\psi$ has exactly one multiple fiber. More precisely, there exists a unique point $t_0 \in \mathbb{P}^1$ such that the fiber $Y^{\prime}_{t_0}$ is not reduced. Moreover $Y^{\prime}_{t_0}=2C$, where $C$ is an indecomposable curve of canonical type. 

From the adjunction formula for elliptic and quasi-elliptic fibrations~\cite[Theorem 7.15]{Ba01} we get that 
\begin{gather}\label{sec-all-p-eq-11}
K_{Y^{\prime}}=\sum_{i=1}^k (m_i-1)P_i,
\end{gather}
where $m_iP_i=Y^{\prime}_{t_i}$ are the multiple fibers of $\psi$, $m_i\geq 2$ (note that since $T=0$, $\psi$ has no exceptional fibers). By assumption, $g$ has exactly one exceptional curve $E$ such that $E^2=-3$. This corresponds to the unique singularity of $Y$ that is of type $1/3(1,1)$. Then since $K_{Y^{\prime}}\cdot E=1$, it follows from equation~\ref{sec-all-p-eq-11} that $\sum_{i=1}^k (m_i-1)P_i\cdot E =1$. Taking into consideration that 
$E$ is not contained in any fiber of $\psi$ (if it did then $K_{Y^{\prime}}\cdot E=0$) it follows that $P_i \cdot E\not=0$, for all $i=1,\ldots, k$. Hence $k=1$ and $m_1=1$. This means that $\psi$ has exactly one multiple fiber $Y^{\prime}_{t_0}$ and moreover, $Y^{\prime}_{t_0}=2C$, where $C$ is an indecomposable curve of canonical type. Then from equation~\ref{sec-all-p-eq-11} it follows that $K_{Y^{\prime}}=C$ and hence $Y^{\prime}_t\sim 2K_{Y^{\prime}}$, for all $t\in \mathbb{P}^1$.

I will next show that $\pi_1^{et}(Y^{\prime})=\{1\}$. In order to do this it suffices to show that $Y^{\prime}$ does not have non trivial \'etale $\mathbb{Z}/n\mathbb{Z}$-covers. For $n\not= 3$, this is equivalent to the property that $Y^{\prime}$ does not have torsion line bundles of order $n$. For $n=3$, this is equivalent to the property  that $\mu_3$ is not a subgroup scheme of $\mathrm{Pic}(Y^{\prime})$.  $\mathbb{Z}/3\mathbb{Z}$ \'etale covers are classified by $H^1_{et}(Y^{\prime},\mathbb{Z}/3\mathbb{Z})$. However, since 
$H^1(\mathcal{O}_{Y^{\prime}})=0$, it follows that $H^1_{et}(Y^{\prime},\mathbb{Z}/3\mathbb{Z})=0$. Therefore $Y^{\prime}$ does not have non trivial \'etale $\mathbb{Z}/3\mathbb{Z}$-covers.

Next I will show that $\mathrm{Pic}(Y^{\prime})$ is torsion free.  Let $L$ be a torsion line bundle on $Y^{\prime}$ of order $n$. Then by the Riemann-Roch theorem, $\chi(L)=\chi(\mathcal{O}_{Y^{\prime}})=2>0$. Hence, either $H^0(L)\not=0$ or $H^2(L)=H^0(L^{-1}\otimes \omega_{Y^{\prime}})\not=0$. Since $L$ is torsion, $H^0(L)=0$. Therefore, $H^0(L^{-1}\otimes \omega_{Y^{\prime}})\not=0$. Then 
\begin{gather}\label{sec-all-p-eq-11-1}
L^{-1}\otimes \omega_{Y^{\prime}}\cong \mathcal{O}_{Y^{\prime}}(W),
\end{gather}
where $W$ is an effective divisor. I will show that all irreducible components of $W$ are contracted to points by $\psi$. Suppose on the contrary that $W$ has a component $W_0$ that dominates $\mathbb{P}^1$. Then since $\omega_{Y^{\prime}}=\mathcal{O}_{Y^{\prime}}(C)$ and $2C=Y^{\prime}_{t_0}\sim Y^{\prime}_t$, for any $t\in \mathbb{P}^1$, it follows that $C^2=0$ and that $C\cdot W_0\not= 0$. But this is impossible since $L^{-1}\otimes \omega_{Y^{\prime}}\cong \mathcal{O}_{Y^{\prime}}(W)$, and $L$ is torsion. Hence every irreducible component of $W$ is contracted by $\psi$. Now decompose $W$ as
\begin{gather}\label{sec-all-p-eq-12}
W=W_1+\cdots+W_k,
\end{gather}
where $W_k \subset Y^{\prime}_{t_k}$, $t_k\in \mathbb{P}^1$ and $t_i\not= t_j$, for $i\not= j$. I will show that $k=1$ and $(W_1)_{red}=(Y^{\prime}_{t_1})_{red}$. Suppose that there exists an $i$ such that $(W_i)_{red}\not=(Y^{\prime}_{t_i})_{red}$. Then it is possible to write $(Y^{\prime}_{t_i})_{red}=(W_i)_{red}+Z$, where $Z$ and $W_i$ do not have any common components. Then since $Y^{\prime}_{t_i}$ is connected it follows that $Z\cdot W_i \not= 0$. But then from equation~\ref{sec-all-p-eq-12} it follows that 
\[
0=Y^{\prime}_{t_0} \cdot Z= 2W \cdot Z=2W_i\cdot Z\not=0,
\]
a contradiction. Hence for all $i$, $(W_i)_{red}=(Y^{\prime}_{t_i})_{red}$. Since $E$ dominates $\mathbb{P}^1$ it follows that $E\cdot W_i \not=0$, for all $i$. But $C\cdot E=1$ and hence $W\cdot E=1$. Then equation~\ref{sec-all-p-eq-12} shows that $k=1$. Hence $W$ is contained in a single fiber $Y^{\prime}_t$ and $W_{red}=(Y^{\prime}_t)_{red}$. 

I will now consider cases with respect to the nature of $Y^{\prime}_t$.

Suppose that $Y^{\prime}_t$ is reduced. Then since $W_{red}=(Y^{\prime}_t)_{red}$, it follows that $W=Y^{\prime}_t$ and hence $C\equiv Y^{\prime}_t$. Therefore 
\[
1=C\cdot E=Y^{\prime}_t\cdot E=2K_{Y^{\prime}} \cdot E=2,
\]
which is impossible.

Suppose that $Y^{\prime}_t$ is not reduced.  

If the fiber $Y^{\prime}_t$ is a multiple fiber, then since $\psi$ has exactly one multiple fiber, it follows that $t=t_0$ and that $W=C$ or $W=2C$. If $W=2C$, then from~\ref{sec-all-p-eq-11-1} it follows that $C\equiv 0$, which is impossible. If $W=C$, then $L\cong \mathcal{O}_{Y^{\prime}}$.

Suppose that $Y^{\prime}_t$ is not a multiple fiber. In this case every irreducible component of $Y^{\prime}_{red}$ is a smooth rational curve of self intersection  $-2$~\cite[Corollary 5.1.1]{CD89}. From the relation 
$E\cdot Y^{\prime}_t=2$, it follows that $E$ meets at most two irreducible components of $Y^{\prime}_t$. Moreover, $2K_{Y^{\prime}}=Y^{\prime}_t$. Then from the adjuction formula for $g$,
\[
K_{Y^{\prime}}+\frac{1}{3}E=g^{\ast}K_Y,
\]
and considering that every irreducible component of $Y^{\prime}_t$ is a smooth rational curve of self intersection $-2$, it follows that every irreducible component of $Y^{\prime}_t$ that does not intersect $E$ must be contracted by $g$. However, the singular locus of $Y$ consists of a single $1/3(1,1)$ singularity and finitely many $A_2$ type canonical singularities. Hence $Y^{\prime}_t$ minus the components that meet $E$ is a disjoint union of chains of rational curves of length 2. By taking now into consideration this, the fact that $Y^{\prime}_t$ is not multiple or reduced and the classification of reducible fibers of an elliptic fibration~\cite[Page 288]{CD89}, we see that there are the following two possibilities for $Y^{\prime}_t$:
\begin{enumerate}
\item Suppose that  $E$ meets exactly one component of $Y^{\prime}_t$.  Then $Y^{\prime}_t$ is a configuration of rational curves of type $\tilde{E}_6$,
\[
\xymatrix{
\overset{1}{\circ} \ar@{-}[r] &\overset{2}{\circ} \ar@{-}[r] &\overset{3}{\bullet} \ar@{-}[r] \ar@{-}[d] & \overset{2}{\circ} \ar@{-}[r]     &\overset{1}{\circ}  \\
                                           &                                           &\overset{2}{\circ}\ar@{-}[r]      &       \overset{1}{\circ}              &                         
 }
\]
where the solid dot corresponds to the curve that meets $E$ and the numbers over each dot indicate the multiplicity with which the corresponding  curve appears in $Y^{\prime}_t$. But from the diagram above we see that $E$ meets a curve which has multiplicity $3$ in $Y^{\prime}_t$. Hence $Y^{\prime}_t \cdot E\geq 3$, which is impossible since $Y^{\prime}_t \cdot E=2$. Hence this case is impossible. 
\item Suppose that  $E$ meets two component of $Y^{\prime}_t$. Then $Y^{\prime}_t$ is a configuration of rational curves of types $\tilde{E}_7$,
\[
\xymatrix{
\overset{1}{\circ} \ar@{-}[r] &\overset{2}{\circ} \ar@{-}[r] &\overset{3}{\bullet} \ar@{-}[r]  & \overset{4}{\circ} \ar@{-}[r]\ar@{-}[d]& \overset{3}{\bullet} \ar@{-}[r]&\overset{2}{\circ} \ar@{-}[r]& \overset{1}{\circ}  \\
                                           &                                           &                                                               &\overset{2}{\circ}    &                                                               &                         
 }
\]
where the solid dots correspond to the curves that meet $E$ and the numbers over each dot indicate the multiplicity with which the corresponding  curve appears in $Y^{\prime}_t$. But from the above  diagram we see that $E$ must meet two curves which appear with multiplicity 2 each one in $Y^{\prime}_t$. Hence $Y^{\prime}_t\cdot E \geq 4$, which is  impossible since $Y^{\prime}_t\cdot E=2$.
\end{enumerate}
Hence the only possibility is that $W=C$ and therefore $L\cong\mathcal{O}_{Y^{\prime}}$. Hence $\mathrm{Pic}(Y^{\prime})$ is torsion free and it does not contain $\mu_3$. Hence $Y^{\prime}$ is simply connected and therefore so is $X$. 

\textbf{Case 1.2.3. } Suppose that $p_g(Y^{\prime})=2$. I will show that this case is impossible. Suppose that this case happens. I will show that $\psi$ has no multiple fibers and that $K_{Y^{\prime}}=Y^{\prime}_t$, for $t\in \mathbb{P}^1$. 

From the equations~\ref{sec-all-p-eq-7} it follows since $p_g(Y^{\prime})=2$ that $h^1(R^1\psi_{\ast}\mathcal{O}_{Y^{\prime}})=2$ and therefore 
$R^1\psi_{\ast}\mathcal{O}_{Y^{\prime}}=\mathcal{O}_{\mathbb{P}^1}(-3)$. Then from the adjunction formula for an elliptic or quasi-elliptic ~\cite[Theorem 7.15]{Ba01} it follows that
\begin{gather}\label{sec-all-p-eq-13}
K_{Y^{\prime}}=Y^{\prime}_t +\sum_ia_i P_i,
\end{gather}
where $m_iP_i=Y^{\prime}_{t_i}$ are the multiple fibers of $\psi$ and $0\leq a_i \leq m_i-1$. But $E\cdot K_{Y^{\prime}}=1$. Hence intersecting the previous equation with $E$ it follows that $a_i=0$, for all $i$ and therefore $K_{Y^{\prime}}=Y^{\prime}_t$, as claimed. Moreover from this it follows that $Y^{\prime}_t \cdot E=1$ and hence $E$ is a section of $\psi$. Hence $\psi$ has no multiple fibers. In particular this shows that $T=0$ (since for any $t\in T$, $Y^{\prime}_t$ is a multiple fiber). Hence from equations~\ref{sec-all-p-eq-7} it follows that $H^1(\mathcal{O}_{Y^{\prime}})=0 $ and $\chi(\mathcal{O}_{Y^{\prime}})=3$. Also from the equation~\ref{eq-pg} it follows that $p_g(X)=2$ and since $h^1(\mathcal{O}_X\leq 1$ it follows that $2\leq \chi(\mathcal{O}_X)\leq 3$.

Suppose that $h^1(\mathcal{O}_X)= 1$ and hence $\chi(\mathcal{O}_X)=2$. Let $F^{\ast}\colon H^1(\mathcal{O}_X) \rightarrow H^1(\mathcal{O}_X)$ be the map on cohomology induced by the Frobenious. 

Suppose that $F^{\ast}$ is not zero. Then $X$ admits an \'etale $\mathbb{Z}/3\mathbb{Z}$ cover. But if such a cover existed, then from Lemma~\ref{b1} it would follow that
\[
3\leq \frac{6}{2\chi(\mathcal{O}_X)-K_X^2}=\frac{6}{4-1}=2,
\]
which is clearly impossible. 

Suppose that $F^{\ast}=0$. Then there exists an $\alpha_3$-torsor $\nu \colon Z\rightarrow X$ over $X$. Then $K_Z=\nu^{\ast}K_X$ and $\chi(\mathcal{O}_Z)=3\chi(\mathcal{O}_X)=6$. Hence $K_Z^2=3K_X^2=3$.  
Then from~\cite[Proposition 2]{Li09} it follows that
\[
3=K_Z^2\geq 2h^0(\omega_Z)-4 =2(\chi(\mathcal{O}_Z)-1+h^1(\mathcal{O}_Z))-4\geq 2\chi(\mathcal{O}_Z) -6=6,
\]
which is impossible.

Therefore $H^1(\mathcal{O}_X)=0$ and hence $\chi(\mathcal{O}_X)=3$. Then Noether's formula gives that $c_2(X)=35$. But also since $\chi(\mathcal{O}_{Y^{\prime}})=3$ it follows that $c_2(Y^{\prime})=36$. But then from equation~\ref{sec-all-p-eq-10} it follows that $g$ has exactly one exceptional curve and therefore the singular locus of $Y$ consists of one singularity of type $1/3(1,1)$. But during the study of the case when $p_g(Y^{\prime})=0$, I showed that this case is impossible. This concludes the study of the $p_g(Y^{\prime})=2$.

\textbf{Case 1.3} Suppose that $\kappa(Y^{\prime})=0$. I will show that this case is impossible.

Indeed. If $\kappa(Y^{\prime})=0$ and $Y^{\prime}$ was minimal at the same time, then $K_{Y^{\prime}}\equiv 0$. Hence any $g$-exceptional curve is a curve with self intersection $-2$. Therefore $Y$ has canonical singularities and $g$ is crepant, i.e., $K_{Y^{\prime}}=g^{\ast}K_Y$. hence $K_X^2=pK_Y^2=pK_{Y^{\prime}}^2=0$, which is impossible since $K_X$ is ample.

\textbf{Case 1.4} Suppose that $\kappa(Y^{\prime})=-\infty$. In this case I will show that
\begin{enumerate}
\item If $K_X^2=1$, then $X$ is unirational and $\pi_1^{et}(X)=\{1\}$.
\item If $K_X^2=2$ then $X$ is uniruled. If in addition $\chi(\mathcal{O}_X)\geq 2$, then $X$ is unirational and $\pi_1^{et}(X)=\{1\}$.
\end{enumerate}

Since $\kappa(Y^{\prime})=-\infty$, it follows that $Y^{\prime}$ is ruled. Hence arguing as in the proof of Proposition~\ref{sec-all-p-prop-1} we get that $X$ is uniruled. In particular, if $K_X^2=1$ (or $K_X^2=2$ and  $\chi(\mathcal{O}_X)\geq 2$) then from Lemma~\ref{b1} it follows that $b_1(X)=0$ and hence from diagram~\ref{sec-all-p-diagram-1} it follows that $b_1(Y^{\prime})=0$ and hence $Y^{\prime}$ is rational. Therefore $X$ is unirational. Moreover, from diagram~\ref{sec-all-p-diagram-1} and the fact that the algebraic fundamental group is invariant under \'etale equivalence and birational maps, we get that
\[
\pi_1^{et}(X)=\pi_1^{et}(X^{\prime})=\pi^{et}_1(Y^{\prime})=\{1\}.
\]

\textbf{Case 2.} $Y^{\prime}$ is not a minimal surface. Then the map $\phi \colon Y^{\prime} \rightarrow Z$ is a composition of $m\geq 1$ blow ups. From the equations~\ref{sec-all-p-eq-2} it follows that
\[
\phi^{\ast}K_Z+B+\frac{1}{p}F=g^{\ast}K_Y.
\]
Then,
\begin{gather}\label{sec-all-p-eq-14}
K_Y=g_{\ast}\phi^{\ast}K_Z+g_{\ast}B.
\end{gather}
Since $Y^{\prime}$ is the minimal resolution of $Y$, $g$ does not contract any curves with self intersection $-1$. But since $\phi$ is a composition of blow ups, $B$ has irreducible components of self intersection $-1$. Then write 
$B=B^{\prime}+B^{\prime\prime}$, where $B^{\prime}=\sum_ib^{\prime}_iB_i$ such that $B_i^2=-1$ and $B^{\prime\prime}=\sum_j b^{\prime\prime}_j B_j$ such  that $B_j^2\leq -2$. Then since $\phi$ is the composition of $m$ blow ups it is not hard to see that $\sum _i b_i^{\prime}\geq m$. Hence 
\[
A=g_{\ast}B=n_1A_1+\cdots n_sA_s,
\]
is an effective divisor such that $\sum_{i=1}^sn_i\geq m$.

Consider now cases with respect to the Kodaira dimension $\kappa(Y^{\prime})$ of $Y^{\prime}$. 

\textbf{Case 2.1.}  Suppose that $1 \leq \kappa(Y^{\prime})\leq 2$. I will show that this case is impossible if $K_X^2=1$ and if $K_X^2=2$ then $p=3$ or $5$.

In this case it follows from the classification of surfaces~\cite{BM76},~\cite{BM77} that $nK_Z\sim W$, where $W$ is a nontrivial effective divisor whose birational transform in $Y^{\prime}$ is not contracted by $g$. Hence 
\[
K_X\cdot \pi^{\ast}g_{\ast}\phi^{\ast}K_Z=\frac{1}{n}K_X\cdot \pi^{\ast}g_{\ast}\phi^{\ast}W>0.
\]
Then since $K_X=\pi^{\ast}K_Y$ and  $K_X$ is ample, it follows from equation~\ref{sec-all-p-eq-14}  that
\[
K_X^2=K_X\cdot \pi^{\ast}K_Y=K_X \cdot \pi^{\ast}( g_{\ast}\phi^{\ast}K_Z) +K_X \cdot \pi^{\ast}A \geq 1+m.
\]
Hence since $m \geq 1$ it follows that $K_X^2 \geq 2$. Hence if $K_X^2=1$, then it is not possible that $1\leq \kappa(Y^{\prime})\leq 2$. If $K_X^2=2$ then $m=1$, which means that $\phi$ is a single blow up. By following  a similar argument as this that will be used in the case when $K_X^2=1$ and $\kappa(Y^{\prime})=0$ it follows that if $K_X^2=2$ and $\kappa(Y^{\prime})\geq 1$, then $p=3,5$. The calculations are similar to the case when $K_X^2=1$ and the details are left to the reader.

\textbf{Case 2.1.}  Suppose that $\kappa(Y^{\prime})=0.$ In this case I will show that  if $K_X^2=1$ then $p=7$ and if $K_X^2=2$ then $p=3$. 

As usual I will do the case when $K_X^2=1$ in detail and leave the other to the reader. The method is identical but there is a fairly larger amount of calculations involved. 

Since $\kappa(Y^{\prime})=0$, it follows that $K_Z\equiv 0$. Therefore from the equation~\ref{sec-all-p-eq-14} it follows that $K_Y\equiv g_{\ast}B=A$. Then 
\[
K_X^2=K_X\cdot \pi^{\ast}K_Y=K_X\cdot \pi^{\ast}A =n_1K_X\cdot A_1 +\cdots +n_s K_X\cdot A_s \geq \sum_{i=1}^s n_i\geq m,
\]
 where $m$ is the number of blow ups that $\phi$ consists of. Hence if $K_X^2=1$ then $m=1$, and if $K_X^2=2$ then $m=1$ or $2$.

Suppose then that $K_X^2=1$ and $\phi$ is the blow up of a single point. Then $B\cong \mathbb{P}^1$ and $K_{Y^{\prime}}^2=B^2=-1$. From the classification of surfaces~\cite{BM76},~\cite{BM77}, it is known that $c_2(Z)\in\{0,12,24\}$. Moreover, since $K_X^2=1$, then $1\leq \chi(\mathcal{O}_X) \leq 3$~\cite{Li09}. Now from the diagram~\ref{sec-all-p-diagram-1} it follows that
\[
c_2(Y^{\prime})=\chi_{et}(Y^{\prime})=\chi_{et}(Y)+k=\chi_{et}(X)+k=c_2(X)+k,
\]
where $k$ is the number of $g$-exceptional curves. Since $\phi$ is a single blow up, $c_2(Y^{\prime})=c_2(Z)+1$ and therefore  
\begin{gather}\label{sec-all-p-eq-15}
c_2(Z)=c_2(X)+k-1.
\end{gather}
Consider now cases with respect to $\chi(\mathcal{O}_X)$. 

\textbf{Case 2.1.1.} Suppose that $\chi(\mathcal{O}_X)=3$. Then $c_2(X)=35$ and hence $c_2(Z)=34+k$. But since $c_2(Z)\leq 24$, this is impossible.

\textbf{Case 2.1.2.} Suppose that $\chi(\mathcal{O}_X)=2$. Then $c_2(X)=23$ and hence from~\ref{sec-all-p-eq-15} it follows that $c_2(Z)=22+k$. Hence $Z$ is a K3 surface and $k=2$, i.e., $g$ has exactly two exceptional curves. Hence $F=a_1F_1+a_2F_2$, $a_i\geq 0$. Hence 
\begin{gather}\label{sec-all-p-eq-16}
K_{Y^{\prime}}+\frac{1}{p}(a_1F_1+a_2F_2)=g^{\ast}K_Y,
\end{gather}
and therefore
\begin{gather}\label{sec-all-p-eq-17}
K_X^2=pK_Y^2=-p+a_1(K_{Y^{\prime}}\cdot F_1)+a_2K_{Y^{\prime}}\cdot F_2.
\end{gather}
Suppose that $F_i^2=-d_i$, $d_i\geq 2$, $i=1,2$. Intersecting the equation~\ref{sec-all-p-eq-16} with $F_1$ and $F_2$ we get that
\begin{gather}\label{sec-all-p-eq-18}
a_1=p\left(1-\frac{2}{d_1}\right)+\frac{a_2}{d_1}(F_1\cdot F_2)\\
a_2=p\left(1-\frac{2}{d_2}\right)+\frac{a_1}{d_2}(F_1\cdot F_2 )\nonumber
\end{gather}

Consider next cases with respect to the values of $d_1$ and $d_2$.

Suppose that $d_i \geq 4$, for $i=1,2$. Then from the equations~\ref{sec-all-p-eq-18} it follows that $a_i \geq p/2$, $i=1,2$.  Taking into consideration also that $K_{Y^{\prime}}\cdot F_i =d_i-2\geq 2$, it follows from the equation~\ref{sec-all-p-eq-17} that $K_X^2\geq p$.

Suppose that at least one of the $d_i$ is less than $4$, say  $d_2\leq 3$.  

Suppose first that $d_2=2$. Then from the equation~\ref{sec-all-p-eq-17} it follows that 
\begin{gather}\label{sec-all-p-eq-19}
K_X^2=-p+a_1(d_1-2).
\end{gather}
 I will show  that $F_1\cdot F_2\not=0$. Indeed, if $F_1\cdot F_2 =0$, then $Y$ would have a singularity of type $A_1$. But by Proposition~\ref{prop3} the local class groups of the singularities of $Y$ are $p$-torsion.  Then since  the local class group of an $A_1$ singularity is 2-torsion  $p$ must be $2$. However we are assuming that 
$p\not=2$. Hence $F_1 \cdot F_2 \not=0$. Then from the equations~\ref{sec-all-p-eq-18} it follows that $a_2=\frac{1}{2}a_1(F_1\cdot F_2)\geq \frac{1}{2}a_1$. Then from the equation~\ref{sec-all-p-eq-18} it follows that 
\begin{gather}\label{sec-all-p-eq-20}
a_1\geq p \frac{2(d_1-2)}{2d_1-1}.
\end{gather}
Then from the equation~\ref{sec-all-p-eq-17} it follows that
\begin{gather}\label{sec-all-p-eq-21}
K_X^2\geq -p+2p \frac{(d_1-2)^2}{2d_1-1}.
\end{gather}
It is not difficult to see that $(d_1-2)^2/(2d_1-1)\geq 1$, if $d_1\geq 5$ and therefore in this case $K_X^2 \geq p$. It remains to examine the cases $d_1\leq 4$. 

Suppose that $d_1=2$. Then since $d_2=2$, $Y$ has a canonical singularity of type $A_2$. But then $g$ would be  crepant, i.e., $K_{Y^{\prime}}=g^{\ast}K_Y$ and hence $a_1=a_2=0$. Then from the equation~\ref{sec-all-p-eq-17} it follows that $K_X^2=-p<0$, which is impossible since $K_X$ is ample.

Suppose that $d_1=3$ and $d_2=2$. In this case $Y$ has a singularity of type $1/5(1,2)$ and hence $p=5$. Moreover, a  straightforward calculation shows that $a_1=2$ and $a_1=1$. Then since $K_{Y^{\prime}}\cdot F_1=1$ and $K_{Y^{\prime}}\cdot F_2=0$, it follows from the equation~\ref{sec-all-p-eq-17} that $K_X^2=-5+2=-3<0$, which is impossible since $K_X$ is ample.

Suppose that $d_1=4$ and $d_2=2$. Then $Y$ has a singularity of type $1/7(1,2)$ and therefore $p=7$. A straightforward calculation shows that $a_1=4$ and $a_2=2$. Then equations~\ref{sec-all-p-eq-17} shows that $K_X^2=1$. Unfortunately I am unable to show that this case is impossible. This is the reason of the assumption $p\not=7$ in~\ref{main-theorem}.

The last cases that need to be considered are when $d_i\geq 3$, $i=1,2$,  and at least one of the $d_i$ is 3. There are the following possibilities. 
\begin{enumerate}
\item $F_1^2=F_2^2=-3$ and $F_1\cdot F_2=1$.
\item $F_1^2=F_2^2=-3$ and $F_1\cdot F_2=0$.
\item $F_1^2=-3$, $F_2^2=-4$ and $F_1\cdot F_2=1$.
\item $F_1^2=-3$, $F_2^2=-4$ and $F_1\cdot F_2=0$.
\end{enumerate}  
All cases are impossible. Indeed. 

A simple calculation shows that in the first case  $Y$ has a singularity of index 8. This is impossible because by Proposition~\ref{prop3}, $Y$ has singularities of prime index. 

In the second case, $Y$ has two singularities of type $1/3(1,1)$ each. Then a straightforward calculation shows that $a_1=a_2=1$ and $p=3$. Then from the equation~\ref{sec-all-p-eq-17} it follows that $K_X^2=-3+1+1=-1<0$, which is impossible. 

In the third case $Y$ has a singularity of type $1/11(1,4)$. Hence $p=11$ and a simple calculation shows that $a_1=6$ and $a_2=7$. Then from the equation~\ref{sec-all-p-eq-17} it follows that $K_X^2=-11+6+7\cdot 2=9 >1$. 

Finally in the fourth case, $F_2$  contracts to a singularity of index 2. But this is again impossible because $p\not= 2$.

\textbf{Case 2.1.3.} Suppose that $\chi(\mathcal{O}_X)=1$.  Then from Noether's formula we get that $c_2(X)=11$. Also from the equation~\ref{sec-all-p-eq-15} it follows that $c_2(Z)=10+k$, where $k$ is the number of $g$-exceptional curves. Hence either $k=2$ and $Z$ is an Enriques surface or $k=14$ and $Z$ is a $K3$ surface. 

Suppose that $k=2$ and $Z$ is an Enriques surface. This situation is similar to the one of case 2.1.2 above with the only difference that $Z$ is now an Enriques surface instead of a K3 surface. However, the only property of a K3 surface that was used in the argument that was used in the case 2.1.2 is that $K_Z\equiv 0$. Hence it applies in this case too. Therefore this case is impossible unless possibly for $p=7$. 

Suppose that $k=14$ and $Z$ is a K3 surface. In this case I will show that $p\in\{3,5\}$ and $X$ is a simply connected supersingular Godeaux surface.

Since $Z$ is a $K3$ surface, it follows that $\pi_1^{et}(Z)=\{1\}$~\cite{BM76},~\cite{BM77}. Then since the algebraic fundamental group is invariant under \'etale and birational equivalence it follows from the 
diagram~\ref{sec-all-p-diagram-1} that $\pi_1^{et}(X)=\{1\}$. Moreover, from equation~\ref{eq-pg} it follows that $p_g(X)=1$. Hence since we are assuming that $\chi(\mathcal{O}_X)=1$, it follows that $H^1(\mathcal{O}_X)=k$ and therefore $X$ is either a singular or a supersingular Godeaux surface and therefore $p\leq 5$~\cite{Li09}. Suppose that $X$ was a singular Godeaux surface. Then $X$ admits a nontrivial  \'etale $p$-cover. But since $\pi_1^{et}(X)=\{1\}$, there is no such cover. Hence $X$ is a simply connected supersingular Godeaux surface. This concludes the proof of Theorem~\ref{sec-all-p-prop-2}.

\end{proof}

The last result of this section gives restrictions on the geometric genus of a smooth canonically polarized surface that has nontrivial global vector fields.

\begin{theorem}\label{sec-all-p-prop-3}
Let $X$ be a smooth canonically polarized surface defined over an algebraically closed field of characteristic $p\not\in\{2,3,5,7\}$ such  that $K_X^2=1$. Suppose that $X$ admits a nontrivial global vector field $D$. Then $p_g(X)\leq 1$.
\end{theorem}
\begin{remark}
According to~\cite[Corollary 1.8]{Ek87},  if $X$ is any canonically polarized surface with $K_X^2=1$, then $p_g(X)\leq 2$ and $1\leq \chi(\mathcal{O}_X)\leq 3$. The previous theorem says that canonically polarized surfaces with $p_g(X)=2$ or $\chi(\mathcal{O}_X)=3$ do not have non trivial global vector fields if $p >7$.
\end{remark}

\begin{proof}[Proof of Theorem~\ref{sec-all-p-prop-3}]
Suppose that $X$ has a nontrivial global vector field $D$. According to Proposition~\ref{prop1} we may assume that $D$ is either or additive or multiplicative type. Note that from the proof of Propositions~\ref{sec-all-p-prop-1},~\ref{sec-all-p-prop-2} it follows that if $K_X^2=1$ and $X$ has nontrivial global vector fields, then $\kappa(Y^{\prime})=-\infty$. Moreover, ~\cite[Corollary 1.8]{Ek87} it follows that if $K_X^2=1$ then $p_g(X)\leq 2$. Therefore in order to prove the proposition it suffices to show that the case $p_g(X)=2$ is impossible. Suppose then that $p_g(X)=2$. Then the linear system $|K_X|$ is 2-dimensional. I will show that the linear system $|K_X|$ has a unique base point $P\in X$ and that every member $C\in |K_X|$ is an integral curve which is smooth at $P$.

Let $C$ be a member of $|K_X|$. Then $K_X\cdot C=K_X^2=1$. Therefore, since $K_X$ is ample, it follows that $C$ is irreducible and reduced. I will next show that $|K_X|$ has exactly one base point $P$.  Since $\dim|K_X|=2$ and $K_X^2=1$, it follows that $|K_X|$ has base points. Let $P_1,\ldots, P_n$ be its base points and let $C_1, C_2 $ be two different members of $|K_X|$. Since $C_1, C_2$ are integral curves, it follows that  $C_1\cdot C_2 \geq n$. However since $C_1\cdot C_2=K_X^2=1$ it follows that $n=1$ and therefore $|K_X|$ has a unique base point $P$. 

Next I will show that every member of $|K_X|$ is smooth at $P$. Let $f \colon W \rightarrow X$ be the blow up of $X$ at $P$ and  let $E$ be the $f$-exceptional curve. Then for any two different members $C_1$ and $C_2$ of $|K_X|$, 
\begin{gather*}
f^{\ast}C_1=C^{\prime}_1+m_1E\\
f^{\ast}C_2=C^{\prime}_2+m_2E
\end{gather*}
Hence
\[
1=f^{\ast}C_1\cdot f^{\ast}C_2=C_1^{\prime}\cdot f^{\ast}C_2+m_1E\cdot f^{\ast}C_2=C_1^{\prime}\cdot C_2^{\prime}+m_1m_2.
\]
Hence $C_1^{\prime}\cdot C_2^{\prime}=0$ and $m_1=m_2=1$. Therefore $C_1$ and $C_2$ are smooth at $P$ as claimed.

Consider now cases with respect to whether $P$ is an isolated fixed point of $D$ or not.

\textbf{Case 1:} Suppose that $P$ is an isolated fixed point of $D$. Let then $f \colon X^{\prime} \rightarrow X$ be the blow up of $P\in X$ and let $E$ be the $f$-exceptional curve. Let $C\in |K_X|$ be any member and $C^{\prime}$ be its birational transform in $X^{\prime}$. Then $C^{\prime}\in |K_{X^{\prime}}-2E|$ and $(C^{\prime})^2=0$. Then the linear system $|K_{X^{\prime}}-2E|$ defines a fibration $h\colon X^{\prime} \rightarrow \mathbb{P}^1$ whose fibers are the birational transforms of the fibers of $|K_X|$. 

\textbf{Case 1.1.} Suppose that every member of $|K_X|$ is singular, or equivalently every fiber of $h$ is singular. The general fiber of $h$ is a normal integral curve $C_{k(t)}$ defined over the function field $k(t)$ of 
$\mathbb{P}^1$. Since it is not smooth, there exists a purely inseparable extension $k(t)\subset K$ of $k(t)$ such that $C_K=C_{k(t)}\otimes_{k(t)}K$ is not normal. Let $\tilde{C}_K$ be its normalization. Then by~\cite{Sch09}, the difference $p_a(C_{k(t)})-p_a(\tilde{C}_K)$ is divisible by $(p-1)/2$. But $p_a(C_{k(t)})=2$ since for any $C\in |K_X|$, $p_a(C)=1/2(K_X\cdot C+C^2)+1=2$. Moreover, $p_a(\tilde{C}_K)\leq 1$. Therefore $p \in\{2,3,5\}$.  

\textbf{Case 1.2.} Suppose that the general member of $|K_X|$ is smooth. Consider now two cases with respect to whether the general member of $|K_X|$ is an integral curve for $D$ or not.

\textbf{Case 1.2.1.} Suppose that the general member of $|K_X|$ is an integral curve for $D$. Then for general $C\in |K_X|$, $C=\pi^{\ast}C_Y$, where $C_Y=\pi(C)$. Then $D$ restricts to a vector field $D_C$ on $C$. However, since the general $C$ is a smooth curve of genus 2, and smooth curves of genus $\geq 2$ do not have any non trivial global vector fields,  it follows that $D_C=0$. Hence $D(\mathcal{O}_X) \subset I_C$, where $I_C$ is the ideal sheaf of $C$. Hence $C$ is contained in the divisorial part of $D$. Since there are infinitely such $C\in |K_X|$, this is impossible.

\textbf{Case 1.2.2.} Suppose that  the general member of $|K_X|$ is not an integral curve for $D$. Then $\pi_{\ast}C\cong C$. I will show first that in fact every member of $|K_X|$ is not an integral curve of $D$. Suppose that there exists a curve $C_0\in |K_X|$ which is an integral curve of $D$. Then $\pi_{\ast}C_0=p\tilde{C}_0$. where $\tilde{C}_0=\pi(C_0)$. Let $C\in |K_X|$ be a general member. Then $\pi_{\ast}C\cong C$, and hence $\pi_{\ast}C$ is a smooth integral curve in $Y$. But also $\pi_{\ast}C \sim \pi_{\ast}(C_0)=p\tilde{C}_0$. Since the local class groups of $Y$ are $p$-torsion it follows that $\pi_{\ast}C$ is a Cartier divisor. Hence since $C$ is smooth and $Q=\pi(P)\in\pi_{\ast}C$, it follows that $Q\in Y$ is a smooth point. But this is impossible since by assumption $P\in X$   is an isolated fixed point of $D$ and therefore by Proposition~\ref{prop3}, $Q=\pi(P)$ is a singular point of $Y$. Therefore $\pi_{\ast}C\cong C$, for all $C \in |K_X|$.

Let $f_1 \colon X_1 \rightarrow X$ be the blow up of $X$ at $P$ and let $E$ be the $f_1$-exceptional curve. Since $P$ is a fixed point of $D$, $D$ lifts to a vector field $D_1$ on $X_1$. Let $\nu \colon X_1 \rightarrow Y_1$ be the quotient of $X_1$ by the $\alpha_p$ or $\mu_p$ action on $X_1$ induced by $D_1$. Let $F=\nu(E)$. then there exists a commutative diagram
\begin{gather}\label{sec-all-p-diagram-2}
\xymatrix{
X_1 \ar[r]^{f_1}\ar[d]_{\nu} & X \ar[d]^{\pi} \\
Y_1 \ar[r]^{g_1} & Y
}
\end{gather}
where $g_1$ is birational and $F$ is the $g_1$-exceptional curve.  

\textit{Claim:} $Y_1$ is smooth along $F$. 

Indeed. For any  $C \in |K_X|$, let $C^{\prime}$ denote its birational transform in $X_1$, $\hat{C}=\nu(C^{\prime})$ and $\tilde{C}=\pi(C)$.  Then since $P\in C$ is a smooth point and $C$ is not an integral curve of $D$, it follows that  $\hat{C}=\nu_{\ast}C^{\prime}\cong C^{\prime}$. Let $C_{i_1},\ldots C_{i_p}$ be $p$ different members of $|K_X|$. Then $\hat{C}_{i_1}+\cdots+\hat{C}_{i_p} \sim p\hat{C}_{i_1}$, and therefore since the local class groups of $Y_1$ are $p$-torsion, it follows that $\hat{C}_{i_1}+\cdots+\hat{C}_{i_p} $ is a Cartier divisor in $Y_1$. Since also $C^{\prime}_{i}\cdot C^{\prime}_j=0$, for any two different curves $C_i,C_j \in |K_X|$, it follows that $\hat{C}_i \cdot \hat{C}_j =0$ as well. Hence every curve $\hat{C}$ is Cartier in $Y$. Considering now that all $C\in |K_X|$ are smooth at $P$, it follows that $C^{\prime}$ is smooth at $E\cap C^{\prime}$ and hence $\hat{C}$ is smooth at $\hat{C} \cap F$. Since $\hat{C}$ cover $F$, $Y_1$ is smooth along $F$ as claimed.

Next I will show that $\nu^{\ast}F=pE$, or equivalently that $E$ is not an integral curve of $D_1$. Suppose otherwise that $E$ is an integral curve of $D_1$. Then $\nu^{\ast}F=E$. Then $E^2=pF^2$ and hence $F^2=-1/p \not\in\mathbb{Z}$. However this is impossible since $Y_1$ is smooth along $F$ and hence $F$ is a Cartier divisor in $Y_1$. Hence $E$ is not an integral curve of $D$ and therefore $\nu^{\ast}F=pE$. Hence $F^2=-p$. Now let $m>0$ such that $g_1^{\ast}\tilde{C}=\hat{C}+mF$.  Then from the diagram~\ref{sec-all-p-diagram-2} it follows that $f_1^{\ast}\pi^{\ast}\tilde{C}=\nu^{\ast}g_1^{\ast}\tilde{C}$. Then since $\pi^{\ast}\tilde{C}=pC$, $f_1^{\ast}C=C^{\prime}+E$, 
$\nu^{\ast}F=pE$ and $\nu^{\ast}\hat{C}=pC^{\prime}$ it follows that
\[
pC^{\prime}+pmE=pC^{\prime}+pE.
\]
Therefore $m=1$ and hence $g_1^{\ast}\tilde{C}=\hat{C}+F$. But since $F^2=-p$, this implies that $\hat{C} \cdot F=p$. But this is impossible because since $C$ is smooth at $P$ and $\tilde{C}\cong C$, $\tilde{C}$ is smooth at 
$Q=\pi(P)$ and hence the restriction map $g_1\colon \hat{C}\rightarrow \tilde{C}$ is an isomorphism.

\textbf{Case 2:} Suppose that $P$ is not an isolated fixed point of $D$. Therefore $Q=\pi(P)$ is a smooth point of $Y$. 

\textit{Claim 1:} There can be at most one member of $|K_X|$ that is an integral curve of $D$. 

Indeed. Suppose that there were are least two members of $|K_X|$ that were integral curves of $D$. Let $C_1, C_2$ be two such members. Then $C_i=\pi^{\ast}\tilde{C}_i$, $i=1.2$, where $\tilde{C}_i=\pi(C_i)$. Then 
\[
K_X^2=C_1\cdot C_2=\pi^{\ast}\tilde{C}_1 \cdot \pi^{\ast}\tilde{C}_2 =p\tilde{C}_1\cdot \tilde{C}_2\geq p,
\]
since $Q\in \tilde{C}_1\cap\tilde{C}_2$ and $Q\in Y$ is a smooth point. 

\textit{Claim 2:} $\Delta \not= 0$. 

In order to prove this I will consider two cases with respect to whether $D$ is of additive or multiplicative type.

Suppose that $D$ is of additive type, i.e., $D^p=0$. Then from Proposition~\ref{structure-of-pi} it follows that there exists a filtration of $\mathcal{O}_Y$-modules
\begin{gather}\label{sec-all-p-eq-22}
\mathcal{O}_Y =E_0 \underset{\not=}{\subset} E_1 \underset{\not=}{\subset} \cdots \underset{\not=}{\subset} E_{p-2} \underset{\not=}{\subset} E_{p-1}=\pi_{\ast}\mathcal{O}_X,
\end{gather}
with the following properties:
\begin{enumerate}
\item $E_k$ is a reflexive sheaf of rank $k+1$, for $k=1,\ldots, p-1$.
\item The quotient $E_k/E_{k-1}=L_{k}$ is a rank 1 torsion free sheaf on $Y$, for $k=1,\ldots, p-1$.
\item There are nonzero maps $\sigma_k \colon L_k \rightarrow \mathcal{O}_Y$ given locally by $\sigma_k(a)=D^k(a)$ which identify $L_k$ with ideal sheaves $I_{Z_k}$ of $\mathcal{O}_Y$. Moreover, these maps  are isomorphisms at every point not in the fixed locus of $D$. 
\end{enumerate}
If $\Delta =0$, then $D$ has only isolated singular points. Therefore $L_k\cong I_{Z_k}$, where $Z_k$ are zero-dimensional subschemes of $Y$ whose support is the singular locus of $Y$.

I will show that $p_g(Y)=p_g(X)=2$. This is a contradiction since from the proof of Propositions~\ref{sec-all-p-prop-1},~\ref{sec-all-p-prop-2}, $\kappa(Y^{\prime})=-\infty$ and therefore $p_g(Y)=p_g(Y^{\prime})=0$. 

Let $V$ be the smooth locus of $Y$ and $U=\pi^{-1}(V)$. Then since  $\Delta=0$, it follows that $\mathrm{codim}(X-U,X)\geq 2$ and $D$ has no fixed points in $U$. Then $I_{Z_k}|_V=\mathcal{O}_V$, for all $k=1,\ldots, p-1$, and  
 $\omega_U=\pi^{\ast}\omega_V$.  Moreover, since $\kappa(Y^{\prime})=-\infty$, it follows that $H^0(\omega_Y)=H^0(\omega_{Y^{\prime}})=0$. Then from the exact sequences
\[
0 \rightarrow E_{k-1}|_V \rightarrow E_k|_V  \rightarrow L_k|_V=\mathcal{O}_V \rightarrow 0,
\]
it follows that there are exact sequences
\[
0 \rightarrow (E_{k-1}\otimes \omega_Y)|_V \rightarrow (E_{k}\otimes \omega_Y)|_V \rightarrow \omega_V \rightarrow 0,
\]
for $k=1,\ldots,p-1$. Taking cohomology and using the fact that $H^0(\omega_V)=H^0(\omega_Y)=0$, it follows that 
\[
H^0(\omega_X)=H^0(\omega_U)=H^0(\pi^{\ast}\omega_V)=H^0(\omega_V\otimes \pi_{\ast}\mathcal{O}_U)=H^0(\omega_V)=H^0(\omega_Y),
\]
which is a contradiction since $H^0(\omega_X)=k^2$ and $H^0(\omega_Y)=0$.

Suppose now that $D$ is of multiplicative type, i.e., $D^p=D$. Then from Proposition~\ref{structure-of-pi} it follows that there exists a decomposition
\[
\pi_{\ast}\mathcal{O}_X=\mathcal{O}_Y\oplus L_1\oplus \cdots \oplus L_{p-2}\oplus L_{p-1},
\]
where $L_k$ is a rank 1 reflexive sheaf on $Y$, $k=1,\ldots,p-1$. Then since $\Delta=0$, $\omega_X=\pi^{\ast}\omega_Y$ and hence
\[
\pi_{\ast}\omega_X= \omega_Y \oplus (L_1\otimes \omega_Y)^{\ast\ast} \oplus \cdots \oplus (L_{p-2}\otimes \omega_Y)^{\ast\ast} \oplus( L_{p-1}\otimes \omega_Y)^{\ast\ast}.
\]
Therefore,
\begin{gather*}
H^0(\omega_X)=H^0(\pi^{\ast}\omega_Y)= H^0(\omega_Y\otimes\pi_{\ast}\mathcal{O}_X) =\\
H^0(\omega_Y) \oplus H^0((L_1\otimes \omega_Y)^{\ast\ast}) \oplus \cdots \oplus H^0((L_{p-2}\otimes \omega_Y)^{\ast\ast}) \oplus H^0(( L_{p-1}\otimes \omega_Y)^{\ast\ast}).
\end{gather*}
Then, since $H^0(\omega_X)=k^2$, there exists either one $\lambda \in\{1,2,\ldots,p-1\}$ such that $H^0((L_{\lambda}\otimes \omega_Y)^{\ast\ast})=k^2$ or there exists $1\leq \lambda_1< \lambda_2\leq p-1$, such that 
$H^0((L_{\lambda_1}\otimes \omega_Y)^{\ast\ast})=k$ and $H^0((L_{\lambda_2}\otimes \omega_Y)^{\ast\ast})=k$. Therefore in any case there exist two different $C_1,C_2 \in|K_X|$ such that $C_i=\pi^{\ast}\tilde{C}_i$, $i=1,2$. But then 
\[
K_X^2=C_1\cdot C_2=\pi^{\ast}\tilde{C}_1\cdot \tilde{C}_2=p\tilde{C}_1 \cdot \tilde{C}_2.
\]
Then $P$ must be an isolated singular point of $D$ because otherwise, $Q=\pi(P)$ is a smooth point of $Y$ and since $Q \in \tilde{C}_1 \cap \tilde{C}_2$, it would follow that $K_X^2\geq p$. Hence $P$ is an isolated fixed point of 
$D$. Now note that in the proof of the case when $P$ was an isolated fixed point of $D$ it was not important if $\Delta =0$ or $\Delta \not=0$. Therefore the same arguments can be used in this case to show that $p_g(X)=2$ is impossible if $\Delta =0$ and $D^p=D$. This concludes the proof of Claim 2.

So far it has been shown that $\Delta\not= 0$ and that for all but at most one $C \in |K_X|$, $C$ is not an integral curve of $D$ and hence $\tilde{C} =\pi_{\ast}C\cong C$, $\pi^{\ast}\tilde{C}=pC$. Hence for any such $C$, 
$\tilde{C}^2=pC^2=p$. Moreover, for any two $C_1, C_2 \in|K_X|$ which are not integral curves, $\tilde{C}_1\cdot \tilde{C}_2 =(1/p)\pi^{\ast}\tilde{C_1}\cdot \pi^{\ast}\tilde{C}_2=(1/p)p^2K_X^2=p$.

Since $P$ is not an isolated fixed point of $D$ it follows that $Q=\pi(P)$ is a smooth point of $Y$. It is not difficult to see now that after blowing up $Q$ p-times we get a map $h \colon W \rightarrow Y$, such that the birational transforms $\hat{C}$ of $\tilde{C}$ in $W$ form a 2-dimensional linear system and $\hat{C}^2=0$. Of course since $\tilde{C}$ is smooth at $Q$, $\hat{C}\cong \tilde{C}$.  Then there exists a fibration $\phi \colon W\rightarrow \mathbb{P}^1$ all whose fibers except possibly one are the birational transforms $\hat{C}$. This fiber  corresponds to the possible member of $|K_X|$ that is an integral curve of $D$. Then for any $C\in|K_X|$ which is not an integral curve of $D$, there exists an exact sequence,
\[
0 \rightarrow \phi^{\ast}\mathcal{O}_{\mathbb{P}^1}(-1)\rightarrow \mathcal{O}_W\rightarrow \mathcal{O}_{\hat{C}} \rightarrow 0.
\]
Since $\kappa(Y^{\prime})=-\infty$, it follows that $H^1(\mathcal{O}_W)=H^2(\mathcal{O}_W)=H^0(\omega_W)=0$. Taking now cohomology and using Serre duality in the above exact sequence, we get that 
\[
H^1(\mathcal{O}_{\hat{C}})=H^2(\phi^{\ast}\mathcal{O}_{\mathbb{P}^1}(-1))=H^0(\phi^{\ast}\mathcal{O}_{\mathbb{P}^1}(1)\otimes \omega_W).
\]
But since $\hat{C} \cong {C}$, it follows that $p_a(\hat{C})=2$ and therefore $H^1(\mathcal{O}_{\hat{C}})=k^2$. Hence $H^0(\phi^{\ast}\mathcal{O}_{\mathbb{P}^1}(1)\otimes \mathcal{O}_W)=k^2$. Hence the linear system $|K_W+\hat{C}|$ is 2-dimensional. Let $Z_W \in |K_W+\hat{C}|$ be a member of the linear system and $Z_Y=h_{\ast}Z_W\in|K_Y+\tilde{C}|$. I will show that $h^{\ast}Z_Y=Z_W$. In particular, if $|K_W+\hat{C}|$ does not have any base components, then $Z_W\cong Z_Y$. Indeed, since $h$ is $p$ successive blow ups of the base points of $|\tilde{C}|$, it is not hard to see that 
\begin{gather*}
K_W=h^{\ast}K_Y+F,\\
h^{\ast}\tilde{C}=\hat{C} +F,
\end{gather*}
where $F$ is an $h$-exceptional divisor supported, since $Q\in Y$ is smooth, on the full $h$-exceptional set. Let $F_i$ be any $h$-exceptional curve. Then 
\[
F_i\cdot Z_W=F_i\cdot K_W+F_i\cdot \hat{C}=F_i\cdot F-F_i\cdot F=0.
\]
Hence $Z_W =h^{\ast}Z_Y$. Hence if $|K_W+\hat{C}|$ does not  have any base components then  $Z_Y\cong Z_X$. Hence $K_Y=Z_Y-\tilde{C}$. Then from this and the 
equation~\ref{sec-all-p-eq-1} it follows that
\begin{gather}\label{sec-all-p-eq-23}
(p+1)K_X=\pi^{\ast}Z_Y+(p-1)\Delta.
\end{gather}
Since $K_X^2=1$ it follows from this equation that 
\[
p+1=K_X\cdot \pi^{\ast}Z_Y +(p-1)K_X \cdot \Delta. 
\]
Since $K_X$ is ample and $\Delta \not=0$, it follows then that $\pi^{\ast}Z_Y \cdot K_X=2$ and $K_X \cdot \Delta =1$. 
This says that $\Delta$ is irreducible, $Z_X=\pi^{\ast}Z_Y$ has at most two irreducible components and if $p\not= 2$, then every irreducible component of $Z_X$ is an integral curve of $D$. If there was an irreducible component $Z_1$ which was not an integral curve, then $Z_X=pZ_1+Z_1^{\prime}$, and hence $K_X\cdot Z_X\geq p>2$. 

Suppose that $|\pi^{\ast}Z_Y|$ does not have a base component.  Since $|\pi^{\ast}Z_Y|$ does not have any base components and since $\pi^{\ast}Z_Y\cdot K_X=2$, it follows from the equation~\ref{sec-all-p-eq-23} that 
\[
2p+2=(\pi^{\ast}Z_Y)^2+(p-1)\Delta \cdot \pi^{\ast}Z_Y.
\]
Then $\Delta \cdot \pi^{\ast}Z_Y\leq 2$. If $\Delta \cdot \pi^{\ast}Z_Y =1$, then from the equation~\ref{sec-all-p-eq-23}  and since $K_X \cdot \Delta =1$, it follows that $p+1=1+(p-1)\Delta^2$, which is impossible if $p\not=2$.  Hence 
$\Delta \cdot \pi^{\ast}Z_Y = 2$ and therefore $(\pi^{\ast}Z_Y)^2=4$.

Suppose that the general member $Z_X\in |\pi^{\ast}K_Y|$ is a smooth irreducible  curve.  Then as was mentioned earlier, $Z_X$ is an integral curve of $D$ and hence $D$ restricts to a vector field $D^{\prime}$ in $Z_X$. Since $|\pi^{\ast}Z_Y|$ has infinitely many members, $D^{\prime}$ is not zero for general $Z_X\in |\pi^{\ast}Z_Y|$. But $Z_X$ is a smooth curve of genus $p_a(Z_X)=(1/2)(Z_X^2+K_X\cdot Z_X)+1=4$ and such curves do not have non trivial global vector fields. 

Suppose that every $Z_X\in|\pi^{\ast}K_Y|$ is singular. In this case I will show that $p\in\{2,3,5,7\}$.  

Let $R\in X$ be a base point of $|\pi^{\ast}Z_Y|$. I will show that $R$ is a fixed point of $D$  and therefore $D$ lifts to a vector field $D_1$ on the blow up $X_1$ of $X$ at $R$. Let $Z_1, Z_2$ be two general members of $|\pi^{\ast}Z_Y|$. Then $Z_1 \cdot Z_2 =4$. Then $R \in Z_1\cap Z_2$ and, since $(\pi^{\ast}Z_Y)^2=4$,  $Z_1, Z_2$ intersect at $R$ with multiplicity at most 4.   The property that $R$ is a fixed point of $D$ is local at $R$. So suppose that $X=\mathrm{Spec} A$, where $(A,m)$ is a regular local ring, $Z_1$ is defined by the ideal $I_1$, $Z_2$ by the ideal $I_2$ and $R$ corresponds to the maximal ideal $m$. Then since $Z_1, Z_2$ intersect at $R$ with multiplicity at most 4, it follows that $\mathrm{length}(A/(I_1+I_2))\leq 4$. Let $x \in m$ be any element and let $s$ be the least positive integer such that  $x^s \in I_1+I_2$. Since $\mathrm{length}(A/(I_1+I_2))\leq 4$, it follows that $s \leq 4$. Then since $Z_1, Z_2$ are integral curves of $D$ $D(I_i)\subset I_i$. $i=1,2$. Then 
\[
D(x^s)\in D(I_1+I_2)=D(I_1)+D(I_2)\subset I_1+I_2 \subset m.
\]
But also $D(x^s)=sx^{s-1}Dx$. Hence $sx^{s-1}Dx \in I_1+I_2$. Suppose that $Dx \not\in m$. Then $Dx$ is a unit in $A$. If $p \geq 5$, then $s \not=0$ and hence $x^{s-1}Dx \in I_1+I_2$. Therefore since $Dx$ is a unit it follows that $x^{s-1} \in I_1+I_2$. But this is a contradiction since $s$ was the minimum integer such that $x^s \in I_1+I_2$. Hence the base point $R$ of $|\pi^{\ast}Z_Y|$ is a fixed point of $D$. 

There are now two cases with respect to the base point $R$. Either there are infinitely many members of $|\pi^{\ast}Z_Y|$ which are singular at $R$ or there are infinitely many which are smooth. 

Consider first the case when the general member of $|\pi^{\ast}Z_Y|$ is singular at $R$. In this case I will show that $R$ is the only base point of $|\pi^{\ast}Z_Y|$ and at this point  the general member $Z_X\in|\pi^{\ast}Z_Y|$ has a double point. Let $R\in X$ be a base point of $|\pi^{\ast}Z_Y|$ and let $f_1 \colon X_1 \rightarrow X$ be the blow up of $R$. let $Z_i \in|\pi^{\ast}Z_Y|$, $i=1,2$, be two distinct general members. Then $f_1^{\ast}Z_i=Z_i^{\prime}+m_i E$, where $E$ is the $f_1$-exceptional curve and $Z_i^{\prime}$ the birational transform of $Z_i$ in $X_1$, $i=1,2$. Then intersecting with $Z_2$, it follows that 
\[
4=Z_2^2=Z^{\prime}_1Z^{\prime}_2 +m_1m_2.
\]
Since $m_i\geq 2$, it follows that the only possibility is that $m_1=m_2=2$ and $  Z^{\prime}_1Z^{\prime}_2=0$. Hence $|\pi^{\ast}Z_Y|$ has exactly one base point which is resolved after blowing up $R$. Hence there exists a map  $\sigma \colon X_1 \rightarrow \mathbb{P}^1$ whose general fiber is the birational transform $Z_{X_1}$ of the general member $Z_X\in |\pi^{\ast}Z_Y|$. If this is a fibration, i.e., $\sigma_{\ast}\mathcal{O}_{\bar{X}}=\mathcal{O}_{\mathbb{P}^1}$, then the general fiber of $\sigma$ is a normal integral curve. If it is not a fibration, then let $\tau \colon X_1 \rightarrow B$ be its Stein factorization. Suppose that the map $B \rightarrow \mathbb{P}^1$ is not purely inseparable. Then the general fiber of $\sigma$ is disconnected. If $ B \rightarrow \mathbb{P}^1$ is purely inseparable, then $Z_X$ has a component of multiplicity $p^k$, for some $k>0$. But since $K_X \cdot Z_X=2$ and $p\not=2$ this is impossible. Hence the birational transform $Z_{X_1}$ of the general $Z_X\in |\pi^{\ast}K_Y|$ in $X_1$ is either an integral curve or (since $Z_X$ has at most two irreducible components) is the disjoint union of two integral curves. 

Suppose that the birational transform $Z_{X_1}$ of the general $Z_X\in|\pi^{\ast}Z_Y|$ is an integral curve. Then 
 \[
p_a(Z^{\prime}_X)=\frac{1}{2}\left(K_{X_1}\cdot Z^{\prime}_X+(Z^{\prime}_X\right)^2)+1=\frac{1}{2}(2+2+0)+1=3.
\]
Suppose that $Z_{X_1}$ is smooth. As was proved earlier, $D$ lifts to a vector field $D_1$ on $X_1$. Since $Z_X$ is an integral curve of $D$ it follows that $X_1$ is an integral curve of $D_1$. Therefore $D_1$ restricts to a vector field on $Z_{X_1}$. But smooth curves of genus $g\geq 2$ do not have any nontrivial global vector fields. Therefore the restriction of $D_1$ on $Z_{X_1}$ is zero. Since there are infinitely many such $Z_{X_1}$ it follows that $D_1=0$ and hence $D=0$.

Suppose that $Z_{X_1}$ is singular. The general curve $C=Z_{X_1}$ is a normal integral curve defined over $k(t)$. Since it is not smooth, there exists a purely inseparable extension $k(t)\subset K$ of $k(t)$ such that $C_K=C \otimes_{k(t)}K$ is not normal. Let $\tilde{C}_K$ be its normalization. Then by~\cite{Sch09}, the difference $p_a(C_{k(t)})-p_a(\tilde{C}_K)$ is divisible by $(p-1)/2$. But $p_a(C)=p_a(Z^{\prime}_X)=3$. Therefore $p \in\{2,3,5,7\}$.  

Suppose that the birational transform $Z_{X_1}$ of the general $Z_X\in|\pi^{\ast}Z_Y|$ is the disjoint union of two integral curves. In this case I will show that both components are integral curves of arithmetic genus 2. Then $D_1$ restricts to a vector field on each component. Hence if at least one component is smooth then the restriction of $D_1$ on this component is zero and since there are infinitely such curves we conclude that $D_1=0$ and hence $D=0$. If on the other hand both components are singular then arguing as in the previous case when $Z_{X_1}$ was integral, we conclude that $p\in \{2,3,5\}$.

Since $Z_{X_1}$ is disconnected, it follows that, since $K_X\cdot Z_X=2$, that $Z_X$ has exactly two irreducible components, say $Z_1, Z_2$ and $K_X\cdot Z_i=1$, $i=1,2$. 

Suppose that $R \not\in  Z_1\cap Z_2$. Then in fact $Z_X$ is the disjoint union of $Z_1$ and $Z_2$. I will show that this case is impossible. A consequence of equation~\ref{sec-all-p-eq-23} was that $K_X \cdot Z_X=2$, $Z_X^2=4$ and $Z_X\cdot \Delta =2$. Suppose that $Z_X=Z_1+Z_2$ with $Z_1\cdot Z_2=0$. Then $Z_1\cdot \Delta +Z_2\cdot \Delta =2$. Suppose that $Z_1 \cdot \Delta =0$. Then from the equation~\ref{sec-all-p-eq-23} it follows that $Z_1^2=p+1 $ and hence $Z_2^2=3-p$. But then 
\[
p_a(Z_2)=\frac{1}{2}(K_X\cdot Z_1+Z_1^2)+1=\frac{1}{2}(4-p)+1 \not\in \mathbb{Z},
\]
if $p\not= 2$. Suppose that $Z_1\cdot \Delta =1$. Then from the equation~\ref{sec-all-p-eq-23} it follows that $Z_1^2=2$ and hence
\[
p_a(Z_1)=\frac{1}{2}(K_X\cdot Z_1+Z_1^2)+1=\frac{1}{2}(1+2)+1 \not\in \mathbb{Z}.
\]
Finally, if $Z_1\cdot \Delta =2$ then $Z_2 \cdot \Delta =0$. But this case has been studied first and it was shown to be impossible. Therefore $R\in Z_1 \cap Z_2$ and in fact since the birational transform of $Z_X$ in $X_1$ is disconnected, it follows that $R=Z_1 \cap Z_2$. Moreover, it is easy to see that $Z_1$ and $Z_2$ are smooth at $R$ and $Z_1\cdot Z_2 =1$. Then similar arguments as before show that the only possibility is $Z_1^2=Z_2^2=1$. Since $K_X\cdot Z_i=1$ it follows  that $p_a(Z_i)=2$, $i=1,2$. Now since $R$ is a smooth point of $Z_i$ it follows that $h_{\ast}^{-1}Z_i\cong Z_i$ and therefore every connected component of $Z_{X_1}$ is an integral curve of arithmetic genus 2, as claimed earlier. 

It remains to consider the case when the general member $Z_X\in |\pi^{\ast}Z_Y|$ is smooth at $R$. This case can be studied in exactly the same way as the case when $Z_X$ was singular at $R$. Let $f_1\colon X_1 \rightarrow X$ be the blow up of $R$. Then the birational transforms $Z_{X_1}$ of $Z_X$ in $X_1$ form a two-dimensional linear system in $X_1$, $Z_{X_1}\cong Z_X$ and $D$ lifts to a vector field $D_1$ on $X_1$. The previous argument can be used in this case now in order to reach the same conclusion, i.e., that $p\in\{2,3,5,7\}$.

The final case that needs to be considered is when $|\pi^{\ast}Z_Y|$ has a base component, say $Z_0$. If $Z_0\not= \Delta$, then the conclusion from the equation~\ref{sec-all-p-eq-23} still holds and the whole argument that was used in the case when $Z_X$ was reducible applies in this case too. Suppose that $Z_0=\Delta$. Then $Z_X=Z+Z_0$, where $Z$ is an integral curve. Then the curves $Z$ form a 2-dimensional linear system without base components. Now the equation~\ref{sec-all-p-eq-23} gives that $(p+1)K_X=Z+p\Delta$ and therefore since $K_X\cdot Z=1$, it follows that $Z^2=1$ and $Z\cdot \Delta =1$. Hence $Z$ is an integral curve of arithmetic genus 2. 
In the beginning of this proof it was shown that the linear system $|K_X|$ has exactly one base point and that the base point is a smooth point of ever member $C \in |K_X|$. The proof of this was based only on the fact that for any $C\in |K_X|$, $C^2=K_X\cdot C=1$.  Since $Z^2=K_X\cdot Z=1$, the same argument shows that the linear system $|Z|$ has exactly one base point $R$ and each member of $|Z|$ is smooth at $R$.  After blowing up $R$ the birational transforms $Z_1$ of $Z$ in the blow up $X_1 \rightarrow X$ of $X$ at $R$ form a base point free linear system and hence define a map $X_1\rightarrow \mathbb{P}^1$, which this time, since $Z$ is integral, is a fibration of curves of arithmetic genus 2. Moreover, the same arguments as in the case when $|\pi^{\ast}Z_Y|$ had no base component show that the general $Z$ is an integral curve of $D$ and $D$ lifts to a vector field $D_1$ on $X_1$. Now if the general member of $|Z|$ is smooth then the restriction of $D_1$ to it is zero and hence again $D_1=0$ and hence $D=0$. If the general member of $|Z|$ is singular then $p\in\{2,3,5\}$.

\begin{remark}
The reason that the characteristics $p=2,3,5,7$ have been excluded is that for these characteristics there are fibrations $\phi \colon X \rightarrow \mathbb{P}^1$ such that the general fiber is a singular curve of arithmetic genus at most 3. In particular in characteristic 2 there are quasi-elliptic fibrations. However, the statement of the proposition is true in characteristic 2, something that will be proved in section~\ref{sec-5}. I believe that the method that was used in the case of characteristic 2 can in principle be used for $p=3,5,7$ as well. However, I am having  a few technical problems to extend it for $p=3,5,7$ and anyway the paper is already too long.
\end{remark}

\end{proof}

%% file: sec6.tex

\section{Vector fields on surfaces in characteristic 2.}

The purpose of this and the following two sections is to study smooth canonically polarized surfaces defined over an algebraically closed field of characteristic 2 with non trivial global vector fields. Even though 2 is the smallest possible nonzero characteristic where many unusual situations appear (like the existence of quasi-elliptic fibrations) and in the view of Theorem~\ref{th1} this is the characteristic that it is most likely that the automorphism scheme of a surface with $K^2=1,2$ is not reduced, there are certain advantages over the higher characteristic case, as exhibited in Proposition~\ref{prop4}. 

The study of this case will be divided in two sub-cases. The case of surfaces with vector fields of multiplicative type and the case of surfaces with vector fields of additive type. This equivalent to the property that $\mu_2$ or $\alpha_2$ is a subgroup scheme of the automorphism scheme of the surface.

The following theorem is slightly stronger than Theorem~\ref{main-theorem} and is the combination of 
Theorems~\ref{mult-type},~\ref{additive-type}. 

\begin{theorem}\label{char-2}
Let $X$ be a smooth canonically polarized surface defined over an algebraically closed field of characteristic 2. Suppose that $\mathrm{Aut}(X)$ is not smooth. Then:
\begin{enumerate}
\item If $K_X^2=2$ then  $X$ is uniruled. If in addition $\chi(\mathcal{O}_X)\geq 2$, then $X$ is unirational and $\pi_1^{et}(X)=\{1\}$.
\item If $K_X^2=1$ then $\pi_1^{et}(X)=\{1\}$, $p_g(X) \leq 1$ and $X$ is unirational. 
\item $X$ is the quotient of a ruled or rational surface (maybe singular) by a rational vector field.\end{enumerate}
Moreover, suppose that $\mu_2$ is a subgroup scheme of $\mathrm{Aut}(X)$ and  $K_X^2 \leq 4$. Then $X$ is uniruled. In particular, if $K_X^2=1$ then $X$ is a simply connected supersingular Godeaux surface.
\end{theorem}

\begin{remarks}
\begin{enumerate}
\item According to Corollary~\ref{subgroup-of-aut}, if $\mathrm{Aut}(X)$ is not smooth then it contains either $\mu_2$ or $\alpha_2$. Theorem~\ref{char-2} shows that the condition $\mu_2$ is a subgroup scheme of  $\mathrm{Aut}(X)$ is more restrictive than $\alpha_2$ is a subgroup scheme of  $\mathrm{Aut}(X)$. Equivalently, it is more rare that a surface has vector fields of multiplicative type than of additive type.
\item The proof of the statement of the theorem in the case when $\mu_2$ is a subgroup scheme of $\mathrm{Aut}(X)$ strongly uses the fact that the quotient map $\pi \colon X \rightarrow Y$ is a torsor in codimension 2, something that is not true in general for higher characteristics. This is the reason that I am unable at the time to generalize it in all characteristics.
\end{enumerate}
\end{remarks}

%% file: sec7.tex
\section{Surfaces with vector fields of multiplicative type in characteristic 2.}\label{sec-4}
The purpose of this section is to study smooth canonically polarized surfaces defined over an algebraically closed field of characteristic 2 which admit nontrivial global vector fields of multiplicative type. As it was shown in section~\ref{sec-2} this is equivalent to the condition that $\mu_2$ is a subgroup scheme of $\mathrm{Aut}(X)$. The main result is the following.

\begin{theorem}\label{mult-type}
Let $X$ be a smooth canonically polarized surface defined over an algebraically closed field $k$ of characteristic 2. Suppose that $X$ has a nontrivial global vector field $D$ of multiplicative type. Then:
\begin{enumerate}
\item If $K_X^2=4$ then $X$ is uniruled and $-2\leq \chi(\mathcal{O}_X)\leq 2$.
\item If $K_X^2=3$ then  $X$ is uniruled and $-1\leq \chi(\mathcal{O}_X)\leq 1$.
\item If $K_X^2=2$ then $X$ is uniruled and $0\leq \chi(\mathcal{O}_X)\leq 1$.
\item If $K_X^2=1$, $X$ is an algebraically simply connected unirational supersingular Godeaux surface. 
\end{enumerate}
Moreover, $X$ is an inseparable quotient of degree 2 of a rational or ruled surface (possible singular) by a rational vector field. 
\end{theorem}

\begin{corollary}\label{sm-multi-type}
Let $X$ be a smooth canonically polarized surface defined over an algebraically closed field $k$ of characteristic 2. Suppose $K_X^2 <5$ and that one of the following happens.
\begin{enumerate}
\item $X$ is not uniruled.
\item $K_X^2\in\{2,3\}$ and $\chi(\mathcal{O}_X)\geq 2$, or $K_X^2=4$ and $\chi(\mathcal{O}_X)\geq 3$.
\item $K_X^2=1$ and either
\begin{enumerate}
\item $X$ is not simply connected, i.e., $\pi_1^{et}(X)\not=\{1\}$.\item $\chi(\mathcal{O}_X) \geq 2$, or
\item $\chi(\mathcal{O}_X) =1$ and $X$ is either a classical or singular Godeaux surface.
\end{enumerate}
\end{enumerate}
Then $X$ does not have any nontrivial global vector fields of multiplicative type. Equivalently $\mathrm{Aut}(X)$ does not have a subgroup scheme isomorphic to $\mu_2$.
\end{corollary}

For the rest of this section fix terminology and notation as in Section~\ref{preparation}.

\begin{proof}[Proof of Theorem~\ref{mult-type}]
Let $\pi \colon X\rightarrow Y$ be the quotient of $X$ by the $\mu_2$-action induced by $D$. Since the characteristic of the base field is 2, it follows from Proposition~\ref{prop4}, that in addition to all the properties stated in Section~\ref{preparation} that $Y$ has, the following is also true:
\begin{enumerate}
\item The singularities of $Y$ are isolated surface singularities locally isomorphic to $xy+z^2=0$. In particular, $Y$ has Gorenstein canonical singularities of type $A_1$. Hence $K_Y$ is Cartier and moreover, since $g \colon Y^{\prime} \rightarrow Y$ is the minimal resolution of $Y$, $g$ is crepant, i.e., $K_{Y^{\prime}}=g^{\ast} K_Y $.
\item $X^{\prime}$ is smooth and $f$ is obtained by successively blowing up the isolated fixed points of $D$. In particular the lifting $D^{\prime}$ of $D$ in $X^{\prime}$ has only divisorial singularities.
\item The divisorial part $\Delta$ of the fixed locus of $D$ is smooth (perhaps disconnected), disjoint from the isolated singular points of $D$ and not an integral divisor of $D$. Therefore if $\Delta^{\prime}$ is the image of $\Delta$ in $Y$, $\Delta^{\prime}$ is in the smooth part of $Y$, $\pi_{\ast}\Delta =\Delta^{\prime}$ and $\pi^{\ast}\Delta^{\prime}=2\Delta$.
\item From Proposition~\ref{structure-of-pi}, $\pi$ is a torsor over a codimension 2 open subset of $Y$. In particular, there is a reflexive sheaf $L=\mathcal{O}_Y(C)$ on $Y$ such that $X=\mathrm{Spec}_Y \left( \mathcal{O}_Y \oplus L^{-1} \right)$, and 
\begin{gather}\label{sec4-eq2}
K_X=\pi^{\ast}(K_Y+C).
\end{gather}
From this it follows, since $K_X$ is ample and $\pi$ a finite morphism, that $K_Y+C$ is ample.
\end{enumerate}

Since $p=2$, the adjunction formula that appeared in equation~\ref{sec-all-p-eq-1} becomes
\begin{gather}\label{sec4-eq1}
K_X=\pi^{\ast}K_Y+\Delta.
\end{gather}
From this and the equation~\ref{sec4-eq2} it follows that $\pi^{\ast}C=\Delta$. Then since $\pi_{\ast}\Delta = \Delta^{\prime}$ it follows that $\Delta^{\prime} \sim 2C$.  Finally, from~\ref{sec4-eq2} and the fact that $\pi$ is finite of degree 2 it follows that
\begin{gather}\label{sec4-eq3}
K_X^2=2(K_Y+C)^2=2K_Y\cdot (K_Y+C) +2C\cdot (K_Y+C).
\end{gather}

\textbf{Claim: If $\Delta\not= 0$ then $K_X^2\geq 8$}. Therefore, unlike the case when $p\geq 3$, the possibility that $\Delta=0$ does not create any problems here.

I proceed to prove the claim. Suppose that $\Delta=0$. This implies that the fixed locus of $D$ does not have a divisorial part. This case can only happen if $k(Y)=2$. Indeed, since $\Delta=0$ it follows that $K_X=\pi^{\ast}K_Y$ and therefore since $\pi$ is finite, $K_Y$ is ample. Moreover, since $Y$ has singularities of type $A_1$, $K_{Y^{\prime}}=g^{\ast}K_Y$. Hence $K_{Y^{\prime}}$ is nef and big and hence $k(Y^{\prime})=2$.

Note that even though $D$ has no divisorial part it may  however have isolated singular points. Then from Proposition~\ref{size-of-sing}, $Y$ has exactly $c_2(X)$ singular points. If $c_2(x)=0$ then $K_X^2=12\chi(\mathcal{O}_X) \geq 12$. Suppose that $c_2(X)>0$. Then $c_2(X)=\chi_{et}(Y)$. Moreover,
\[
\chi_{et}(Y^{\prime})=\chi_{et}(Y)+\sum(\chi_{et}(E_i)-1)
\]
where $E_i$ are the reduced connected components of the $g$-exceptional locus. Since $Y$ has singularities of type $xy+z^2=0$, the $g$-exceptional curves are exactly $c_2(X)$ number isolated $(-2)$-curves. Hence $\chi_{et}(Y^{\prime})=2c_2(X)$.

Since $K_X=\pi^{\ast}K_Y$ it follows that $K_X^2=2K_Y^2=2K_{Y^{\prime}}^2$=2d. Then from Noether's formula $c_2(X)=12\chi(\mathcal{O}_X)-2d$. Moreover, again from Noether's formula on $Y^{\prime}$,  
\[
12\chi(\mathcal{O}_{Y^{\prime}}) =d+c_2(Y^{\prime})=d+2c_2(X)=-3d+24\chi(\mathcal{O}_X).
\]
 However, this relation is impossible for $d\leq 3$,or equivalently if $K_X^2<8$. Hence in the following, since the theorem only studies the cases $K_X^2 \leq 5$, we can assume that $\Delta \not= 0$.

The proof  of Thorem~\ref{mult-type} will be by considering cases with respect to the Kodaira dimension $\kappa(Y)$ of $Y$.

\textbf{Case 1.} Suppose $k(Y^{\prime})=2$. In this case I will show that $K_Y^2 \geq 5$.

In this case take $Z$ in the commutative diagram~\ref{sec-all-p-diagram-1} not to be the minimal model of $Y^{\prime}$ but its  canonical model instead.  Canonical models of smooth surfaces exist in any 
characteristic by~\cite{Art62}. 

Suppose that $Y^{\prime}$ is not a minimal surface. Then $B\not= 0$ and moreover $g_{\ast}B\not=0$. Indeed. Since $\phi$ is a composition of blow ups it follows that $B$ contains $-1$ curves. However,  $Y^{\prime}$ is the minimal resolution of $Y$ and hence $g$ does not contract $-1$ curves. Therefore $g_{\ast}B\not=0$. Moreover, since $ 2C\sim \Delta^{\prime}$ which is Cartier and effective, and $K_Y+C$ is ample, it follows that 
\begin{gather}\label{sec4-eq4}
2C\cdot (K_Y+C)\geq 1.
\end{gather}
 Also,
\begin{gather}\label{sec4-eq5}
K_Y\cdot (K_Y+C)=g^{\ast}K_Y \cdot g^{\ast}(K_Y+C)=\\
K_{Y^{\prime}}\cdot g^{\ast}(K_Y+C)=(\phi^{\ast}K_Z +B)\cdot g^{\ast}(K_Y+C).\nonumber
\end{gather}

Then $B\cdot g^{\ast}(K_Y+C)=g_{\ast}B \cdot (K_Y+C)>0$, since $K_Y+C$ is ample and $g_{\ast}B\not= 0$. 
Moreover, $\phi^{\ast}K_Z \cdot g^{\ast}(K_Y+C) >0$. Indeed, take $n>>0$ such that $n(K_Y+C)$ is very ample. Then there is a divisor $H\in |n(K_Y+C)|$ such that $g^{\ast}H$ is not contracted by $\phi$. Since $K_Z$ is ample then $\phi^{\ast}K_Z \cdot g^{\ast}(K_Y+C) >0$, as claimed. Hence $K_Y\cdot (K_Y+C) \geq 2$. Then from~\ref{sec4-eq3} we get that $K_X^2 \geq 5$, as claimed.

Suppose that $Y^{\prime}$ is a minimal surface and hence $B=0$. $K_Y^2=K_{Y^{\prime}}^2\geq 1$. Moreover $K_Y\cdot C=1/2(K_Y \cdot \Delta^{\prime} )\geq 0$. Suppose that $K_Y \cdot \Delta^{\prime}=0$. Since $\Delta^{\prime}$ is in the smooth part of $Y$ it follows that $\Delta^{\prime\prime}=g^{\ast}\Delta^{\prime}\cong \Delta^{\prime}$. Then 
$K_{Y^{\prime}}\cdot \Delta^{\prime\prime}=0$. Since $Y^{\prime}$ is minimal of general type it follows that every irreducible component of $\Delta^{\prime\prime}$ is a smooth rational curve of self intersection $-2$. Therefore the same holds for $\Delta^{\prime}$. Let now $W^{\prime}$ be an irreducible component of $\Delta^{\prime}$. Then $\pi^{\ast}W^{\prime}=2W$, where $W$ is an irreducible component of $\Delta$. Then $4W^2=2(W^{\prime})^2=-4$ and hence $W^2=-1$. But then $K_X\cdot W=-1 <0$, which is impossible since $K_X$ is ample. Therefore $K_Y \cdot C >0$, and in fact $\geq 1$ since $K_Y$ is Cartier. Then again from equation~\ref{sec4-eq3} it follows that $K_X^2\geq 5$.

\textbf{Case 2.} Suppose $k(Y^{\prime})=1$. In this case I will show again that $K_Y^2 \geq 5$.

The minimal model $Z$ of $Y^{\prime}$ is then a minimal surface of Kodaira dimension 1. Then by the classification of surfaces~\cite{BM76}~\cite{BM77}~\cite{Ba01}, $Z$ admits either an elliptic or quasi-elliptic fibration, i.e., there is a fibration $\psi \colon Z \rightarrow B$, where $B$ is a smooth curve and the fibers of $\psi$ have arithmetic genus 1. Also from the classification of surfaces, there is $n>>0$ such that $nK_Z=\psi^{\ast}(W)$, where $W$ is a positive divisor in $B$.

Next I will show that $\phi^{\ast}K_Z\cdot F =0$, where $F$ is any $g$-exceptional curve. Indeed,
\[
\phi^{\ast}K_Z \cdot F=K_{Y^{\prime}}\cdot F -B\cdot F=g^{\ast}K_Y \cdot F -B\cdot F=-B\cdot F.
\]
Since $K_Z$ is nef, it follows that $\phi^{\ast}K_Z \cdot F=K_Z\cdot \phi_{\ast}F \geq 0$. Therefore, $F\cdot B \leq 0$. If it is strictly negative, then $F\subset B$ and hence $F$ is $\phi$-exceptional. But then $\phi^{\ast}K_Z\cdot F=0$. Hence in any case $\phi^{\ast}K_Z \cdot F=0$ for any $g$-exceptional curve. Therefore $\phi^{\ast}K_Z=g^{\ast}H$, where $H$ is a Cartier divisor on $Y$. Moreover, since $nK_Z$ is positive for large enough $n$, it follows that $nH$ is positive too. Therefore from the equation~\ref{sec-all-p-eq-2} it follows that
\[
K_Y=H+g_{\ast}B.
\]
Since $K_Y$ and $H$ are Cartier, $g_{\ast}B$ is an effective Cartier divisor as well. Then from~\ref{sec4-eq2} it follows that
\begin{gather}\label{sec4-eq5}
K_X^2=2(K_Y+C)^2=2K_Y\cdot (K_Y+C) +2C\cdot (K_Y+C)=\\ 
2H \cdot (K_Y+C)+2g_{\ast}B\cdot (K_Y+C) +2C\cdot(K_Y+C).\nonumber
\end{gather}
Suppose that $g_{\ast}B\not= 0$, i.e., $Y^{\prime}$ is not a minimal surface. Then, since $2C$ is Cartier and equivalent to an effective divisor, the above equation implies that $K_X^2 \geq 5$, as claimed.

Suppose that $g_{\ast}B=0$, hence $Y^{\prime}$ is a minimal surface. Then in any case, equation~\ref{sec4-eq5} shows that $K_X^2 \geq 3$ and $K_X^2 \geq 5$ unless $K_Y \cdot (K_Y+C)=1$ and $C \cdot (K_Y+C)=1/2$. Suppose that this is the case. Then considering that $\pi^{\ast}C=\Delta$ and that $K_Y^2=K_{Y^{\prime}}^2=0$, it follows that $K_X\cdot \Delta =1$ and $\Delta^2 =-1$. 

Let $N_{fix}$ be the number of isolated fixed points of $D$. Then by Corollary~\ref{no-of-sing-of-quot}, 
\[
N_{fix}=K_X\cdot \Delta +\Delta^2 +c_2(X)=c_2(X).
\]
Hence $Y$ has exactly $c_2(X)$ singular points, all of them of type $A_1$. Such singularities are resolved by a single blow up. Hence $f$ is the composition of $c_2(X)$ blow ups. Hence from diagram~\ref{sec-all-p-diagram-1} it follows that
\[
c_2(Y^{\prime})=c_2(X^{\prime})=c_2(X)+c_2(X)=2c_2(X).
\]
Noether's formula (true also in characteristic 2) gives that
\[
12\chi(\mathcal{O}_{Y^{\prime}})=K_{Y^{\prime}}^2+c_2(Y^{\prime})=2c_2(X)
\]
and therefore $c_2(X)=6d$. However, Noether's formula for $X$ gives that
\[
12\chi(\mathcal{O}_X) =K_X^2 +c_2(X)=3+6d,
\]
which is clearly impossible.

\textbf{Case 2.} Suppose $k(Y^{\prime})=0$. In this case I will show that;
\begin{enumerate}
\item If $K_X^2 \leq 4$, then $X$ is algebraically simply connected, unirational and $\chi(\mathcal{O}_X)=1$. Moreover, $X$ is an inseparable quotient of degree 2 of a rational or ruled surface (possible singular) by a rational vector field. 
\item If $K_X^2=1$, then $X$ is an algebraically simply connected unirational supersingular Godeaux surface. 
\end{enumerate}

Suppose that $B\not=0$, i.e., $Y^{\prime}$ is not minimal. Its minimal model $Z$ is a minimal surface of Kodaira dimension zero. Therefore $12K_Z=0$. Hence 
\[
K_{Y^{\prime}}=\phi^{\ast}K_Z+B \equiv B
\]
and therefore $K_Y\equiv g_{\ast}B=B^{\prime}$, where $B^{\prime}$ is a nonzero effective Cartier divisor. Hence \[
K_Y\cdot (K_Y+C)=B^{\prime}\cdot (K_Y+C) \geq 1
\]
since $ K_Y+C$ is ample. Moreover, as explained in Case 1, $2C\cdot (K_Y+C)\geq 1$. Therefore from~\ref{sec4-eq2} we get that 
\[
K_X^2=2(K_Y+C)^2=2K_Y\cdot (K_Y+C) +2C\cdot (K_Y+C)\geq 3.
\]
In fact, $K_X^2 \geq 5$ unless 
\begin{gather}\label{sec4-eq6}
K_Y\cdot (K_Y+C)=1\\
C\cdot (K_Y+C)=1/2,\nonumber
\end{gather}
in which case $K_X^2=3$. I will show that this case is impossible. Suppose that this is the case. Then 

\textbf{Claim 1.} 
\begin{enumerate}
\item $\Delta \cong \mathbb{P}^1$.
\item $\Delta^2=-3$, $K_X \cdot \Delta = 1$ and $K_{Y^{\prime}}^2=-1$.
\end{enumerate}
Indeed. From~\ref{sec4-eq6} it follows that 
\[
K_X \cdot \Delta =\pi^{\ast}(K_Y+C)\cdot \pi^{\ast}C=2(K_Y+C)\cdot C =1.
\]
Therefore, since $K_X$ is ample, $\Delta$ is an irreducible and reduced curve. On the other hand the equation $K_Y\cdot (K_Y+C)=1$ gives that
\[
\pi^{\ast}K_Y \cdot \Delta = 2K_Y \cdot C=2(1-K_Y^2)=2(1-K_{Y^{\prime}}^2)=2(1-B^2)\geq 4,
\]
since $B^2<0$. Now again from~\ref{sec4-eq1} we get that
\[
\Delta^2=1-\Delta\cdot \pi^{\ast}K_Y \leq -3.
\]
However the genus formula for $\Delta$ gives that
\[
\mathrm{p}_a(\Delta)=\frac{1}{2}(2+K_X\cdot \Delta+\Delta^2)=\frac{1}{2}(3+\Delta^2)\leq 0,
\]
since $\Delta^2 \geq -3$. This implies that $\Delta^2=-3$ and $\mathrm{p}_a(\Delta)=0$. Hence $\Delta \cong \mathbb{P}^1$. Finally from the relation $K_X=\pi^{\ast}K_Y+\Delta$ it follows that \[
K_Y^2=\frac{1}{2}(K_X-\Delta)^2=K_X^2+\Delta^2-2K_X\cdot \Delta =-1,
\]
and the claim is proved.

\textbf{Claim 2.} The map $\phi$ is a single blow up. 

Indeed. From Claim 1 it follows that $K_{Y^{\prime}}^2=K_Y^2=-1$. On the other hand, $K_Z^2=0$ and $\phi$ is a composition of blow ups. Considering that $K^2$ is reduced by 1 after every blow up it follows that $\phi$ is a single blow up.

Let $N_{fix}$ be the number of isolated singular points of $D$. Then from Corollary~\ref{no-of-sing-of-quot} and Claim 1,
\[
N_{fix}=K_X\cdot \Delta +\Delta^2+c_2(X)=-2+c_2(X).
\]
$Y^{\prime}$ is the minimal resolution of $Y$ which has exactly $N_{fix}$ singular points, all of type $A_1$. Hence the $g$-exceptional curves are  exactly $N_{fix}$ isolated $-2$ curves. Then from diagram~\ref{sec-all-p-diagram-1} and Claim 2 we get that \[
c_2(Z)=c_2(Y^{\prime})-1=\chi_{et}(Y^{\prime})+N_{fix}-1=c_2(X)+N_{fix}-1=2c_2(X)-3.
\]
By the classification of surfaces~\cite{BM76}~\cite{BM77}, $c_2(Z)\in \{0,12,24\}$. Hence $2c_2(X)\in\{ 3, 15, 27\}$, which is clearly impossible. Hence the case $K_X^2=3$ is impossible and therefore, if $Y^{\prime}$ is not minimal, then $K_X^2 \geq 5$.

Assume now that $Y^{\prime}$ is minimal. Then $K_{Y^{\prime}}\equiv 0$ and hence $K_Y \equiv 0$ which implies that $K_X\equiv \Delta$. Then 
\begin{gather}\label{sec4-eq8}
N_{fix}=K_X\cdot \Delta +\Delta^2 +c_2(X)=2K_X^2+c_2(X).
\end{gather}
This implies that $Y$ is singular. Indeed. If $Y$ was smooth, then $Y=Y^{\prime}$ and $c_2(X)=c_2(Y^{\prime}) \in\{0,12,24\}$, since $Y^{\prime}$ is a minimal surface of Kodaira dimension zero. In particular, $c_2(X)\geq 0$ and hence $N_{fix}>0$. Hence $Y$ is singular and so $Y^{\prime}\not= Y$. Moreover,
\begin{gather}\label{sec4-eq7}
c_2(Y^{\prime})=\chi_{et}(Y)+N_{fix}=c_2(X)+N_{fix}=2(K_X^2+c_2(X))=24\chi(\mathcal{O}_X),
\end{gather}
from Noether's formula. Now since $c_2(Y^{\prime})\in \{0,12,24\}$ it follows from~\ref{sec4-eq7} that $c_2(Y^{\prime})=0$ or $24$ and consequently $\chi(\mathcal{O}_X)=0$ or $1$. Hence $Y^{\prime}$ is either an Abelian, a K3 or a quasi-hyperelliptic surface~\cite{BM76}~\cite{BM77}. I will show that $Y^{\prime}$ is a K3 surface. Suppose not. Then $\chi(\mathcal{O}_X)=0$, $c_2(Y^{\prime})=0$ and $Y^{\prime}$ is either an Abelian surface or a quasi-hyperelliptic surface. Since $\chi(\mathcal{O}_X)=0$, Noether's formula gives that $c_2(X)=-K_X^2 <0$, since $K_X$ is ample. Hence from~\cite{SB91} it follows that $X$ is uniruled and therefore so is $X^{\prime}$. But a uniruled surface cannot dominate an Abelian surface. Hence this case is impossible and therefore $Y^{\prime}$ is a quasi-hyperelliptic surface. Then there exists an elliptic or quasi-hyperelliptic fibration $\Phi \colon Y^{\prime}\rightarrow E$, where $E=\mathrm{Alb}(Y^{\prime})$ is a smooth elliptic curve. Moreover every fiber or $\Phi$ is 
irreducible~\cite{BM76}. if $Y$ is singular then the $g$-exceptional curves are isolated smooth rational $-2$ curves. Since $E$ is elliptic, every $g$-exceptional curve must contract to a point by $\Phi$. But this is impossible because every fiber of $\Phi$ is an irreducible curve or arithmetic genus 1. Hence $Y$ is a K3 surface. Therefore taking into consideration that the \'etale fundamental groupd is a birational invariant and that $\pi^{\prime}$ gives an equivalence between the \'etale sites of $X^{\prime}$ and $Y^{\prime}$ we get that 
\[
\pi_1^{et}(X)=\pi_1^{et}(X^{\prime})=\pi_1^{et}(Y^{\prime})=\{1\}
\]
and therefore $X$ is algebraically simply connected.

Next I will show that $Y^{\prime}$ is unirational. In order to show this I will show that $Y$ has at least 13 singular points of type $A_1$. Then, $Y^{\prime}$ has at least 13 isolated $-2$ curves and hence in the terminology of~\cite{SB96}, $Y^{\prime}$ has a special configuration $E$ of rank at least 13. Therefore it is unirational~\cite{SB96}. Considering that $\chi(\mathcal{O}_X)=1$, from~\ref{sec4-eq8} we get that
\[
N_{fix}=2K_X^2+c_2(X)=K_X^2+12\chi(\mathcal{O}_X)=K_X^2+12\geq 13,
\]
since $K_X^2>0$. Hence $Y$ has at least 13 singular points (all necessarily of type $A_1$), as claimed. Therefore $Y^{\prime}$ and hence $Y$ is unirational. Now the same argument that was used in the proof of 
Theorem~\ref{sec-all-p-prop-1} shows that $X$ is unirational and that it is the quotient of a rational surface by a rational vector field.

Suppose now that $K_X^2=1$. Since we already proved that $\chi(\mathcal{O}_X)=1$, $X$ is a numerical Godeaux surface. It remains to show that $X$ is a supersingular Godeaux. This means that $h^1(\mathcal{O}_X)=1$ and map $F^{\ast} \colon H^1(\mathcal{O}_X) \rightarrow H^1(\mathcal{O}_X)$ induced by the Frobenious is zero. Since $\chi(\mathcal{O}_X)=1$ it follows that $h^1(\mathcal{O}_X)=p_g(X)=h^0(\omega_X)$. However, since $Y^{\prime}$ is a K3 surface, $\omega_{Y^{\prime}}\cong \mathcal{O}_{Y^{\prime}}$ and therefore $\omega_Y\cong \mathcal{O}_Y$. Then from~\ref{sec4-eq2} it follows that $\omega_X=(\pi^{\ast}L)^{\ast\ast}$, and therefore 
\[
H^0(\omega_X)=H^0(\pi_{\ast}(\pi^{\ast}L)^{\ast\ast})=H^0((L\otimes (\mathcal{O}_Y\oplus L^{-1})^{\ast\ast})=H^0(\mathcal{O}_Y\oplus L)\not=0.
\]
Hence $X$ is either a singular or a supersingular Godeaux surface. If $X$ was singular, $F^{\ast}$ is injective and the exact sequence~\cite[Page 127]{Mi80}
\[
0 \rightarrow \mathbb{Z}/2\mathbb{Z} \rightarrow \mathcal{O}_X \stackrel{F-1}{\rightarrow} \mathcal{O}_X \rightarrow 0
\]
gives that $H^1_{et}(X,\mathbb{Z}/2\mathbb{Z})\not= 0$. Hence $X$ has \'etale $2$-covers. But this is impossible since $\pi_1^{et}(X)=\{1\}$. Hence $X$ is a supersingular Godeaux surface.

\textbf{Case 4.} Suppose that $k(Y^{\prime})=-\infty$. In this case I will show that $K_X^2 \geq 2$, $X$ is uniruled and moreover, 
\begin{enumerate}
\item If $K_X^2=4$, then $-2\leq \chi(\mathcal{O}_X)\leq 2$.
\item If $K_X^2=3$, then $-1\leq \chi(\mathcal{O}_X)\leq 1$.
\item If $K_X^2=2$, then $0\leq \chi(\mathcal{O}_X)\leq 1$.
\end{enumerate}

In order to prove this the following simple result is needed.

\begin{lemma}\label{cor-b1}
Let $X$ be a smooth surface of general type defined over an algebraically closed field $k$ such that one of the following happens.
\begin{enumerate}
\item $K_X^2=1$.
\item $K_X^2\in\{2,3\}$ and $\chi(\mathcal{O}_X) \geq 2$.
\item $K_X^2=4$ and $ \chi(\mathcal{O}_X)\geq 3$.
\end{enumerate}
Then $b_1(X)=0$. 
\end{lemma}

\begin{proof}
This is an immediate consequence of Lemma~\ref{b1}. By~\cite[Corollary 1.8]{Ek87}, $\chi(\mathcal{O}_X) \geq 2-K_X^2$. Hence if $K_X^2=1$, then $2\chi(\mathcal{O}_X)-K_X^2>0$. Also under all the other  hypotheses of the corollary,  $\chi(\mathcal{O}_X)-K_X^2>0$ as well. Therefore in all these cases it follows from Lemma~\ref{b1} that $\pi_1^{et}(X)$ is finite and therefore  $b_1(X)=0$. 
\end{proof}

Suppose then that $k(Y^{\prime})=-\infty$. Then $Z$ is a ruled surface over a smooth curve $B$ and therefore $Y$ is ruled and hence for the same reasons as in the proof of Proposition~\ref{sec-all-p-prop-1}, $X$ is uniruled. 

In order to conclude the proof of Case 4, it remains to show that the case $K_X^2=1$ does not happen and the claim about the euler characteristics. 

From Lemma~\ref{cor-b1} it follows that if $K_X^2=1$ then $b_1(X)=0$. Moreover  if the inequalities for the euler characteristics in the statement  of Case 4 do not hold, then $b_1(X)=0$ as well. Hence all the claims will follow if I show that if $b_1(X)=0$ then $X$ does not have nontrivial global vector fields of multiplicative type. 

Suppose then that $b_1(X)=0$. Then $B \cong \mathbb{P}^1_k$ and hence both $Y^{\prime}$, $Y$ are rational and therefore $X$ is unirational. 

\textbf{Claim:} $X$ lifts to $W_2(k)$, where $W_2(k)$ is the ring of second Witt vectors over $k$.

Suppose that the claim is true. Then from Corollary~\ref{lifts-to-zero}, $X$ has no nontrivial global vector fields.

It remains then to prove the claim. Recall that $Y^{\prime}$ is the quotient of $X^{\prime}$ by the $\mu_2$ action on $X^{\prime}$ induced by the lifting $D^{\prime}$ of $D$ on $X^{\prime}$. Since both $X^{\prime}$ and $Y^{\prime}$ are smooth it follows from Proposition~\ref{structure-of-pi} that 
\[
X^{\prime}=\mathrm{Spec}\left( \mathcal{O}_{Y^{\prime}}\oplus M^{-1} \right)
\]
where $M$ is an invertible sheaf on $Y^{\prime}$ and the ring structure on $\mathcal{O}_{Y^{\prime}}\oplus M^{-1}$ is induced by a section $s$ of $M^{\otimes 2}$. It is well known~\cite{Ha10} that $H^2(T_{Y^{\prime}})$ is an obstruction space for deformations of $X$ over local Artin rings and $H^1(\mathcal{O}_{Y^{\prime}})$, $H^2(\mathcal{O}_{Y^{\prime}})$ are obstruction spaces for deformations of line bundles and sections of line bundles on $Y^{\prime}$. Since $Y^{\prime}$ is rational, all these spaces are zero. Therefore, $Y^{\prime}$, $M$ and the section $s \in H^0(M^{\otimes 2})$ all lift compatively to $W_2(k)$. Let $Y^{\prime}_2$, $M_2$ and $s_2$ be the liftings of $Y^{\prime}$, $M$ and $s$, respectively. Then
\[
X^{\prime}_2=\mathrm{Spec}\left( \mathcal{O}_{Y_2^{\prime}}\oplus M_2^{-1}\right)
\]
is a lifting of $X^{\prime}$ over $W_2(k)$. Next I will show that $X$ lifts over $W_2(k)$ too. Let $X_2$ be the ringed space $(X,f_{\ast}\mathcal{O}_{X^{\prime}_2})$. Then $X_2$ is a deformation of $X$ over $W_2(k)$. Indeed. From the construction of second Witt vectors~\cite{EV92} there exists an exact sequence \[
0 \rightarrow k \stackrel{\sigma}{\rightarrow} W_2(k) \rightarrow k \rightarrow 0,
\]
where $\sigma(x)=x\cdot p$. Then tensoring with $\mathcal{O}_{X^{\prime}_2}$ we get the exact sequence
\[
0 \rightarrow \mathcal{O}_{X^{\prime}} \stackrel{\sigma}{\rightarrow} \mathcal{O}_{X^{\prime}_2} \rightarrow \mathcal{O}_{X^{\prime}} \rightarrow 0
\]
Applying $f_{\ast}$ we get the exact sequence
\[
0 \rightarrow f_{\ast}\mathcal{O}_{X^{\prime}} \rightarrow f_{\ast}\mathcal{O}_{X^{\prime}_2} \rightarrow f_{\ast}\mathcal{O}_{X^{\prime}} \rightarrow R^1f_{\ast}\mathcal{O}_{X^{\prime}}
\]
Considering that $f_{\ast}\mathcal{O}_{X^{\prime}}\cong \mathcal{O}_X$ and that $R^1f_{\ast}\mathcal{O}_X=0$, we get the following exact sequence
\[
0 \rightarrow \mathcal{O}_X \stackrel{p}{\rightarrow} f_{\ast}\mathcal{O}_{X_2^{\prime}} \rightarrow \mathcal{O}_X \rightarrow 0.
\]
Now tensoring with $k$ over $W_2(k)$ we get that $\mathcal{O}_{X^{\prime}_2}\otimes_{W_2(k)}k\cong \mathcal{O}_{X}$. Finally from the infinitesimal criterion of flatness, $X_2$ is flat over $W_2(k)$ and hence $X_2$ is a deformation of $X$ over $W_2(k)$, as claimed.

\end{proof}

%% file: sec8.tex
\section{Surfaces with vector fields of additive type in characteristic 2.}\label{sec-5}

The purpose of this section is to study smooth canonically polarized surfaces defined over an algebraically closed field $k$ of characteristic 2 which admit nontrivial global vector fields of additive type. As it was shown in section~\ref{sec-2} this is equivalent to the condition that $\alpha_2$ is a subgroup scheme of $\mathrm{Aut}(X)$. 

This case is more complicated  from the multiplicative case essentially because $\alpha_2$ actions are harder to describe than $\mu_2$ actions. One of the  difficulties is that the singularities of the quotient of the surface with the induced $\alpha_2$ action are more complicated than those that appear in the multiplicative case. In fact, not only they are not necessarily canonical, but they may not even be rational. However, from Proposition~\ref{prop4}, they are Gorenstein. 

The main result of this section is the following.

\begin{theorem}\label{additive-type}
Let $X$ be a smooth canonically polarized surface defined over an algebraically closed field of characteristic 2. Suppose that $X$ has a nontrivial global vector field of additive type, or equivalently that $\alpha_2$ is a subgroup scheme of $\mathrm{Aut}(X)$. Then:
\begin{enumerate}
\item If $K_X^2=2$ then $X$ is uniruled. Moreover, if $\chi(\mathcal{O}_X)\geq 2$, then $X$ is unirational and $\pi_1^{et}(X)=\{1\}$.
\item If $K_X^2=1$ then $\pi_1^{et}(X)=\{1\}$, $p_g(X) \leq 1$ and $X$ is unirational. 
\end{enumerate}
Moreover, $X$ is the quotient of a ruled or rational surface (maybe singular) by a rational vector field.
\end{theorem}

From the following theorem it immediately follows that.

\begin{corollary}\label{cor-additive-type}
Let $X$ be a smooth canonically polarized surface defined over an algebraically closed field of characteristic 2. Suppose that either $K_X^2=2$ and $X$ is not uniruled or $K_X^2=1$ and one of the following happens
\begin{enumerate}
\item $p_g(X)=2$.
\item $\chi(\mathcal{O}_X) =3$.
\item $\pi_1^{et}(X) \not= \{1\}$.
\item $X$ is not unirational.
\end{enumerate}
Then $X$ has no nontrivial global vector field of additive type. In particular, $\alpha_2$ is not a subgroup scheme of $\mathrm{Aut}(X)$.
\end{corollary}

\begin{proof}[Proof of Theorem~\ref{additive-type}]
The proof will be along the lines of the proof of Theorem~\ref{mult-type}. However, there are many differences between the two cases that complicate things. Suppose that $X$ admits a nontrivial global vector field $D$ of additive type.  Then $D$ induces a nontrivial $\alpha_2$ action on $X$. Let $\pi \colon X \rightarrow Y$ be the quotient. Then by Proposition~\ref{prop4}, $Y$ is normal and its local class groups are 2-torsion. Moreover, there are commutative diagrams
\begin{equation}\label{sec5-diagram-1}
\begin{tabular}{lllrrr}
\xymatrix{
 & X^{\prime\prime}\ar[r]^{f^{\prime\prime}} \ar[d]_{\pi^{\prime\prime}} & X \ar[d]^{\pi} \\
 & Y^{\prime\prime} \ar[r]^{g^{\prime\prime}} & Y \\
}
&&&&&
\xymatrix{
 & X^{\prime}\ar[r]^f \ar[d]^{\pi^{\prime}} & X \ar[d]^{\pi} \\
Z & Y^{\prime} \ar[l]_{\phi}\ar[r]^g & Y \\
}
\end{tabular}
\end{equation}
such that: 
\begin{enumerate}
\item $f^{\prime\prime}$ is a resolution of the isolated singularities of $D$ through successive blow ups of its isolated singular points. $X^{\prime\prime}$ is smooth, $D$ lifts to a vector field $D^{\prime\prime}$ in $X^{\prime\prime}$ with only divisorial singularities and $Y^{\prime\prime}$ is the quotient of $X^{\prime\prime}$ by the corresponding action of $\alpha_2$. However, unlike the multiplicative case the singularities of $Y$ are not necessarily canonical and $Y^{\prime\prime}$ may not be the minimal resolution of $Y$. 
\item $Y^{\prime}$ is the minimal resolution of $Y$ and $Z$ its minimal model. However, $X^{\prime}$ may now be singular. In any case, $X^{\prime}$ has rational singularities and the $f$ and $g$ exceptional sets are trees of smooth rational curves.
\end{enumerate}

 Suppose that
\begin{gather}\label{sec5-eq-1}
K_{Y^{\prime}}=g^{\ast} K_Y -F\\
K_{Y^{\prime}}=\phi^{\ast}K_Z+B \nonumber
\end{gather}
where $F$, $B$ are effective $g$ and $\phi$-exceptional divisors, respectively. Let $\Delta$ be the divisorial part of $D$. Unlike the multiplicative case, $\Delta$ may be singular, nonreduced and it may even contain isolated fixed points of $D$. By adjunction for purely inseparable morphisms~\cite{Ek87}~\cite{R-S76},
\begin{gather}\label{sec5-eq-2}
K_X=\pi^{\ast}K_Y+\Delta.
\end{gather}
Moreover, from Proposition~\ref{structure-of-pi}, $\pi$ is a torsor over a codimension 2 open subset of $Y$. Moreover, since $Y^{\prime}$ is smooth and $X^{\prime}$ is normal, $\pi^{\prime}$ is a torsor too. In particular $X^{\prime}$ has hypersurface singularities and hence $K_{X^{\prime}}$ is Cartier. Moreover,  since $\pi$ and $\pi^{\prime}$ are torsors, it follows from Proposition~\ref{structure-of-pi} that there are exact sequences
\begin{gather}\label{sec5-eq-3}
0 \rightarrow \mathcal{O}_Y \rightarrow E \rightarrow L^{-1} \rightarrow 0 \\
0 \rightarrow \mathcal{O}_{Y^{\prime}} \rightarrow E^{\prime} \rightarrow M^{-1} \rightarrow 0 \nonumber
\end{gather}
where $E=\pi_{\ast}\mathcal{O}_X$, $E^{\prime}=\pi^{\prime}_{\ast}\mathcal{O}_{X^{\prime}}$, $L=\mathcal{O}(C)$ is a reflexive sheaf on $Y$ and $M=\mathcal{O}_{Y^{\prime}}(C^{\prime})$ is an invertible sheaf on $Y^{\prime}$. Moreover,
\begin{gather}\label{sec5-eq-4}
K_X=\pi^{\ast}(K_Y+C)\\
K_{X^{\prime}}=(\pi^{\prime})^{\ast}(K_{Y^{\prime}}+C^{\prime})
\end{gather}
From this and~\ref{sec5-eq-2} it follows that $\pi^{\ast}C=\Delta$. Moreover, since $K_X$ is ample and $\pi$ a finite morphism, it follows that $K_Y+C$ is ample too.

Finally, from~\ref{sec5-eq-4} and the fact that $\pi$ is finite of degree 2 it follows that
\begin{gather}\label{sec4-eq-5}
K_X^2=2(K_Y+C)^2=2K_Y\cdot (K_Y+C) +2C\cdot (K_Y+C).
\end{gather}

\textbf{Claim:} If $\Delta =0$, then $K_X^2 \geq 4$. Hence, unlike the case when $p \geq 3$, the possibility that $\Delta=0$ does not create any problems here. 

I proceed to prove the claim. Suppose that $\Delta=0$. Then $K_X=\pi^{\ast}K_Y$ and hence $K_Y$ is ample. In the case of vector fields of multiplicative type, if this happened then $k(Y)=2$. This happened because $Y$ had singularities of type $A_1$, in particular rational. If we knew that $Y$ had rational singularities it would be possible to show that $k(Y)=2$ again by comparing $H^0(\omega_Y^{[n]})$ and $H^0(\omega_{Y^{\prime}}^n)$. However, $Y$ may have non rational singularities in the additive case so in principle it could happen that $\Delta=0$ and $k(Y)< k(X)$.

By its construction, $\pi$ factors through the geometric Frobenious $F \colon X \rightarrow X^{(2)}$. In fact there is a commutative diagram
\[
\xymatrix{
     & Y \ar[dr]^{\nu} \\
X \ar[ur]^{\pi}\ar[rr]^F & & X^{(2)}
}
\]
Since $X^{(2)}$ is smooth and $Y$ is normal, then $\nu$ is a torsor over $X^{(2)}$~\cite{Ek87}. Therefore, \[
K_Y=\nu^{\ast}(K_{X^{(2)}}+W^{(2)})
\]
where $W^{(2)}$ is a divisor on $X^{(2)}$. Recall that the geometric Frobenious is constructed from the next commutative diagram
\[
\xymatrix{
X  \ar[drr]^{F_{ab}}\ar[ddr] \ar[dr]^F     &                                 &        \\     
         &   X^{(2)} \ar[r]^{pr_1}\ar[d]^{pr_2}                      &       X\ar[d]^{\pi}   \\
         &   \mathrm{Spec}(k) \ar[r]^{F_{ab}}             &   \mathrm{Spec}(k)
}
\]
where $F_{ab}$ is the absolute Frobenious. Since $k$ is algebraically closed, $pr_1$ is an isomorphism. Hence $W^{(2)}=pr_1^{\ast}W$, where $W$ is a divisor on $X$. Then 
\[
K_X=\pi^{\ast}K_Y=\pi^{\ast}\nu^{\ast}(K_{X^{(2)}}+W^{(2)})=F^{\ast}(K_{X^{(2)}}+W^{(2)})=F^{\ast}_{ab}(K_X+W)=2K_X+2W.
\]
Therefore $K_X=-2W$ and hence $K_X^2 =4W^2 \geq 4$, as claimed. Hence in the following we can assume the $\Delta\not= 0$.

As in the multiplicative case, the proof of the Theorem~\ref{additive-type} will be in several steps, according to the Kodaira dimension $\kappa(Y^{\prime})$ of $Y^{\prime}$.

\textbf{Case 1.} Suppose $\kappa(Y^{\prime})=2$ or $\kappa(Y^{\prime})=1$. Then arguing similarly as in Cases 1 and 2 in the proof of Theorem~\ref{mult-type}, we get that $K_X^2 \geq 3$.

\textbf{Case 2.} Suppose that $\kappa(Y^{\prime})=0$. In this I will show that if $1\leq K_X^2\leq 2$, then $X$ is unirational and $\pi_1^{et}(X)=\{1\}$. In particular, if $K_X^2=1$, then $X$ is a simply 
connected supersingular Godeaux surface. 

I will only prove the statement for the case when $K_X^2=1$. The proof for case $K_X^2=2$ is similar with minor differences and it is left to the reader.

Suppose then that $K_X^2=1$. Similar arguments as in the cases $k(Y)=2$ and $k(Y)=1$ show that if $\Delta=0$ then $K_X^2 \geq 4$.

Suppose that $Y^{\prime}$, in diagram~\ref{sec5-diagram-1} is not minimal. Then similar arguments as in the multiplicative case give that $K_X^2 \geq 3$. 

Suppose now that $Y^{\prime}$ is minimal. The argument of the multiplicative case used in an essential way the fact that $Y$ has singularities of type $A_1$ and cannot be used in this case directly.  

Suppose that $K_X^2=1$. Then it is known~\cite{Li09} that $1\leq \chi(\mathcal{O}_X) \leq 3$ ($0\leq \chi(\mathcal{O}_X) \leq 4$, if $K_X^2=2$~\cite[Corollary1.8]{Ek87}). Hence in order to show the claim it suffices to show that the cases $\chi(\mathcal{O}_X)\in\{2,3\}$ is impossible, that $\pi_1^{et}(X)=\{1\}$ and that $X$ is unirational and not singular. FromLemma~\ref{b1} it follows that $b_1(X)=0$ and hence $b_1(Y^{\prime})=b_1(Y)=b_1(X)=0$. Hence $Y^{\prime}$ is either an Enriques or a K3 surface and hence $c_2(Y^{\prime})=12$, if $Y^{\prime}$ is Enriques, and 24 if it is K3.

From diagram~\ref{sec5-diagram-1} it follows that 
\begin{gather}\label{sec5-eq-6}
c_2(Y^{\prime})=\chi_{et}(Y^{\prime})=\chi_{et}(X^{\prime})=c_2(X)+k,
\end{gather}
where $k$ is the number of $f$-exceptional curves.

Suppose that $\chi(\mathcal{O}_X)=3$. Then from Noethers formula we get that $c_2(X)=35$ and from~\ref{sec5-eq-6} that $c_2(Y^{\prime})=35+k>24$. Hence this case is impossible.

Suppose that $\chi(\mathcal{O}_X)=2$. Then from Noethers formula we get that $c_2(X)=23$. Then $c_2(Y^{\prime})=c_2(X)+k=23+k$. Hence the only possibility is that $Y^{\prime}$ is a K3 and $k=1$. Then I claim that $Y^{\prime}$ has exactly one singular point which must be canonical of type $A_1$. $Y$ has canonical singularities since $K_{Y^{\prime}}=0$. Let $E$ and $F$ be the $f$ and $g$-exceptional curves, respectively. Both are smooth rational curves. Then since $K_{Y^{\prime}}=0$, it follows that $F^2=-2$ and so $Y$ has exactly one $A_1$ singular point (if $K_X^2=2$ then $k=2$ and $Y$ has canonical singularities whose minimal resolution has two exceptional curves. Then by Proposition~\ref{prop4}, $Y$ has exactly two $A_1$ singular points).

Let $\Delta$ be the divisorial part of $D$ and $\Delta^{\prime}$ the divisorial part of $D^{\prime}$, the lifting of $D$ on $X^{\prime}$. Then since $Y^{\prime}$ is K3, $K_X=\Delta$ and $K_{X^{\prime}}=\Delta^{\prime}$. Moreover $\Delta^{\prime}=(\pi^{\prime})^{\ast}M$, where $M$ is as in equations~\ref{sec5-eq-3}. From~\ref{sec5-eq-3} it follows that
\[
\chi(M^{-1})=\chi(E^{\prime})-\chi(\mathcal{O}_{Y^{\prime}})=\chi(\mathcal{O}_{X^{\prime}})-\chi(\mathcal{O}_{Y^{\prime}})=2-2=0.
\]
Now Riemann-Roch gives that
\[
0=\chi(M^{-1})=\chi(\mathcal{O}_{Y^{\prime}})+\frac{1}{2}(M^2+M\cdot K_{Y^{\prime}})=2+\frac{1}{2}M^2.
\]
Hence $M^2=-4$ and therefore, since $\Delta^{\prime}=(\pi^{\prime})^{\ast}M$, $K_{X^{\prime}}^2=(\Delta^{\prime})^2=-8$. Since $X$ is smooth and $K_{X^{\prime}}$ Cartier, there is a positive $a\in \mathbb{Z}$ such that 
\begin{gather}\label{sec5-eq-7}
K_{X^{\prime}}=f^{\ast}K_X+aE.
\end{gather}
Now $E$ may or may not be an integral curve for $D^{\prime}$. If it is an integral curve, then $(\pi^{\prime})^{\ast}F=E$ and hence $E^2=-4$. Then from~\ref{sec5-eq-7} we get that 
\begin{gather}\label{sec5-eq-8}
-8=K_{X^{\prime}}^2=K_X^2-4a^2=1-4a^2,
\end{gather}
which is impossible. Suppose that $E$ is not an integral curve for $D^{\prime}$. Then $2E=(\pi^{\prime})^{\ast}F$ and hence $E^2=-1$. Then from~\ref{sec5-eq-8} it follows that $a=3$ and hence $K_{X^{\prime}}=f^{\ast}K_X+3E$. Hence $3E$ is Cartier. But since $2E=(\pi^{\prime})^{\ast}F$ and $Y^{\prime}$ is smooth, it follows that $2E$ is Cartier as well. Hence $E$ is Cartier. But since $E=\mathbb{P}^1$ it follows that $X^{\prime}$ is in fact smooth and $f$ is the contraction of a $-1$ curve. But then $K_{X^{\prime}}=f^{\ast}K_X+E$, a contradiction. Hence the case $\chi(\mathcal{O}_X)=2$ is impossible too. Hence $\chi(\mathcal{O}_X)=1$ and therefore $X$ is a Godeaux surface. 

Next I will show that $Y^{\prime}$ is a K3 surface with a special configuration~\cite{SB96} of rank 13. 

Again as before we find that
\begin{gather}\label{sec5-eq-9}
c_2(Y^{\prime})=c_2(X)+k=11+k,
\end{gather}
where $k$ is the number of $f$-exceptional curves. If $Y^{\prime}$ was Enriques, then $c_2(Y^{\prime})=12$ and hence $k=1$. I now repeat the previous argument. Exactly as before we get that $K_{X^{\prime}}^2=-4$. Suppose that $(\pi^{\prime})^{\ast}F=E$. Then $E^2=-4$ and hence from~\ref{sec5-eq-8} we get that $-4=1-4a^2$, which is impossible. If on the other hand $(\pi^{\prime})^{\ast}F=2E$, then $E^2=-1$ and hence again from~\ref{sec5-eq-8} we get that $-4=1-a^2$ and hence $a^2=5$, which is again impossible since $a\in \mathbb{Z}$. Hence $Y^{\prime}$ is a $K3$ surface and therefore $c_2(Y^{\prime})=24$. Then from~\ref{sec5-eq-9} it follows that $k=13$ and hence $g$ has exactly 13 exceptional curves. Moreover, $Y$ has canonical singularities and therefore by~\ref{prop4} they must be of type either $A_1$ or $D_{2n}$. Hence by~\cite{SB96}, $Y^{\prime}$ has a special configuration of rank 13. Then $Y^{\prime}$ is unirational~\cite{SB96} and hence $X$  is unirational as well. Moreover considering that the \'
etale fundamental group is a birational invariant between smooth varieties and also invariant under purely inseparable finite maps, we get from~\ref{sec5-diagram-1} that 
\[
\pi_1^{et}(X)=\pi_1^{et}(X^{\prime\prime})=\pi_1^{et}(Y^{\prime\prime})=\pi_1^{et}(Y^{\prime})=\{1\}.
\]
Finally, I will show that $X$ is supersingular. This means that $\mathrm{p}_g(X)=h^1(\mathcal{O}_X)=1$ and the induced map $F^{\ast}$ of the Frobenious on $H^1(\mathcal{O}_X)$ is zero. Since $\omega_Y\cong \mathcal{O}_Y$ we get from duality for finite morphisms~\cite{Ha77} that $\omega_X=\pi^{!}\mathcal{O}_Y$. Hence
\[
H^0(\omega_X)=\mathrm{Hom}_Y(\pi_{\ast}\mathcal{O}_X,\mathcal{O}_Y).
\]
But this is nonzero since the map $\phi \colon \pi_{\ast}\mathcal{O}_X \rightarrow \mathcal{O}_Y$ defined by $\phi(a)=Da$ is nonzero and $\mathcal{O}_Y$-linear. Then from~\cite[Corollary 1.8]{Ek87} it follows that $\mathrm{p}_g(X)=h^1(\mathcal{O}_X)=1$. Finally for exactly the same reasons as in the multiplicative case, $F^{\ast}$ is zero and hence $X$ is supersingular.

\textbf{Case 3.} Suppose that $\kappa(Y^{\prime})=-\infty$. In this case I will show the following.
\begin{enumerate}
\item If $K_X^2=2$, then $X$ is uniruled. Moreover, if $\chi(\mathcal{O}_X)\geq 2$, then $X$ is unirational and $\pi_1^{et}(X)=\{1\}$.
\item If $K_X^2=1$ then $X$ is unirational, simply connected and $p_g(X) \leq 1$. In particular 
$1 \leq \chi(\mathcal{O}_X) \leq 2$.
\end{enumerate}

This case is very different from the corresponding multiplicative case where the core of the argument was that if $K_X^2 \leq 2$ then $X$ lifts to characteristic zero and hence Kodaira-Nakano vanishing holds which implies that $X$ does not have any nonzero global vector fields. The proof that $X$ lifts to characteristic zero was based on showing that every space and map that appears in~\ref{sec-all-p-diagram-1} lifts to charracteristic zero.  However, even though in this case too $Y^{\prime}$ lifts to characteristic zero, the construction of $X^{\prime}$ as a torsor over $Y^{\prime}$ depends heavily on being in positive characteristic and does not necessarily lift to characteristic zero.

The first part of the claim is obvious. Only the statement that if $\chi(\mathcal{O}_X)\geq 2$, then $X$ is unirational and $\pi_1^{et}(X)=\{1\}$ needs some justification. From Lemma~\ref{b1} it follows that if $\chi(\mathcal{O}_X)\geq 2$, then $b_1(X)=0$. Hence $b_1(X^{\prime})=0$ and therefore $X^{\prime}$ is rational and hence $X$ unirational and simply connected.

Suppose in the following that $K_X^2=1$. Fix notation as in diagram~\ref{sec5-diagram-1}. Since $k(Y^{\prime})=-\infty$, then $Z$ is ruled over a smooth curve $B$. From Lemma~\ref{b1} it follows that $b_1(X)=0$ and hence $B\cong \mathbb{P}^1$. Hence $X$ is unirational and for the same reasons as in the previous cases, $\pi_1^{et}(X)=\pi_1^{et}(Z)=\{1\}$. Hence $X$ is simply connected. It remains to show that $p_g(X) \leq 1$ and that $1 \leq \chi(\mathcal{O}_X) \leq 2$. Since $X$ is canonically polarized, Noether's inequality gives that $K_X^2\geq 2\mathrm{p}_g(X)-4$. Hence if $K_X^2=1$ then $\mathrm{p}_g(X) \leq 2$. Moreover 
from~\cite[Corollary 1.8]{Ek87}, $1\leq \chi(\mathcal{O}_X) \leq 3$. Hence to show the claim it suffices to show that the case $\mathrm{p}_g(X)=2$ is impossible. If this happens, then by~\cite[Corollary 1.8]{Ek87}, then $\chi(\mathcal{O}_X)\in \{2,3\}$.

Suppose then that $\mathrm{p}_g(X)=2$.  For the same reasons as in Case 1., if $\Delta=0$ then $K_X^2\geq 4$. So we may assume that $\Delta\not=0$.

Next I will study some properties of the linear system $|K_X|$. 

\textbf{Claim.} 
\begin{enumerate}
\item The linear system $|K_X|$ has a unique base point $P$. Moreover every member of $|K_X|$ is smooth at $P$. 
\item For any $W \in |K_X|$ there exists a smooth rational curve $\tilde{W} \in |K_Y+C|$ such that $\pi^{\ast}\tilde{W}=W$.
\item $Q=\pi(P)$ is the unique singular point of $Y$. Moreover, it is of type $A_1$.
\end{enumerate}

Before I proceed to prove the claim I would like to make the following  remarks. 

$Y$ is certainly singular since $(K_Y+C)^2=1/2K_X^2=1/2$.

Let $W \in |K_X|$ be any member. Then since $K_X$ is ample and $W \cdot K_X=1$ it follows that $W$ is reduced and irreducible. Moreover, for any two distinct $C_1, C_2 \in |K_X|$, $C_1 \cdot C_2=1$ and their point of intersection is the base point $P$ of $|K_X|$.

The proof of the first part of the claim is exactly the same as in the case when $p\geq 3$ in Proposition~\ref{sec-all-p-prop-3} and for this reason it is omitted. 

Next I will show the second part of the claim. So let $W\in |K_X|$. I will show that there exists  $\tilde{W} \in |K_Y+C|$ such that $\pi^{\ast}\tilde{W}=W$. In the notation of diagram~\ref{sec5-diagram-1}, $H^0(\omega_{X^{\prime\prime}})=H^0(\omega_X)$. Let then $W^{\prime\prime}\in|K_{X^{\prime\prime}}|$ be a lifting of $W$ in $X^{\prime\prime}$. Then I will show that there exists $\bar{W}\in |K_{Y^{\prime\prime}}+M|$ such that $(\pi^{\prime\prime})^{\ast}\bar{W}=W^{\prime\prime}$. Then $\tilde{W}=h_{\ast}\bar{W} \in |K_Y+C|$ and $\pi^{\ast}\tilde{W}=W$. This is because all the previous equations hold over a codimension 2 open subset of $Y$ (its smooth part) and so everywhere. Now from the equation~\ref{sec5-eq-4} it follows that,
\[
H^0(\omega_{X^{\prime\prime}})=H^0((\pi^{\prime\prime})^{\ast}(\omega_{Y^{\prime\prime}}\otimes M))=H^0(\omega_{Y^{\prime\prime}}\otimes M\otimes \pi^{\prime\prime}_{\ast}\mathcal{O}_{X^{\prime\prime}})
\]
Then from the equation~\ref{sec5-eq-3} we get the exact sequence
\[
0 \rightarrow \omega_{Y^{\prime\prime}}\otimes M \rightarrow \pi^{\prime\prime}_{\ast}\mathcal{O}_{X^{\prime\prime}}\otimes \omega_{Y^{\prime\prime}}\otimes M \rightarrow  \omega_{Y^{\prime\prime}} \rightarrow 0.
\]
This gives an exact sequence in cohomology
\[
0 \rightarrow H^0(\omega_{Y^{\prime\prime}}\otimes M) \rightarrow H^0(\pi^{\prime\prime}_{\ast}\mathcal{O}_{X^{\prime\prime}}\otimes \omega_{Y^{\prime\prime}}\otimes M) \rightarrow  H^0(\omega_{Y^{\prime\prime}}) \rightarrow \cdots
\]
Now since $Y^{\prime\prime}$ is rational, it follows that $H^0(\omega_{Y^{\prime\prime}})=0$. Hence 
\[
H^0(\omega_{Y^{\prime\prime}}\otimes M) = H^0(\pi^{\prime\prime}_{\ast}\mathcal{O}_{X^{\prime\prime}}\otimes \omega_{Y^{\prime\prime}}\otimes M)
\]
and therefore there exists $\bar{W}\in |K_{Y^{\prime\prime}}+M|$ such that $(\pi^{\prime\prime})^{\ast}\bar{W}=W^{\prime\prime}$. This concludes the proof of the second part of the claim.

I will next show that $\tilde{W}\cong \mathbb{P}^1$. Since $\tilde{W} \in |K_Y+C|$, there exists an exact sequence
\[
0 \rightarrow \mathcal{O}_Y(-K_Y-C) \rightarrow \mathcal{O}_Y \rightarrow \mathcal{O}_{\tilde{W}} \rightarrow 0.
\]
This gives an exact sequence in cohomology
\[
\cdots \rightarrow H^1(\mathcal{O}_Y(-K_Y-C)) \rightarrow H^1(\mathcal{O}_Y) \rightarrow H^1(\mathcal{O}_{\tilde{W}}) \rightarrow H^2(\mathcal{O}_Y(-K_Y-C)) \cdots 
\]
I will show that $H^1(\mathcal{O}_Y)=H^2(\mathcal{O}_Y(-K_Y-C))=0$. Therefore, $H^1(\mathcal{O}_{\tilde{W}})=0$ and hence $\tilde{W}\cong \mathbb{P}^1$. Since $Y^{\prime}$ is rational it follows that 
$H^i(\mathcal{O}_{Y^{\prime}})=0$, $i=1,2$. If $Y$ had rational singularities then the same would hold for $Y$. But $Y$ may have nonrational singularities. 
However, the Leray spectral sequence we get the exact sequence
\[
0 \rightarrow H^1(g_{\ast}\mathcal{O}_{Y^{\prime}}) \rightarrow H^1(\mathcal{O}_{Y^{\prime}}) \rightarrow H^0(R^1g_{\ast}\mathcal{O}_{Y^{\prime}}) \rightarrow H^2(\mathcal{O}_Y^{\prime})
\]
Since $g_{\ast}\mathcal{O}_{Y^{\prime}}=\mathcal{O}_Y$ and $ H^1(\mathcal{O}_{Y^{\prime}})=0$, it follows that $H^1(\mathcal{O}_Y)=0$ as well. Next I will show that $H^2(\mathcal{O}_Y(-K_Y-C))=0$. By Serre duality for Cohen-Macauley sheaves~\cite{KM98}, we get that
\[
H^2(\mathcal{O}_Y(-K_Y-C))=H^0(\mathcal{O}_Y(2K_Y+C)).
\]
Suppose that $H^0(\mathcal{O}_Y(2K_Y+C))\not= 0$. Then there exists a nonzero effective divisor $Z \in |2K_Y+C|$. Then $K_Y+C = Z-K_Y$. Therefore from~\ref{sec5-eq-4} we get that 
\begin{gather}
K_X^2=2(K_Y+C)^2=2K_Y\cdot(K_Y+C) +2C\cdot (K_Y+C) = \label{sec5-eq-10}\\
2Z\cdot (K_Y+C)-2K_Y\cdot (K_Y+C).\nonumber
\end{gather}
Since $K_Y+C$ is ample and $2Z, 2C$ are Cartier, $2Z\cdot (K_Y+C)\geq 1$ and $2C\cdot (K_Y+C) \geq 1$ (since $2C$ is equivalent to $\pi_{\ast}\Delta$, which is effective). 

Suppose that $K_Y \cdot (K_Y+C) <0$. Then from the second equality of~\ref{sec5-eq-10} it follows that $K_X^2 \geq 2$. Suppose that $K_Y \cdot (K_Y+C) >0$. Then from the first equality of~\ref{sec5-eq-10} it follows that $K_X^2\geq 3$. 

Suppose now that $K_Y \cdot (K_Y+C)=0$. Then $\pi^{\ast}K_Y \cdot K_X=0$ and hence from the adjunction $K_X=\pi^{\ast}K_Y+\Delta$ we get that $K_X \cdot \Delta=K_X^2=1$.

Fix notation as in diagram~\ref{sec5-diagram-1}. Then from Corollary~\ref{sec3-cor-1} we get that
\begin{gather*}
K_{X^{\prime\prime}}\cdot \Delta^{\prime\prime}=4\left(\chi(\mathcal{O}_{X^{\prime\prime}})-2\chi(\mathcal{O}_{Y^{\prime\prime}})\right).
\end{gather*}
Since $Y^{\prime\prime}$ is rational, $\chi(\mathcal{O}_{Y^{\prime\prime}})=1$ and hence
\begin{gather}\label{sec5-eq-11}
K_{X^{\prime\prime}}\cdot \Delta^{\prime\prime}=4\left(\chi(\mathcal{O}_{X^{\prime\prime}})-2\right)
\end{gather}
If $\mathrm{p}_g(X)=2$, then either $\chi(\mathcal{O}_X)=3$ or $\chi(\mathcal{O}_X)=2$. Suppose that $\chi(\mathcal{O}_X)=3$. Then from~\ref{sec5-eq-11} it follows that $K_{X^{\prime\prime}}\cdot \Delta^{\prime\prime}=4>1=K_X\cdot \Delta$, which is impossible by Corollary~\ref{sec3-cor-1}.

Suppose that $\chi(\mathcal{O}_X)=2$. In this case, $K_{X^{\prime\prime}}\cdot \Delta^{\prime\prime}=0$.

\begin{claim}\label{sec5-claim} $Y$ has exactly one singular point which must be of type $A_1$. 
\end{claim}

Indeed. From Corollary~\ref{sec3-cor-1}, $K_X \cdot \Delta$ decreases from $X$ to $X^{\prime\prime}$. In order to show the claim I will study how exactly $K_X \cdot \Delta$ decreases. Since $f$ is a composition of blow ups of isolated singular points of $D$, it suffices to examine what happens after a single blow up. So let $f_1 \colon X_1 \rightarrow X$ be the blow up of an isolated singular point of $D$. Let $D_1$ be the lifting of $D$ on $X_1$,  $\Delta_1$ its divisorial part and $Y_1$ the quotient of $X_1$ by $D_1$. Then there exists a commutative diagram
\begin{gather}\label{sec5-diagram-4}
\xymatrix{
 & X_1\ar[r]^{f_1} \ar[d]_{\pi_1} & X \ar[d]^{\pi} \\
 & Y_1 \ar[r]^{g_1} & Y \\
}
\end{gather}
Let $E$ be the $f_1$-exceptional curve and $F=\pi_1(E)$ the $g$-exceptional curve. Then, since $K_Y$ is Cartier, there is $a \in \mathbb{Z}$ such that
\[
K_{Y_1}=g_1^{\ast}K_Y+aF.
\]
Moreover, $K_X=\pi^{\ast}K_Y+\Delta$. Hence from~\ref{sec5-diagram-4} we that
\[
K_{X_1}=f_1^{\ast}K_X+E=f_1^{\ast}\pi^{\ast}K_Y+f_1^{\ast}\Delta+E=\pi_1^{\ast}K_{Y_1}+f_1^{\ast}\Delta+(1-ka)E
\]
where $k\in\{1,2\}$ is such that $\pi_1^{\ast}F=kE$. If $E$ is an integral curve for $D^{\prime}$, then $k=1$. Otherwise $k=2$. Therefore,
\begin{gather}\label{sec5-eq-16}
\Delta_1=f_1^{\ast}\Delta+(1-ka)E.
\end{gather}
From the proof of Proposition~\ref{K-decreases} it follows that $1-ka\geq 0$. If $1-ka>1$, then 
\[
K_{X_1}\cdot \Delta_1=K_X\cdot \Delta -(1-ka)\leq  K_X\cdot \Delta -2.
\]
But considering that $K_X\cdot \Delta =1$ and $K_{X^{\prime\prime}}\cdot \Delta^{\prime\prime}=0$, this cannot happen. Hence $1-ka\in \{0,1\}$.

Suppose that $1-ka=1$. Then $a=0$ and $K_{X_1}\cdot \Delta_1=K_X\cdot \Delta -1$, hence $K_X\cdot \Delta$ drops by one. Again considering that $K_X\cdot \Delta =1$ and $K_{X^{\prime\prime}}\cdot \Delta^{\prime\prime}=0$, this case can happen only once. In all other cases, $K_X\cdot \Delta$ stays the same and hence $1-ka=0$, i.e, $a=1/k$ (and hence $k=1$ since $a\in \mathbb{Z}$). 

In the notation then of~\ref{sec5-diagram-1}, one can write
\[
K_{Y^{\prime\prime}}=h^{\ast}K_Y+\sum_ia_iF_i,
\]
such that $a_i \in \mathbb{Z}$ and there is exactly one $i$ such that $a_i=0$ and $a_j >0$ for all $j\not=i$. Hence $Y$ has canonical singularities. Now consider the commutative diagram
\[
\xymatrix{
                & Y^{\prime} \ar[dr]^{g}      &\\
Y^{\prime\prime} \ar[ur]^{\phi} \ar[rr]^{h} & & Y \\
}
\]  
where $g \colon Y^{\prime} \rightarrow Y$ is the minimal resolution. Then since $h$ has exactly one crepant divisor, $g$ has exactly one crepant divisor too. Therefore since $Y$ has canonical singularities, it follows that $K_Z=g^{\ast}K_Y$ and $g$ contracts exactly one smooth rational curve of self intersection -2. Hence $Y$ has exactly 1 singular point which must be of type $A_1$, as claimed.

Then \[
c_2(Y^{\prime})=\chi_{et}(Y)+1=c_2(X)+1.
\]
If $\chi(\mathcal{O}_X)=2$, then $c_2(X)=23$. Therefore, $c_2(Y^{\prime})=24$. However, since $Y^{\prime}$ is rational, $\chi(\mathcal{O}_{Y^{\prime}})=1$ and hence from Noethers formula
\[
K_Y^2=K_{Y^{\prime}}^2=12-24=-12.
\]
Moreover, since $K_X=\pi^{\ast}K_Y+\Delta$ it follows that
\[
-24=2K_Y^2=(K_X-\Delta)^2=K_X^2+\Delta^2-2K_X\cdot \Delta =-1+\Delta^2
\]
and hence $\Delta^2=-23$. However, since $K_X \cdot \Delta =1$ and $K_X$ is ample, it follows that $\Delta $ is irreducible and reduced. But then from the genus formula
\begin{gather}\label{sec5-eq-222}
2p_a(\Delta)-2=K_X\cdot \Delta +\Delta^2=1-23=-22.
\end{gather}
But this is impossible. Therefore $K_Y\cdot (K_Y+C) \not=0$ and hence
\[
H^0(\mathcal{O}_Y(2K_Y+C))\not= 0,
\]
as claimed. Hence $\tilde{W}\cong \mathbb{P}^1$. In particular, $\tilde{W}$ is smooth.

Next I will show that $Q=\pi(P)$, where $P=C_1\cap C_2$ is the only singular point of $Y$. Indeed. $P$ is the only base point of the 2-dimensional linear system $|K_X|$. Hence $Q$ is the only base point of $|K_Y+C|$. Let $W_1, W_2\in |K_X|$ be two distinct members. Then as was shown earlier, there are $\tilde{W}_i\in|K_Y+C|$ such that $W_i=\pi^{\ast}\tilde{W}_i$, $i=1,2$. Since $\tilde{W}_1+\tilde{W}_2 \in |2K_Y+2C|$, and the local class groups of $Y$ are 2-torsion, it follows that $\tilde{W}_1+\tilde{W}_2$ is Cartier. Moreover, since both $\tilde{W}_1$ and $\tilde{W}_2$ are smooth, it follows that $Y$ is smooth everywhere except $Q$ (as explained earlier $Y$ is singular and therefore cannot be smooth at $Q$). Hence $Y$ has exactly one singular point.

Next I will show that $Q\in Y$ is an $A_1$ point. Let $f_1 \colon X_1 \rightarrow X$ be the blow up of $P$. Let $D_1$ be the lifting of $D$ on $X_1$,  $\Delta_1$ its divisorial part and $Y_1$ the quotient of $X_1$ by $D_1$. Then there exists a commutative diagram
\begin{gather}\label{sec5-diagram-5}
\xymatrix{
 & X_1\ar[r]^{f_1} \ar[d]_{\pi_1} & X \ar[d]^{\pi} \\
 & Y_1 \ar[r]^{g_1} & Y \\
}
\end{gather}
Let $E$ be the $f$-exceptional curve and $F=\pi_1(E)$ the $g$-exceptional curve. Then there is $a \in \mathbb{Z}$ such that $K_{Y_1}=g_1^{\ast}K_Y+aF$. I will show that $Y_1$ is smooth, $a=0$ and $F^2=-2$.

Let $\bar{W}_i$ be the birational transforms of $W_i$ in $Y_1$, $i=1,2$. Then I claim that
\begin{gather}\label{sec5-eq-14}
g^{\ast}\tilde{W}_i=\bar{W}_i+\frac{1}{2}F,
\end{gather}
for $i=1,2$. Indeed. Suppose that $g^{\ast}\tilde{W}_i=\bar{W}_i+m_iF$. Then clearly $\pi_1^{\ast}\bar{W}_i=W_i^{\prime}$. Hence,
\[
W_i^{\prime}+E=f^{\ast}_1W_i=f_1^{\ast}\pi^{\ast}\tilde{W}_i=\pi_1^{\ast}g^{\ast}\tilde{W}_i=W_i^{\prime}+m_i\pi_1^{\ast}F.
\]
Therefore, $m_i\pi_1^{\ast}F=E$. This implies that $m_1=m_2$. Moreover, if $\pi_1^{\ast}F=E$, then $m_1=m_2=1$. If on the other hand, $\pi_1^{\ast}F=2E$, then $m_1=m_2=1/2$. Suppose that $\pi^{\ast}F=E$. Then $F^2=-1/2$. Also, 
\begin{gather}\label{sec5-eq-15}
g^{\ast}\tilde{W}_i=\bar{W}_i+F.
\end{gather}
Since $\tilde{W}_1\sim \tilde{W}_2$, it follows that $\bar{W}_1+F\sim \bar{W}_2+F$ and hence $\bar{W}_1 \sim \bar{W}_2$. Therefore, $\mathcal{O}_{Y_1}(\bar{W}_1)\cong \mathcal{O}_{Y_1}(\bar{W}_1)$. Considering now that $\bar{W}_1\cap \bar{W}_2=\emptyset$ it follows that both $\mathcal{O}_{Y_1}(\bar{W}_i)$, $i=1,2$, are invertible and hence $\bar{W}_1$ and $\bar{W}_2$ are Cartier. In addition. $\bar{W}_i^2=0$, $i=1,2$. But then from~\ref{sec5-eq-15} it follows that
\[
\bar{W}_i \cdot F=-F^2=1/2
\]
which is impossible since $\bar{W}_i$ are Cartier. Hence $\pi_1^{\ast}F=2E$ and $m_1=m_2=1/2$ and~\ref{sec5-eq-14} holds.

Next I will show that $Y_1$ is smooth. From~\ref{sec5-eq-14} it follows that
\[
g^{\ast}(\tilde{W}_1+\tilde{W}_2)=\bar{W}_1+\bar{W}_2+F.
\]
Since $\tilde{W}_1+\tilde{W}_2\in|2K_Y+2C|$, $\tilde{W}_1+\tilde{W}_2$ is Cartier and hence $\bar{W}_1+\bar{W}_2+F$ is Cartier as well. Since $\bar{W}_1$, $\bar{W}_2$ and $F$ are smooth, then the only possible singularities of $Y_1$ are at $\bar{W}_1\cap F$ and $\bar{W}_2\cap F$. I will show however that $\bar{W}_i$, $i=1,2$, are both Cartier and hence $Y_1$ is smooth as claimed. Since $2\tilde{W}_1 \sim \tilde{W}_1+\tilde{W}_2$, and are both Cartier, it follows that $g^{\ast}(2\tilde{W}_1)\sim g^{\ast}(\tilde{W}_1+\tilde{W}_2)$. Hence
\[
2\bar{W}_1+F \sim \bar{W}_1+\bar{W}_2+F
\]
and therefore $\bar{W}_1 \sim \bar{W}_2$. Hence $\mathcal{O}_{Y_1}(\bar{W}_1)\cong \mathcal{O}_{Y_1}(\bar{W}_2)$. But $\bar{W}_1 \cap \bar{W}_2 =\emptyset$ (because $W^{\prime}_1 \cap W^{\prime}_2 =\emptyset$, where $W_1^{\prime}, W_2^{\prime}$ are the birational transforms of $W_1, W_2$ in $X_1$). Hence $\mathcal{O}_{Y_1}(\bar{W}_i)$ are invertible, $i=1,2$, and therefore $\bar{W}_1$ and $\bar{W}_2$ are both Cartier as claimed and hence $Y_1$ is smooth. Therefore $Y$ has exactly one singular point which is of type $A_1$. Moreover the diagram~\ref{sec5-diagram-5} is its resolution.

Next I claim that 
\begin{gather}\label{sec5-eq-17}
K_X\cdot \Delta +\Delta^2 +c_2(X)=1.
\end{gather}
Indeed. With notation as in~\ref{sec5-diagram-5}, since $Y_1$ is smooth, $D_1$ has no isolated fixed points. Hence if $\Delta_1$ is its divisorial part, then from Proposition~\ref{size-of-sing}
\[
K_{X_1}\cdot \Delta_1 +\Delta_1^2+c_2(X_1)=0.
\]
But from~\ref{sec5-eq-16} we get that $\Delta_1=f_1^{\ast}\Delta +E$. Now from this, the fact that $c_2(X_1)=c_2(X)+1$, the previous equation and some straightforward calculations we get~\ref{sec5-eq-17}.

If $\mathrm{p}_g(X)=2$ then $\chi(\mathcal{O}_X)\in\{2,3\}$. 

Suppose that $\chi(\mathcal{O}_X)=2$. Then since $K_X^2=1$, from Noethers formula we get that $c_2(X)=23$. Hence $c_2(Y_1)=\chi_{et}(Y)+1=c_2(X)+1=24$. Then since $Y_1$ is rational, from Noethers formula for $Y_1$ it follows that 
\[
K_{Y_1}^2=12\chi(\mathcal{O}_{Y_1})-24=1-24=-12
\]
Hence since $f_1$ is crepant, $K_Y^2=K_{Y_1}^2=-12$. From the adjunction formula $K_Y=\pi^{\ast}K_Y+\Delta$ we get that
\[
-24=2K_Y^2=(K_X-\Delta)^2=1+\Delta^2-2K_X\cdot \Delta,
\]
and hence
\begin{gather}\label{sec5-eq-20}
\Delta^2-2K_X \cdot \Delta =-25.
\end{gather}
However, since $c_2(X)=23$,~\ref{sec5-eq-17} gives also that
\begin{gather}\label{sec5-eq-21}
K_X\cdot \Delta +\Delta^2=-22.
\end{gather}
Now from~\ref{sec5-eq-20} and~\ref{sec5-eq-21} it follows that $K_X \cdot \Delta =1$ and $\Delta^2=-23$. But now for the exactly the same reasons as in~\ref{sec5-eq-222}, this is impossible. Hence the case $\chi(\mathcal{O}_X)=2$ is impossible.

Suppose that $\chi(\mathcal{O}_X)=3$. Arguing similarly as before we get that
\begin{gather}\label{sec5-eq-22}
K_X\cdot \Delta =5\\
\Delta^2=-39. \nonumber
\end{gather}
I will show that these relations are impossible by examining all possible cases for the structure of $\Delta$ as a cycle. Suppose that 
\[
\Delta=\sum_{i=1}^kn_i \Delta_i,
\]
where $\Delta_i$ are distinct prime divisors. Since $K_X$ is ample and $K_X \cdot \Delta =5$, it follows that $k\leq 5$ and there are the following possibilities.
\begin{enumerate}
\item $\Delta$ is reduced. Then $\Delta $ has at most 5 irreducible components.
\item $\Delta$ is not reduced. Then there are the following possibilities
\begin{enumerate}
\item $\Delta =2\Delta_1+\Delta_2+\Delta_3+\Delta_4$, and $K_X\cdot \Delta_1=1$, for all $i$.
\item $\Delta =2\Delta_1+\Delta_2+\Delta_3$, and $K_X\cdot \Delta_1=K_X\cdot \Delta_2=1$, $K_X\cdot \Delta_3=2$.
\item $\Delta =2\Delta_1+\Delta_2$, and $K_X\cdot \Delta _1 =1$, $K_X\cdot \Delta_2=3$, or $K_X\cdot \Delta _1 =2$, $K_X\cdot \Delta_2=1$.
\item $\Delta=2\Delta_1+2\Delta_2+\Delta_3$, and $K_X \cdot \Delta_i=1$, for all $i$.
\item $\Delta=4\Delta_1+\Delta_2$, and $K_X\cdot \Delta_1=K_X\cdot \Delta_2=1$.
\item $\Delta = 5\Delta_1$, $K_X \cdot \Delta_1=1$.
\item $\Delta=3\Delta_1 +\Delta_2+\Delta_3$, and $K_X\cdot \Delta_i=1$, for all $i$.
\item $\Delta=3\Delta_1+2\Delta_2$, and $K_X\cdot \Delta_i=1$, for all $i$. 
\end{enumerate}
\end{enumerate}
The proof that the equations~\ref{sec5-eq-22} are impossible will be by studying each one of the above cases separately. In all cases except 2.h, there are no solutions to~\ref{sec5-eq-22} under the restriction coming from the genus formula $2\mathrm{p}_a(\Delta_i)-2=K_X\cdot \Delta_i+\Delta_i^2$. Next I will work the cases 1. and 2.h. The cases 2.a to 2.g are treated in exactly the same way as 1. but 2.h needs some deeper geometric argument since there is a numerical solution to~\ref{sec5-eq-22}.

Suppose then that $\Delta=\sum_{i=1}^k\Delta_i$, $i\leq 5$ is reduced. Then
\begin{gather*}
K_X\cdot \Delta +\Delta^2=\sum_{i=1}^k(K_X\cdot \Delta_i+\Delta_i^2)+2\sum_{1\leq i<j\leq k}\Delta_i \cdot \Delta_j \geq \\
\sum_{i=1}^k(K_X\cdot \Delta_i+\Delta_i^2)=\sum_{i=1}^k(2\mathrm{p}_a(\Delta_i)-2) \geq -2k \geq -10,
\end{gather*}
which is impossible since from~\ref{sec5-eq-22}, $K_X\cdot \Delta +\Delta^2=-34$.

Suppose now that $\Delta=3\Delta_1+2\Delta_2$, $K_X\cdot \Delta_1 =K_X\cdot \Delta_2=1$. Then from the genus formula it follows that $K_X\cdot \Delta_i +\Delta_i^2 \geq -2$ and hence $\Delta_i^2\geq -3$. Then it is easy to see that the only solutions to~\ref{sec5-eq-22} are $\Delta^2_i=-3$, $i=1,2$, and $\Delta_1\cdot \Delta_2 =0$. In particular $\mathrm{p}_a(\Delta_i)=0$ and hence $\Delta_i=\mathbb{P}^1$, $i=1,2$. 

Fix notation as in~\ref{sec5-diagram-2} and let $\tilde{\Delta_i}$, $i=1,2$, be the images of $\Delta_i$ in $Y$, with reduced structure. From the previous discussion, $X^{\prime}$ is the blow up of the unique isolated fixed point $P$ of $D$ and $Y^{\prime}$ the minimal resolution of the singularity $Q=\pi(P)\in Y$. Moreover, $Q\in Y$ is an $A_1$ singular point and in particular $K_Y$ is Cartier. Then
\[
K_X=\pi^{\ast}K_Y+3\Delta_1+2\Delta_2.
\]
From this it follows that 
\begin{gather}\label{sec5-eq-23}
\pi^{\ast}K_Y\cdot \Delta_1=10\\
\pi^{\ast}K_Y\cdot \Delta_2 =7\nonumber
\end{gather}
From the projection formula we get that
\begin{gather*}
K_Y \cdot \pi_{\ast}\Delta_1=10\\
K_Y\cdot \pi_{\ast}\Delta_2 =7.
\end{gather*}
Moreover, $\pi_{\ast}\Delta_i$ is equal to either $\tilde{\Delta}_i$ or $2\tilde{\Delta}_i$, depending on whether $\Delta_i$ is an integral curve for $D$ or not. From the second equation and since $K_Y$ is Cartier, it follows that $\pi_{\ast}\Delta_2=\tilde{\Delta}_2$ and hence $\Delta_2$ is not an integral curve for $D$. Hence $\pi^{\ast}\tilde{\Delta}_2=2\Delta_2$ and therefore, $\tilde{\Delta}_2^2=-6$ and $K_Y \cdot \tilde{\Delta}_2 = 7$. I will now show that $Q\in \tilde{\Delta}_2$. If $Q\not\in\tilde{\Delta}_2$, then $\tilde{\Delta}_2$ is in the smooth part of $Y$. Hence from the genus formula again it follows that 
\[
2\mathrm{p}_a(\tilde{\Delta}_2)-2=K_Y\cdot \tilde{\Delta}_2+\tilde{\Delta}_2^2=7-6=1,
\]
which is impossible. Hence $Q \in \tilde{\Delta}_2$. Let $\Delta^{\prime}_2$ be the birational transform of $\tilde{\Delta}_2$ in $Y^{\prime}$. Then
\[
g^{\ast}\tilde{\Delta}_2=\Delta_2^{\prime} +mF
\]
for some $m \in \mathbb{Z}$. Since $F^2=-2$ and $\tilde{\Delta}_2$ is smooth, it follows that $m=1/2$. But then
\[
(\Delta^{\prime}_2)^2=\tilde{\Delta}_2^2-\frac{1}{2}=-6-\frac{1}{2}\not\in \mathbb{Z},
\]
which is impossible since $Y^{\prime}$ is smooth. Therefore the case 2.h is impossible too.

\end{proof}

%% file: sec9.tex
\section{Examples.}\label{examples}
As mentioned in the introduction, there are several examples by now of canonically polarized surfaces with or without non-trivial global vector fields. The most common method to obtain them is as quotients of a rational surface by a rational vector field. By this method one always obtains uniruled surfaces, but nonuniruled canonically polarized surfaces with vector also exist~\cite{SB96}. However, In characteristics $p\not= 2$, it is not known if non uniruled examples exist.

\begin{problem}
Are there non uniruled smooth canonically polarised surfaces with non trivial global vector fields (equivalently with non reduced automorphism scheme) in characteristic $p>2$?
\end{problem}

Smooth hypersurfaces in $\mathbb{P}^3_k$ of degree $\geq 5$ have no vector fields~\cite{MO67}. The proof given in~\cite{MO67} shows that $H^0(T_X)=0$ by using standard exact sequences of $\mathbb{P}^3_k$. However, like Case 4. of the proof of Theorem~\ref{mult-type} it also follows from the Kodaira-Nakano vanishing which holds in this case since any smooth hypersurface lifts to $W_2(k)$.

A Godaux surface $X$ with $\pi_1^{et}(X)=\{1\}$ has no vector fields. This follows from Theorem~\ref{main-theorem}. In particular Godeaux surfaces $\pi_1^{et}(X)=\mathbb{Z}/5\mathbb{Z}$ do not have vector fields. It is known~\cite{La81} that a general Godeaux surface with $\pi_1^{et}(X)=\mathbb{Z}/5\mathbb{Z}$ in any characteristic $p\not= 5$ is the quotient of a smooth quintic in $\mathbb{P}^3$ by a free action of $\mathbb{Z}/5\mathbb{Z}$. Then the fact that $X$ has no global vector fields follows also from the fact that smooth hypersurfaces have no global vector fields. However, this proof only works for general $X$ while Theorem~\ref{main-theorem} gives the result for
 any $X$.

Smooth examples of canonically polarized surfaces were given W. Lang~\cite{La83} and N.I. Shepherd-Barron~\cite{SB96}. In particular for any Del Pezzo surface $Y$  defined over an algebraically closed field of characteristic 2 and $H \in |-K_Y|$, N.I. Shepherd-Barron~\cite[Theorem 5.2]{SB96} constructed a surface $X$ which admits maps
\[
X \stackrel{\alpha}{\rightarrow} Z \stackrel{\beta}{\rightarrow} Y,
\]
such that $\alpha$, $\beta$ are purely inseparable of degree 2,  $K_X^2=s$, $Z$ is a K3 surface with $12+s$ nodes and $\alpha^{\ast}\beta^{\ast}(\frac{1}{2}H)\sim K_X$. If $Y=\mathbb{P}^1 \times \mathbb{P}^1$, then $K_Y=-2G$ and hence $K_X=\alpha^{\ast}\beta^{\ast}G$. Since $K_Z=0$, the adjunction for purely inseparable morphisms~\cite{Ek87} shows that $\alpha $ is a foliation over $Z$ defined by the subsheaf 
$L=\mathcal{O}_X(\alpha^{\ast}\beta^{\ast}G)$ of $T_X$. Since $L=\mathcal{O}_X(\alpha^{\ast}\beta^{\ast}G)$ has sections, $X$ has nontrivial global vector fields. Moreover, it is unirational and $K_X^2=8$. In fact I do not know of any examples of smooth canonically polarized surfaces with vector fields and $K^2<8$. 

Finally, Liedtke~\cite{Li08} has constructed a series of examples of uniruled surfaces of general type in characteristic 2 with arbitrary high $K^2$. These examples are quotients of $\mathbb{P}^1 \times \mathbb{P}^1$ by rational vector fields. However the resulting surfaces are only of general type and not canonically polarized. However, taking their canonical models one obtains examples of canonically polarized surfaces with canonical singularities and $K^2$ arbitrary high. 

This is what happens in my knowledge for smooth canonically polarized surfaces. However, singular surfaces should be studied too. Especially because they are important in the moduli problem of canonically polarized surfaces and in particular its compactification. 

If $X$ is allowed to be singular and there are no restrictions on the singularities, then $K_X^2$ can take any value. Next I will present two examples of singular surfaces with vector fields and low $K^2$. The first example shows that unlike the smooth case (Theorem~\ref{lifts-to-zero}), the property \textit{Lifts to characteristic zero} does not imply smoothness of the automorphism scheme in the singular case. The second example shows that even in the presence of the mildest possible singularities (like canonical), the property \textit{Smooth automorphism scheme} is not deformation invariant and cannot be used to construct proper Deligne-Mumford moduli stacks in positive characteristic. 

\begin{example}\label{ex1}
In this example  I will construct a singular canonically polarized surface $X$ with $K_X^2=1$ defined over an algebraically closed field of characteristic 2. Moreover this surface lifts to characteristic zero and nevertheless, unlike smooth surfaces that lift to characteristic zero, it has vector fields.

Let $k$ be an algebraically closed field of characteristic 2 and $Y$ be the weighted projective space $\mathbb{P}_k(2,1,1)$. It is singular with one singularity locally isomorphic to $xy+z^2=0$. Let 
\[
\pi \colon X=\mathrm{Spec} \left( \mathcal{O}_Y \oplus \mathcal{O}_Y(-5) \right) \rightarrow Y
\]
be the 2-cyclic cover defined by $\mathcal{O}_Y(-5)$ and a general section $s$ of $\mathcal{O}_Y(5)^{[2]}=\mathcal{O}_Y(10)$. Then $X$ is normal, $K_X$ is ample and $K_X^2=1$. Moreover, $X$ is liftable to characteristic zero and has global vector fields of multiplicative type.

From the theory of weighted projective spaces it follows that
\[
H^0(L^{[2]})=H^0(\mathcal{O}_Y(10))=k[x_0,x_1,x_2]_{(10)},
\]
the space of homogeneous polynomials of degree 10 with weights, $w(x_0)=2$, $w(x_i)=1$, $i=1,2,3$. Let $s=x_0^5+f(x_1,x_2) \in k[x_0,x_1,x_2]_{(10)}$. Then I claim that for general choice of $f(x_1,x_2)$, $X$ is normal. Indeed, locally over a smooth point of $Y$, 
\[
\mathcal{O}_X=\frac{\mathcal{O}_Y[t]}{(t^2-s)}.
\]
It is now straightforward to check that for general choice of $f(x_1,x_2)$, $X$ is smooth. Hence $X$ is normal. 

Next I will show that $K_X$ is ample and $K_X^2=1$. Indeed, by the construction of $X$ as a 2-cyclic cover over $Y$,
\[
\omega_X=\pi^{\ast}(\omega_Y\otimes \mathcal{O}_Y(5))^{\ast\ast}=\pi^{\ast}\mathcal{O}_Y(1)^{\ast\ast}.
\]
From this it follows that $\omega_X^{[2]}=\pi^{\ast}\mathcal{O}_Y(2)$ and hence $K_X$ is ample and has index 2. Moreover,
\[
K_X^2=\frac{1}{4}c^2_1(\omega_X^{[2]})=\frac{1}{4}2c_1^2(\omega_Y^{[2]}\otimes \mathcal{O}_Y(10))=\frac{1}{2}c_1^2(\mathcal{O}_Y(2))=\frac{1}{2}4c_1^2(\mathcal{O}_Y(1))=\frac{1}{2}\cdot 4 \cdot \frac{1}{2}=1.
\]
Finally since $Y$, $\mathcal{O}_Y(5)$ and $s$ lift to characteristic zero, the construction of $X$ as a 2-cyclic cover lifts also to characteristic zero. However, by its construction as a 2-cyclic cover, $X$ has nontrivial global vector fields of multiplicative type.
\end{example}

\begin{example}\label{ex2}

In this example I will construct a singular surface $X$ defined over an algebraically closed field of characteristic 2 which has canonical singularities of type $A_n$ such that $K_X^2=5$, $X$ has nonzero global vector fields and moreover there is a flat morphism $f \colon \mathcal{X} \rightarrow C$, where $C$ is a curve of finite type over $k$ such that $X=f^{-1}(s)$ for some $s\in C$ and whose general fiber has no vector fields. Therefore the property \textit{Smooth automorphism scheme} is not deformation invariant (even in the presence of the mildest possible singularities)
and cannot be used to construct proper moduli stacks in positive characteristic. 

Let $k$ be an algebraically closed field of characteristic 2 and $X \subset \mathbb{P}^3_k$ be the quintic surface given by
\[
xy(x^3+y^3+z^3)+zw^4=0.
\]
It is not difficult to check that its singularities are locally isomorphic to $xy+zf(x,y,z)=0$ and therefore they are canonical of type $A_n$. Moreover, the equation of $X$ 
is invariant under the homogeneous derivation $D=z\frac{\partial}{\partial w}$ of $k[x,y,z,w]$. which therefore induces a nonzero global vector field on $X$. $X$ is smoothable by $X_t$ given by 
\[
(1-t)\left( xy(x^3+y^3+z^3)+zw^4\right) +t\left( x^5+y^5+z^5+w^5\right)=0,
\]
$t \in k$. For $t \not=0$, $X_t$ is a smooth quintic surface in $\mathbb{P}^3_k$ and hence it has no global vector fields and therefore $\mathrm{Aut}(X_t)$ is smooth for $t\not=0$.

\end{example}